\documentclass[a4paper,10pt]{amsart} 
\usepackage[utf8]{inputenc}
\usepackage[english]{babel}
\usepackage{mathrsfs}
\usepackage{amsthm}
\usepackage{amsfonts}
\usepackage{amsmath}
\usepackage{amssymb}
\usepackage{faktor}
\usepackage{mathtools}

\DeclarePairedDelimiter\floor{\lfloor}{\rfloor}
\usepackage{tikz}
\usepackage{esint}
\usepackage{centernot}
\usepackage{verbatim}
\usepackage[a4paper,top=3.5cm,bottom=3.5cm,left=2.5cm,right=2.5cm]{geometry}

\usepackage{cite}
\usepackage{color}
\definecolor{citegreen}{rgb}{0,0.8,0}
\definecolor{refred}{rgb}{0.8,0,0}
\usepackage[colorlinks, citecolor=citegreen, linkcolor=refred]{hyperref} 

\newtheorem{theorem}{Theorem}
\newtheorem{lemma}[theorem]{Lemma}
\newtheorem{corollary}[theorem]{Corollary}

\newtheorem{proposition}[theorem]{Proposition}
\theoremstyle{definition}
\newtheorem{remark}[theorem]{Remark}

\newtheorem{definition}[theorem]{Definition}

\newcommand{\R}{\mathbb R}\newcommand{\M}{\mathcal M}\newcommand{\N}{\mathcal N}
\numberwithin{equation}{section} \numberwithin{theorem}{section}
\allowdisplaybreaks

\DeclareMathOperator{\Hess}{Hess}
\DeclareMathOperator{\Div}{div}

\newcommand{\e}{\varepsilon}
\newcommand{\set}[1]{\{ {#1} \}}
\newcommand{\norm}[1]{\| {#1} \|}
\newcommand{\scal}[2]{\langle {#1} , {#2} \rangle}

\title{Higher dimensional Sacks-Uhlenbeck-type functionals and applications}
\author{Gianmichele Di Matteo, Tobias Lamm}
\address{Gianmichele Di Matteo, Scuola Superiore Meridionale, Largo San Marcellino 10, 80138 Napoli, Italy}
\address{Tobias Lamm, Karlsruhe Institute of Technology (KIT), Englerstrasse 2, 76131 Karlsruhe, Germany}

\begin{document}

\begin{abstract}
In this work, we generalize Sacks-Uhlenbeck's existence result for harmonic spheres, constructing for $n \ge 2$, regular, non-trivial, $n$-harmonic $n$-spheres into suitable target manifolds. We obtain an infinite family of new null-homotopic such maps. The proof follows a similar perturbative argument, which in high dimensions leads to a degenerate and double-phase-type Euler-Lagrange system, making the uniform regularity needed to formalize the bubbling harder to achieve. Then, we develop a refined neck-analysis leading to an energy identity along the approximation, assuming a suitable Struwe-type entropy bound along a sequence of critical points. Finally, we combine these results to solve quite general min-max problems for the $n$-energy modulo bubbling.
\end{abstract}

\maketitle

\begin{center}
\textbf{2024 Mathematics Subject Classification:} {35-XX, 49-XX, 53-XX}.\\
\textbf{Keywords}: $n$-harmonic maps, min-max, degenerate elliptic PDEs, double phase functionals. 
\end{center}

\section{Introduction}
In the seminal paper \cite{sac}, Sacks and Uhlenbeck proved the existence of branched minimal $2$-dimensional spheres in a vast class of closed manifolds, found as critical points of the Dirichlet energy of maps from $S^2$ into the target manifold $\N$. In order to deal with the conformal invariance of the Dirichlet energy, they developed an approximation scheme to recover compactness and produce critical points, and then analysed carefully their limiting behaviour. Supposedly, they were the first authors identifying the bubbling phenomenon, which revealed typical of most problems in geometric analysis, and their ideas led to a plethora of developments and analogous results in different contexts. The main purpose of this paper is to extend their analysis to higher dimensional domains, that is we aim to find non-trivial $n$-harmonic $n$-spheres in suitable target manifold.

Given a real number $p \ge 2$, two Riemannian manifolds $(\M^n,g)$ and $(\N^k,h)$ and a map $u:\M \rightarrow \N$, we define its $p$-energy $D_p$ as 
\begin{equation*}
   D_p(u):= D_p(u;\M):=\tfrac{1}{p} \int_M |d u|^p d\mu_g, 
\end{equation*}
where $d\mu_g$ is the volume element of $g$, and $|d u|$ is the norm of $d u$ seen as a section of $T^*\M \otimes u^* T^*\N$. Critical points of the energy $D_p$ are distinguished in different classes: a critical point with respect to outer variations is called a \textit{weakly $p$-harmonic map}; if a map is critical also with respect to inner variations, it is called a \textit{stationary $p$-harmonic map}; finally, a map $u$ is called a \textit{locally minimizing $p$-harmonic map} if for any compact set $K\subset \subset \M$, $D_p(u)\le D_p(v)$ for all maps $v:\M \rightarrow \N$ coinciding with $u$ outside $K$. A \emph{closed Riemannian homogeneous space} $(\N,h)$ consists of a compact quotient of Lie groups $G/H$, where $G$ is connected and $H$ is a closed subgroup of it, endowed with a left invariant metric $h$.

The existence result in \cite{sac} was expressed in terms of the non-contractibility of the universal cover of the target $\N$; we could in principle formulate our result completely analogously, relying on the concept of $n$-covering space, which however would render the result more complex, so we opted for the following statement.
\begin{theorem}[Existence of n-harmonic spheres]\label{th.existence}
For a fixed dimension $n\ge 2$, suppose $(\N,h)$ is a closed Riemannian manifold such that for some $k\in \mathbb{N}$ we have $\pi_{n+k}(\N)\neq 0$. Then the following existence result holds true:
\begin{enumerate}
    \item[(i)] For any homogeneous left-invariant metric $h$ on $\N$, there exists a non-trivial, $C^{1,\alpha}$-regular, $n$-harmonic map $u:(S^n,g_{round}) \rightarrow (\N,h)$;
    \item[(ii)] If $n=3$, for any arbitrary Riemannian metric $h$ on $\N$ there exists a non-trivial, $C^{1,\alpha}$-regular, $3$-harmonic map $u:(S^3,g_{round}) \rightarrow (\N,h)$.
\end{enumerate}
\end{theorem}
The exact same argument of part $(ii)$ restricted to dimension $n=2$, allows to recover Theorem $5.7$ in \cite{sac}. Due to the conformal invariance of the $D_n$-energy, the same existence result remains true for all conformally equivalent metrics. Notice that the statement is naturally "graded", leading to the existence of $i$-harmonic $S^i$'s into the same targets for any smaller $2\le i<n$, see Corollary \ref{cor.graded}. Also, a topological assumption is needed, as any harmonic map of a closed manifold into the Euclidean space $\R^\ell$ is trivial.

In case the group $\pi_n(\N)$ is not trivial, we re-obtain the existence results for a generating set of homotopy classes admitting $D_n$-minimizers already shown in \cite{duz1,wei0}, see Theorem \ref{th.generating}.
The most interesting case covered by Theorem \ref{th.existence} is when $\pi_n(\N)=0$, in which $D_n$-minimizers must be trivial, and our theory picks $D_n$-min-maximizers or bubbles, which provide new examples not present in the literature so far.
\begin{theorem}[Infinite null-homotopic $n$-harmonic spheres]\label{th.infinite}
There exist infinitely many explicit dimensions $n$ and Riemannian targets $(\N,h)$, such that there exists a non-trivial, $C^{1,\alpha}$-regular, null-homotopic, $n$-harmonic map $u:(S^n,g_{round})\rightarrow (\N,h)$.
\end{theorem}
In dimension $n=3$, after expliciting some examples of targets $\N$ verifying $\pi_3(\N)=0$ but $\pi_{3+k}(\N)\neq 0$ for some $k$, we show that this class of manifolds is stable under several topological operations, see Section \ref{sec.applications}.
For higher dimension of the domain $n \ge 4$, we need to endow the targets with a homogeneous Riemannian metric; even though this class may not be stable under the same topological operations considered in the $n=3$-dimensional case, we can still provide infinitely many examples appealing to some stable homotopy theory. Moreover, notice that most of the homogeneous spaces admit no totally geodesic submanifolds (see \cite{tsu}, for example Theorem $7.2$), so the inherited geometry on the images of maps as in Theorem \ref{th.infinite} may not be so rigid. Their minimality is our of reach, as expected in higher dimensions.
\subsection{Sacks-Uhlenbeck's Approach}
Let us recall briefly the original Sacks-Uhlenbeck approximation before describing our approach. In \cite{sac}, the authors developed a general approach to produce harmonic maps defined on $2$-dimensional surfaces. Since the Dirichlet energy is critical in dimension $2$, they adopt a perturbative argument, which consists in considering a family of approximating functionals, for which one could produce critical points via Morse-Palais-Smale's theory, and later analyze their convergence along the approximation procedure. A way to do this, would be by integrating higher powers of the gradient of the map; this however, would cause a troublesome degeneracy in the Euler-Lagrange system of the functionals, as explicitly remarked by the authors. In order to avoid this issue, they opt for a slightly different approximant family of functionals, all of which lead to uniformly non-degenerate elliptic Euler-Lagrange systems: given a map $u : (\M^2,g) \rightarrow (\N,h) \subset \R^N$, where the target manifold has been isometrically embedded thanks to Nash's theorem \cite{nas}, for $\alpha>1$ they define the following functionals on $W^{1,2 \alpha}(\M;\N)$
\begin{equation}\label{eq.ESU}
E^{SU}_\alpha(u):= \tfrac{1}{2} \int_M (1+|\nabla u|^2)^\alpha.
\end{equation}
Notice that in the limit $\alpha \searrow 1$, this functional differs from the Dirichlet energy only by the constant $\tfrac{1}{2}Vol_g(\M)$. For $\alpha>1$, the Euler-Lagrange system of this functional is
\begin{equation}\label{eq.ELSU}
\footnotesize \Div( (1+|\nabla u|^2)^{\alpha-1} \nabla u)=-(1+|\nabla u|^2)^{\alpha-1} A_u(\nabla u, \nabla u) \iff \Delta u +2(\alpha-1) \frac{\Hess u(\nabla u, \nabla u)}{1+|\nabla u|^2}=-A_u(\nabla u, \nabla u).
\end{equation}
They are allowed to rewrite the equation as done above since the critical points of $E^{SU}_\alpha$ are a-priori smooth (see Proposition $2.3$ in \cite{sac}). In order to prove uniform (in $\alpha$) a-priori bounds on the solutions $u_\alpha$ to these equations, for $\alpha$ close enough to $1$ and under uniform smallness of the Dirichlet energy, the authors absorb the second summand involving the Hessian "into" the Laplace operator, thanks to Calderon-Zygmund's inequality, we refer the reader to Proposition $3.1$ in \cite{sac} for further details. This is the first major building block of their analysis, ensuring the smooth convergence of the critical maps to a harmonic map away from at most finitely many points, where the energy concentrates.

We explicitly notice here that the proof of the Proposition $3.1$ in \cite{sac} works without assuming any structure on the right hand side of \eqref{eq.ELSU} other than a quadratic growth in the gradient of $u$, and therefore their regularity result reveals possible non-uniqueness for solutions to similar systems: as a remarkable example, the well-known system introduced by Frehse in \cite{fre}
\begin{equation}
\begin{cases}
\Delta u^1+\frac{2}{1+|u|^2}(u_1+u_2)|\nabla u|^2=0,\\
\Delta u^2-\frac{2}{1+|u|^2}(u_1-u_2)|\nabla u|^2=0,
\end{cases}
\end{equation}
admits for all $r<e^{-2}$ at least two solutions, $u_1:=(\sin(\log(\log(|x|^{-1}))),\cos(\log(\log(|x|^{-1}))))$ which is singular and $u_2:=(\sin(\log(\log(r^{-1}))),\cos(\log(\log(r^{-1}))))$ which is smooth, under the same Dirichlet boundary condition (posed on $\partial B_r(0)$). Thus, the regularity of solutions to systems with critical growth seems to be linked to their uniqueness.
\subsection{Our approximating functionals}
We are now ready to introduce the functionals at the center of our Sacks-Uhlenbeck type approximation. For any natural number $n \ge 2$, any $\delta \in [0,1]$ we consider exponents $p \in [n,+\infty)$ and two Riemannian manifolds $(\M^n,g)$ and $(\N,h)$, where as before $(\N,h)$ is isometrically embedded into some Euclidean space $\R^N$. For any function $u \in W^{1,p}(\M;\N)$ and any domain $\Omega \subset \M$ we set
\begin{equation}\label{eq.Epdelta}
E_{p,\delta}(u;\Omega):= \tfrac{1}{p} \int_\Omega (1+(\delta+|\nabla u|^2)^{\frac{n}{2}})^{\frac{p}{n}} -(1+\delta^{\frac{n}{2}})^{\frac{p}{n}}  \ \  d \mu_g.
\end{equation}
Notice that we recover the conformal invariant case $E_{n,0}\equiv D_n$. Moreover, for $n=2$, $\delta=0$ and $\alpha:=p/2$, we recover Sacks-Uhlenbeck's functionals from \cite{sac}.
The Euler-Lagrange system of the energy $E_{p,\delta}$ is
\begin{equation}\label{eq.ELEpdelta}
-\Div[(1+(\delta+|\nabla u|^2)^{\frac{n}{2}})^{\frac{p-n}{n}}(\delta+|\nabla u|^2)^{\frac{n-2}{2}} \nabla u]=(1+(\delta+|\nabla u|^2)^{\frac{n}{2}})^{\frac{p-n}{n}}(\delta+|\nabla u|^2)^{\frac{n-2}{2}} A_u(\nabla u, \nabla u).
\end{equation}
A first difference with respect to \cite{sac}, is that the approximant family depends on two parameters. In order to explain the role of $\delta$, notice that for $\delta=0$, even by summing $+1$ in the integrands, we have not altered the degenerating nature of the problem. Some of our results work even in this degenerate elliptic case ($n>2$ and $\delta=0$), whereas others will need the regularization by $\delta>0$. The parameter $p$ plays the same role of $\alpha$ in \cite{sac}, i.e. it justifies the use of Morse-Palais-Smale's theory. The importance of adding $1$ in the integrand relies in the gain of monotonicity of the family $E_{p,\delta}$ with respect to the parameter $p$, which ultimately allows us to use the celebrated Struwe's monotonicity trick \cite{stru3} to solve min-max problems for $D_n$ modulo bubbling, see Theorem \ref{th.minmax}. Let us remark explicitly that the systems are not uniformly elliptic for any value of the parameters since they are of double-phase type: the elliptic ratio $\mathcal{R}$ of the system verifies on any open set $\Omega \subset \M$
\begin{equation*}
\mathcal{R}(\xi;\Omega):= \frac{\sup_\Omega \lambda_{max} \in \text{eigen} \set{ \partial_\xi[ (1+(\delta+|\xi|^2)^{\frac{n}{2}})^{\frac{p-n}{n}}(\delta+|\xi|^2)^{\frac{n-2}{2}} \xi ] }}{\inf_\Omega \lambda_{min}\in \text{eigen} \set{ \partial_\xi[ (1+(\delta+|\xi|^2)^{\frac{n}{2}})^{\frac{p-n}{n}}(\delta+|\xi|^2)^{\frac{n-2}{2}} \xi ] }} \sim \underset{\delta>0}{|\xi|^{p-2}} \text{ or } \underset{\delta=0}{|\xi|^{p-n}} \underset{|\xi|\rightarrow \infty}{\rightarrow} +\infty.
\end{equation*}

For all $p>n$ and $\delta \in [0,1]$, the solutions are clearly $C^{0,\alpha}$ for $\alpha=1-\tfrac{n}{p}$ thanks to Sobolev's embedding, however this regularity clearly degenerates as $p \searrow n$, which motivates the need to develop a regularity theory uniform in the parameters $p$ and $\delta$ in order to extract a local regular convergence of critical points.

The regularity properties of solutions to \eqref{eq.ELEpdelta} resemble partially the ones of $p$-harmonic maps, and the latter have been focus of long and thorough research, which we shortly summarise. In all of the following results, we implicitely assume the smallness of the (possibly rescaled)-energy $D_p$. Firstly, in the case $p=2$, the regularity theory for harmonic maps is rather complete, we refer the reader to the survey \cite{hel2}.
For general $p$, the situation is more complicated and a full regularity theory is still not available. The case of maps minimizing $D_p$ was treated in parallel by Hardt-Lin \cite{har1}, Fuchs \cite{fuc0} and Luckhaus \cite{luc1}, where the respective authors prove their $C^{1,\alpha}$-regularity, see also \cite{duz3} for a more flexible method; this is the optimal regularity to be expected in degenerate elliptic problems (see for example \cite{boj}), although their exponent $\alpha$ may not be the optimal one. Shortly after, Evans \cite{eva1}, Mou-Yang \cite{mou0}, Strzelecky \cite{strz0} and Takeuchi \cite{tak0}, proved the $C^{0,\alpha}$-regularity for sphere-valued stationary $p$-harmonic maps, all of them relying on a compensated compactness argument. A refinement of the same method, led Toro and Wang in \cite{tor} to extend this regularity to maps valued in some homogeneous Riemannian manifold $(\N,h)$. Finally, for what regards arbitrary targets $(\N,h)$, the conjectured regularity under small rescaled energy is still to be proven. Striking results in this direction, although partial, were obtained by Miskiewicz-Petraszczuk-Strzelecki in \cite{mis}, Riviere-Strzelecky in \cite{riv} and Martino-Schikorra in \cite{mar}, the latter two being the sharpest and most recent results.

For what regards double-phase problems, sometimes also called with $(p,q)$-growth, the literature is quite vast and we cite only a few results. To fix ideas, we restrict the attention to equations whose highest order term is in divergence form $-\Div(a(x,\xi))$ satisfying
\begin{equation*}
\partial_\xi a(x,\xi) \xi \cdot \xi \ge a |\xi|^{p}-b,\ \ |\partial_\xi a(x,\xi)| \le c|\xi|^q+d,\ \ \mathcal{R}(\xi;\Omega):= \frac{\sup_\Omega \lambda_{max} \in \text{eigen} \set{ \partial_\xi a(x,\xi) }}{\inf_\Omega \lambda_{min}\in \text{eigen} \set{ \partial_\xi a(x,\xi) }} \sim |\xi|^{q-p}.
\end{equation*}
Problems of this kind originated in the works \cite{ural1} by Ural'tseva-Urdaletova and \cite{zhi1,zhi2,zhi3} by Zhikov, concerning Homogenization and Lavrentiev's phenomenon. A comprehensive variational-analytical treatement of them was carried out by Marcellini in \cite{marc1,marc2,marc3,marc4} (see also Simon's work \cite{sim2}). As highlighted in these papers, a condition ensuring the regularity of solutions, also somewhat necessary (see \cite{esp1}), is the gap $q-p$ to be small enough, with a crucial role played also by the $x$-regularity of the coefficient $a$. This heuristic is also confirmed by several subsequent papers, see for example \cite{def0,def1,def2,esp1,esp2,esp3,lad1}. We refer the reader to the survey \cite{min0}. The first work considering minimizers to double-phase functionals of manifold-targeted maps is the work of De Filippis-Mingione \cite{def0} (apart from \cite{sac} which however avoids degeneracy of the associated Euler-Lagrange system).
\begin{theorem}[Uniform regularity]\label{th.SUregularity}
There exist constants $P_0 \in (n,+\infty)$, $\e_0>0$, $C_0>0$, $\tau_0 \in (0,1)$ and $\alpha_0 \in (0,1)$ all depending only on the data $n,N,(\M,g)$ and $(\N,h)$ satisfying the following statement. Consider a family $(u_{p,\delta})_{p\in (n,P_0),\delta \in (0,1]}$ of solutions to \eqref{eq.ELEpdelta} defined on a ball $B_{R_0}(x_0) \subset \M$, satisfying the uniform energy bound $R_0^{p-n}E_{p,\delta}(u_{p,\delta};B_{R_0}(x_0))\le \e_0^p$ for all $p\in (n,P_0)$ and $\delta \in (0,1]$. Suppose either of the following conditions is met:
\begin{enumerate}
    \item[(a)] any element of the family $(u_{p,\delta})$ is locally minimizing the corresponding energy $E_{p,\delta}$;
    \item[(b)] the target $(\N,h)$ is a homogeneous Riemannian manifold endowed with a left-invariant metric;
    \item[(c)] the dimension of the domain manifold $\M$ is $n=3$.
\end{enumerate}
Then for all $p\in (n,P_0)$ and $\delta \in (0,1]$ we have $u_{p,\delta} \in C^{1,\alpha_0}(B_{\tau_0 R_0}(x_0);\N)$ with uniform bounds on the semi-norms
\begin{equation}
    [\nabla u_{p,\delta}]_{C^{0,\alpha_0}(B_{\tau_0 R_0}(x_0);\N)} \le C_0 E_{p,\delta}(u_{p,\delta};B_{R_0}(x_0))^{\frac{1}{p}}.
\end{equation}
\end{theorem}
Since $p>n$, we can use the monotonicity of the rescaled energies to see that an uniform bound holds even on the full ball $B_{R_0}(x_0)$, see Section \ref{sec.regularity}.
For minimizing maps considered in point $(a)$, we apply Luckhaus' regularity result in \cite{luc1}, after carefully checking its uniformity as $p\searrow n$ and $\delta \searrow 0$, to deduce an initial uniform $C^{0,\alpha}$-regularity for all $\alpha \in (0,1)$; this relies on the  fact that our "blow up" functionals (see \cite{luc1}) are the $D_p$-energies, for which uniform regularity applies (notice that only the bigger phase survives in the blow up!). We can ultimately gain the claimed $C^{1,\alpha_0}$-regularity adapting methods from Duzaar-Mingione in \cite{duz3}. This methods works even for $\delta=0$ all along.
\begin{remark}
It is worth mentioning that the same method would not quickly ensure the regularity result in \cite{def0}, as in their case, the "blow-up" functional is still of double-phase type (for example $\mathcal{F}(u)=\int |\nabla u|^p+a(x) |\nabla u|^q$) and its minimizers enjoy the same regularity theory as those of the original functional.

Moreover, even though our functionals are not autonomous (as the metric on $T\M \otimes u^* T\N$ depends on the points $x$ and $u(x)$), this metric never annihilates, so from the regularity theory viewpoint, our functionals behave like autonomous functionals, like those treated in \cite{esp2}, compare with \cite{fus}.
\end{remark}

In order to treat the case $(b)$, we adopt Toro-Wang's argument from \cite{tor}, hence relying on a compensated compactness argument in combination with Hardy-BMO duality. Also this argument would work for $\delta=0$ as well.

In the last case $(c)$, the proof follows similar lines to Sacks-Uhlenbeck's ones. More precisely, we first show the local uniqueness of the solutions to \eqref{eq.ELEpdelta}, and from this deduce their local minimality. We explicitly remark that this works on a scale $\tau R_0$ such that $\tau \rightarrow 0$ as $p \searrow n$, so the newly gained uniform regularity (deduced from Theorem \ref{th.SUregularity} point $(a)$) cannot be extended to a uniform scale. However, since $\delta>0$, the system \eqref{eq.ELEpdelta} is uniformly non-degenerate with (by what we just said) locally bounded coefficients, hence uniformly elliptic, and standard Schauder's theory yields a-priori smoothness. This smoothness degenerates as $\delta \searrow 0$, but it allows us to recast the system similarly to what done in \cite{sac}
\begin{equation}\label{eq.ELEpdeltaSU}
-(\delta+|\nabla u|^2)^{\frac{2-n}{2}} \Div[(\delta+|\nabla u|^2)^{\frac{n-2}{2}} \nabla u]=(p-n) \tfrac{(\delta+|\nabla u|^2)^{\frac{n-2}{2}}\nabla^2 u (\nabla u,\nabla u)}{1+(\delta+|\nabla u|^2)^{\frac{n}{2}} } +A_u(\nabla u, \nabla u).
\end{equation}
The left-hand-side can be seen as a perturbation of the renormalised $n$-laplacian operator $\Delta_n^N$ (see \cite{kuh}). In dimension $n=3$, this operator enjoys some uniform Calderon-Zygmund-type estimates, and we can carry out the proof similarly to \cite{sac}. The dimensional restriction comes from the so-called Cordes' condition, an algebraic condition expressing closeness of the operator to the Laplace operator, see Definition \ref{def.Cordes}, which ensures the CZ-type estimates claimed above. As far as we know, it is the first time Cordes' condition is used in the context of $p$-harmonic maps; its application to their conjectured regularity, under small rescaled energy and for $p$ in the Cordes' interval (see \eqref{eq.pCordes}), is the focus of a forthcoming paper.

Moreover, we can prove uniform $C^{1,\alpha_0}$-estimates assuming only the smallness of the $3$-energy $D_3$ (Theorem \ref{th.regularityD3}). It is worth considering the latter theorem from a PDE perspective: this is a uniform (higher order) H\"older regularity result for a (essentially) degenerate non-uniformly elliptic variational system (in Uhlenbeck's quasi-diagonal form), with $\mathbb{L}^1$-right-hand-side (the only \textbf{uniform} bound assumed), assuming smallness of a norm \textbf{below the natural energy level}. We adopted the nomenclature from the book \cite{min1}. This may sound surprising. However, one should not forget that we implicitely assumed the solution to be in the energy space (in particular, non-energetic solutions as those constructed in \cite{col1} are excluded), as well as that the two phases are almost coinciding $p \sim 3$ (coherently to the works cited above). In a parallel work \cite{dim}, we prove such a Sacks-Uhlenbeck's type regularity for $p$-harmonic maps into homogeneous targets; notice that a proof for any target $\N$ would prove also the regularity of general-targeted $n$-harmonic maps, which is out of reach.

Remarkably, our attempts to exploit compensation phenomena, together with Iwaniec-Sbordone's stability estimates and Coifman-Rochberg-Weiss's estimate, as done in the recent works \cite{mar,mis} cited above, failed.
\subsection{Limiting analysis}
Consider now a family $(u_{p,\delta})_{p\in (n,P_0), \delta \in (0,1)}$ of critical maps for $E_{p,\delta}$ whose $E_{p,\delta}$ energies remain uniformly bounded as $p\searrow n$ and $\delta \searrow 0$. Exactly as in \cite{sac}, due to the asymptotic conformal invariance of the problem, we cannot prove in full generality the regular convergence of $(u_{p,\delta})$ (not even after extracting a sub-sequence), but we expect some "bubbles" to form at finitely many points in the limit. In our case, these bubbles can be obtained through an iterated blow-up procedure at points where the $E_{p,\delta}$-energies of the $u_{p,\delta}$'s concentrate, see Section \ref{sec.quantization} for details. Analytically, these bubbles can be identified to $C^{1,\alpha}$-regular, $n$-harmonic maps $\omega:S^n \rightarrow \N$, after using Duzaar-Fuchs' singularity removability Theorem from \cite{duz0}, see Section \ref{sec.quantization}. Theorem \ref{th.SUregularity} at hand, we can formalise the formation of these bubbles along the sequence $(u_{p,\delta})$, obtaining a limit base map $u_n$, given by a weak $W^{1,n}$-limit of a sub-sequence $u_{p_k,\delta_k}$, along with finitely many bubbles $\omega^{i,j}$; the convergence of $u_{p_k,\delta_k}$ to $u_n$ is in $C^{1,\alpha}$ on $\M$ away from finitely many points, where one can perform dilations at suitable infinitesimal scales, converging in $C^{1,\alpha}$ to one of the bubbles (one has to suitably iterate this procedure).
 
As it is customary in this kind of problems, one can ask for convergence in energy of $u_{p,\delta}$ to the "bubble tree" $\set{u_n} \cup \set{\omega^{i,j}}$ (sometimes also called quantization of energy or energy identity). Before stating our Theorem, let us summarise some related results in the current literature.
 
We will start with the Sacks-Uhlenbeck approximation, mostly important for our analysis. Chen-Tian \cite{che} proved the quantization of the energy for sequences of minimizers of the energy $E^{SU}_\alpha$ in a given homotopy class. Later Moore \cite{moo} proved the convergence in energy for $E^{SU}_\alpha$-min-max sequences under the additional assumption that the target manifold has finite fundamental group. Jost \cite{jos} removed this additional topological assumption. In \cite{liw0}, Li-Wang proved the energy identity for $E^{SU}_\alpha$-minimizers, each in their own homotopy class. In full generality, sequences of critical maps to $E^{SU}_\alpha$ have been costructed in \cite{liw1} by the same authors, for which the convergence in energy fails. Without assuming any variational characterization of the critical maps in consideration, the energy identity was proven in \cite{liz} by Li-Zhu for spherical targeted critical maps. Finally, in \cite{lam1}, the second-named author proved the quantization of the energy for sequences of critical maps assuming a Struwe-type entropy condition, and this is the result we aim to generalise. The main reason for this choice is that we want to consider general target manifold $\N$ (at least for $n=3$), while at the same time keeping a strong enough generality so to apply our theory to min-max problems for the $n$-energy (see Theorem \ref{th.minmax} below). For energy identities in the context of harmonic maps, see for example \cite{lin1,nab2,par}.

Finally, we mention the energy identites results for sequences of: $n$-harmonic maps in \cite{mou1}, for minimizers of the $n$-energy in homotopy classes \cite{wei0,nak}, for minimizers of conformally invariant energies of $n$-growth in homotopy classes in \cite{duz1}, and for suitable $D_n$-Palais-Smale sequences in \cite{wan1}.
\begin{theorem}[Energy identity]\label{th.quantization}
Let $(\M^n,g)$ and $(\N,h)$ be smooth, closed, Riemannian manifolds. For sequences $p_k\searrow n$, $\delta_k \searrow 0$, let $(u_k):=(u_{p_k,\delta_k}):\M\rightarrow \N$ be a family of critical maps to $E_{p_k,\delta_k}$ with uniformly bounded energy $E_{p_k,\delta_k}(u_k,\M)\le \Lambda_0<+\infty$. Suppose that either of the conditions $(a)$, $(b)$ and $(c)$ in Theorem \ref{th.SUregularity} holds true. Assume further the following Struwe-type entropy condition
\begin{equation}\label{eq.entropy}
\lim_{k\rightarrow n} (p_k-n)\int_M (1+(\delta_k+|\nabla u_k|^2)^{\frac{n}{2}})^{\frac{p_k}{n}}\log(1+(\delta_k+|\nabla u_k|^2)^{\frac{n}{2}}) =0.
\end{equation}
Then, up to subsequences, for any $\alpha_1 \in (0,\alpha_0)$, where $\alpha_0$ is as in Theorem \ref{th.SUregularity}, there exist finitely many points $x^1,...,x^K \in \M$ (possibly none), finitely many non-trivial $n$-harmonic maps $\omega^{i,j} \in C^{1,\alpha_1}(S^n;\N)$ and a $n$-harmonic map $u_n \in C^{1,\alpha_1}(\M^n;\N)$ such that $u_k \rightarrow u_n$ weakly in $W^{1,n}(\M;\N)$ and in $C^{1,\alpha_1}_{loc}(\M \setminus \set{x^1,...,x^K};\N)$. At any such point $x^i$, one can obtain finitely many of the $\omega^{i,j}$'s, for $j=1,...,j_i$, as limits of suitable blow-up procedures, centered at family of points $x^{i,j}_k \in \M$, with $x^{i,j}_k \rightarrow x^i$, and relative to parameters $r^{i,j}_k>0$, with $r^{i,j}_k \rightarrow 0$. These parameters verify
\begin{equation}\label{eq.parameters}
\begin{aligned}
&\max \set{\tfrac{r^{i,j}_k}{r^{i,j'}_k},\tfrac{r^{i,j'}_k}{r^{i,j}_k}, \tfrac{d_g(x^{i,j}_k,x^{i,j'}_k)}{r^{i,j}_k+r^{i,j'}_k} } \rightarrow +\infty, \quad \forall 1\le i \le K, \ \ 1\le j,j' \le j_i, \ \ j \neq j', \\
&\limsup_{k \rightarrow +\infty} (r^{i,j}_k)^{p_k-n}=1, \quad \forall 1\le i \le K, \ \ 1\le j \le j_i.
\end{aligned}
\end{equation}
Finally, the following energy identities hold
\begin{equation}\label{eq.quantization1}
\lim_{k \rightarrow +\infty} E_{p_k,\delta_k}(u_{p_k,\delta_k},\M)=D_n(u_n,\M)+\sum_{i=1}^K \sum_{j=1}^{j_i} D_n(\omega^{i,j},S^n),
\end{equation}
and
\begin{equation}\label{eq.quantization2}
\lim_{k \rightarrow +\infty} D_n(u_{p_k,\delta_k},\M)=D_n(u_n,\M)+\sum_{i=1}^K \sum_{j=1}^{j_i} D_n(\omega^{i,j},S^n).
\end{equation}
\end{theorem}
By the results of Duzaar and Kuwert \cite{duz1} (Theorem 2), the above Theorem
implies that we also have a decomposition in terms of homotopy classes.

The proof of this result resembles the $2$-dimensional analogue treated in \cite{lam1} by the second-named author. More precisely, we formalise quantitatively the formation of bubbles with the use of Brezis-Coron maximal concentration function \cite{bre0}. This allows us to pick suitable concentration radii $r_k^{i,j}$ and points $x_k^{i,j}$ such that the balls $B_{r_k^{i,j}}(x_k^{i,j})$ definitively (in $k$) contain a fixed positive amount of energy. Since the original sequence of critical point has uniformly bounded energy, only finitely many bubbles can be produced this way. Given the "regular" convergence of the sequence well inside the concentration balls and far away from them, we only need to prove that the energy along the connecting annuli is asymptotically vanishing to conclude the energy identity. In order to do that, we first show the sharp estimate on the concentration radii expressed in \eqref{eq.parameters}, using our entropy estimate. Then we prove some Hopf-differential type estimate, which heuristically says that energy density is asymptotically aligned along the polar direction of the annulus. Adapting Sacks-Uhlenbeck's dyadic harmonic replacement from \cite{sac}, we also obtain a definite bound on the energy of this polar component. Finally, we combine all these estimates to deduce \eqref{eq.quantization1}. However, on the annuli in consideration, the $D_n$-energy is suitably controlled by the $E_{p,\delta}$-energy, so we deduce also \eqref{eq.quantization2}.

In full generality, the necessity of the entropy condition towards the energy identity has been proven in \cite{liw1}, and one may wonder how restrictive it is to assume it. Conjecturally, one could make use of compensation phenomena to remove the entropy assumption for homogeneous targets as in \cite{liz}. Instead, we will show in Lemma \ref{lemma.minmax} that this assumption is always satisfied along a sequence of min-maximizers of the $E_{p,\delta}$-energy, adapting the celebrated Struwe's monotonicity trick from \cite{stru0} (as remarked before, the energy $E_{p,\delta}$ is monotone in $p$). As a Corollary, we can show that min-max problems for the $D_n$-energy are satisfied up to bubbling:
\begin{theorem}\label{th.minmax}
Suppose $A$ is a compact parameter manifold, with $\partial A=\emptyset$, and let $h_0:\M\times A \rightarrow \mathcal{\N}$ be a continuous map and $\alpha_1<\alpha_0$ from Theorem \ref{th.SUregularity}. Denote by $H$ the class of all maps $C^0-$homotopic to $h_0$ and define the min-max value
\begin{equation}
\beta:=\inf_{h \in H} \max_{t \in A} E_n(h(\cdot,t)).
\end{equation}
Then there exist a $n$-harmonic map $u_n\in C^{1,\alpha_1}(\M;\N)$ and finitely many $n$-harmonic maps $\omega^{i,j} \in C^{1,\alpha_1}(S^n;\mathcal{\N})$ such that
\begin{equation}
\beta=D_n(u_n,\M)+\sum_{i=1}^l \sum_{j=1}^{j_i} D_n(\omega^{i,j},S^n).
\end{equation}
Moreover, the set of maps $\set{u_n,\omega^{i,j}}$ arises as "bubble tree" limit of a sequence of maps $u_{p_k,\delta_k} \in C^{\infty}(\M;\mathcal{\N})$ critical for $E_{p_k,\delta_k}$ as described in Theorem \ref{th.quantization}. These critical maps also verify related min-max problems
\begin{align}
E_{p_k,\delta_k}(u_{p_k,\delta_k})=\beta_{p_k,\delta_k}:=\inf_{h \in H} \max_{t \in A} E_{p_k,\delta_k}(h(\cdot,t));\ \ \beta_{p_k,\delta_k} \longrightarrow \beta.
\end{align}
\end{theorem}

\subsection{Structure of the paper}
In Section \ref{sec.preliminary} we discuss the main preliminary material needed in the following sections. In Section \ref{sec.regularity} we prove the main regularity Theorem \ref{th.SUregularity}. In Section \ref{sec.quantization}, we obtain the energy identity Theorem \ref{th.quantization}. In Section \ref{sec.minmax}, we adapt Struwe's monotonicity trick to deduce Theorem \ref{th.minmax}. Finally, in Section \ref{sec.applications}, we establish Theorem \ref{th.existence} and other existence results.
\subsection{Acknowledgements}
The authors wish to thank Cristoph B\"ohm for interesting discussions leading to a simplified statement of the main existence Theorem \ref{th.existence}. The first named author wishes to thank Antonio Tarsia for introducing him to Cordes' condition in a Master degree course, as well as Carlo Collari and Luca Pol for discussions about topology. The first named author has been partially supported by the PRIN Project 2022AKNSE4 \emph{Variational and Analytical aspects of Geometric PDEs} during the writing of this paper.

\section{Preliminary}\label{sec.preliminary}
\subsection{Notation}
Open balls in the domain manifold $\M$ or $S^n$ will be denoted by $B_r(x)$, where the radius is $r$ and the center is $x$; annuli in the domain manifolds will be denoted by $A(x,r,R):= B_R(x) \setminus \bar{B}_r(x)$; balls in the target will have the upper-script $\N$, will be centered at point $y$, with radius $\rho$, hence $B^{\N}_\rho(y)$. Given $\Omega \subset \M$, the average of a function $u$ on it is denoted by $[u]_\Omega$. We adopt Einstein's convention of summing over repeated indices.

Given an open set $\Omega \subset \R^n$, we introduce: Morrey's space $\mathbb{L}^{p,\lambda}(\Omega)$ to be the subset of $\mathbb{L}^p$ of elements with finite norm
\begin{equation*}
\norm{u}_{\mathbb{L}^{p,\lambda}(\Omega)}^p:= \sup_{B_r(x)\subset \R^n} r^{-\lambda}\int_{B_r(x) \cap \Omega} |u|^p.
\end{equation*}
Campanato's space $\mathcal{L}^{p,\lambda}(\Omega)$ to be the subset of  $\mathbb{L}^p$ of elements with finite semi-norm
\begin{equation*}
[u]_{\mathcal{L}^{p,\lambda}(\Omega)}^p:= \sup_{B_r(x)\subset \R^n} r^{-\lambda}\int_{B_r(x) \cap \Omega} |u-[u]_{B_r(x) \cap \Omega}|^p.
\end{equation*}
In case $\lambda=n$ and $p=1$, we call $\mathcal{L}^{1,n}(\Omega)=:BMO(\Omega)$ the space of bounded mean oscillation.

Hardy's space $\mathcal{H}^1(\R^n)$ is the subset of $\mathbb{L}^1(\R^n)$ of elements such that for some $\Phi \in C^{\infty}_c(\R^n)$ with $\int \Phi=1$, the following semi-norm is finite
\begin{equation*}
[u]_{\mathcal{H}^1}:= \sup_{t \in (0,+\infty)} \int_{\R^n} |u * t^{-n} \Phi \big( \tfrac{\cdot}{t}\big)|(x) \ dx.
\end{equation*}
The values of the not indexed constants appearing in the proofs of our results may vary from line to line.
\subsection{Setting}
We start by recalling the definition of the energies given in \eqref{eq.Epdelta}:
\begin{equation*}
E_{p,\delta}(u;\Omega)= \tfrac{1}{p} \int_\Omega (1+(\delta+|\nabla u|^2)^{\frac{n}{2}})^{\frac{p}{n}}-(1+\delta^{\frac{n}{2}})^{\frac{p}{n}} \ d \mu_g,
\end{equation*}
as well as their Euler-Lagrange systems from \eqref{eq.ELEpdelta}
\begin{equation*}
-\Div[(1+(\delta+|\nabla u|^2)^{\frac{n}{2}})^{\frac{p-n}{n}}(\delta+|\nabla u|^2)^{\frac{n-2}{2}} \nabla u]=(1+(\delta+|\nabla u|^2)^{\frac{n}{2}})^{\frac{p-n}{n}}(\delta+|\nabla u|^2)^{\frac{n-2}{2}} A_u(\nabla u, \nabla u).
\end{equation*}
First of all, we notice the convexity of the integrand defining \eqref{eq.Epdelta} and its monotonicity in the parameter $p$. Moreover, since $p>n$, we will use the monotonicity of the rescaled energies $r^{p-n}D_p(u;B_r(x))$ and $r^{p-n}E_{p,\delta}(u;B_r(x))$ with respect to both $r$ and the domain of integration.

With the aim in mind of applying our uniform regularity theorems even to rescalings of our solutions, we briefly introduce a more general system depending on a bounded parameter $s\in (0,1)$:
\begin{equation*}
-\Div[(s+(\delta+|\nabla u|^2)^{\frac{n}{2}})^{\frac{p-n}{n}}(\delta+|\nabla u|^2)^{\frac{n-2}{2}} \nabla u]=(s+(\delta+|\nabla u|^2)^{\frac{n}{2}})^{\frac{p-n}{n}}(\delta+|\nabla u|^2)^{\frac{n-2}{2}} A_u(\nabla u, \nabla u),
\end{equation*}
Clearly, this system is just the Euler-Lagrange system associated to the energy
\begin{equation*}
E_{p,\delta}^{(s)}(u;\Omega):= \tfrac{1}{p} \int_\Omega (s+(\delta+|\nabla u|^2)^{\frac{n}{2}})^{\frac{p}{n}}-(s+\delta^{\frac{n}{2}})^{\frac{p}{n}}\ d \mu_g.
\end{equation*}
In the regularity theory that we will develop in the next Section \ref{sec.regularity}, we will always rescale the parameter $s=1$, rather than the radius of the balls involved to $1$; we choose to do this in order to highlight the different scales at which our double-phase functional $E_{p,\delta}$ acts, however there would be no major difference in doing the opposite.
\subsection{Cordes' Condition}
As anticipated in the Introduction, solutions of \eqref{eq.ELEpdelta} enjoy some a-priori smoothness, see Corollary \ref{cor.uniqueness2}. This allows us to partially develop the divergence and divide by the non-trivial factor $(1+(\delta+|\nabla u|^2)^{\frac{n}{2}})^{\frac{p-n}{n}}(\delta+|\nabla u|^2)^{\frac{n-2}{2}} $ as in \cite{sac}, to rewrite \eqref{eq.ELEpdelta} as in \eqref{eq.ELEpdeltaSU} which we recall
\begin{equation*}
-(\delta+|\nabla u|^2)^{\frac{2-n}{2}} \Div[(\delta+|\nabla u|^2)^{\frac{n-2}{2}} \nabla u]=(p-n) \tfrac{(\delta+|\nabla u|^2)^{\frac{n-2}{2}}\nabla^2 u (\nabla u,\nabla u)}{1+(\delta+|\nabla u|^2)^{\frac{n}{2}} } +A_u(\nabla u, \nabla u).
\end{equation*}
The non-linear operator on the left-hand-side $\Delta^{N (\delta)}_n u :=(\delta+|\nabla u|^2)^{\frac{2-n}{2}} \Div[(\delta+|\nabla u|^2)^{\frac{n-2}{2}} \nabla u]$, is a perturbation of the classical normalised $n$-Laplacian. We can rewrite it as $(\Delta^{N (\delta)}_n u)^\alpha=: \mathcal{L}_u [u]$, where $\mathcal{L}_u$ is a second order linear operator depending non-linearly on $u$, acting on maps $v$ as follows
\begin{equation}\label{eq.definitionLu}
 (\mathcal{L}_u [v])^\alpha:=(\delta_{i j} \delta^{\alpha \beta}+(n-2) \tfrac{\nabla_i u^\alpha \nabla_j u^\beta}{\delta+|\nabla u|^2}) \nabla^2_{i j} v^\beta =:L^{\alpha \beta}_{i j}(u) \nabla^2_{i j} v^\beta.
\end{equation}
For all $u$, the operator $\mathcal{L}_u$ is second-order, uniformly elliptic, in non-divergence form, with measurable coefficients. Calderon-Zygmund's regularity theory does not hold in general under such hypotheses, but one can still get integrability of the Hessian under Cordes' condition which we recall from \cite{mau}.
\begin{definition}[Cordes' condition]\label{def.Cordes}
An operator $\mathcal{A}$ of the form $(\mathcal{A} v)^\alpha=A^{\alpha \beta}_{i j} \nabla^2_{i j} v^\beta$ acting on maps $v \in W^{2,1}(\R^n;\R^N)$ satisfies the Cordes' condition if there exists a constant $\e\in (0,1]$ such that
\begin{equation}\label{eq.cordes}
\sum_{i,j=1}^{n} \sum_{\alpha,\beta=1}^{N} (A^{\alpha \beta}_{i j})^2 \le \tfrac{1}{nN-1+\e} \Big( \sum_{i=1}^{n} \sum_{\alpha=1}^{N} A^{\alpha \alpha}_{i i}\Big)^2.
\end{equation}
\end{definition}
When considering the normalised $p$-laplacian acting on scalar valued functions, Cordes' condition was firstly identified in \cite{man}; we present here the case of vector valued maps.
\begin{lemma}
If the parameters $n$, $N$ and $p$ satisfy
\begin{equation}\label{eq.pCordes}
\begin{cases}
1 \le p < +\infty \text{ if } nN \le 2,\\
1 \le p < 3+\tfrac{2}{nN-2} \text{ if } n N>2,
\end{cases}
\end{equation}
then for some $\e=\e(p,nN) \in (0,1]$, the operator $\mathcal{A}_{p,u}$ defined as
\begin{equation}
(\mathcal{A}_{p,u} [v])^\alpha:=(\delta_{i j} \delta^{\alpha \beta}+(p-2) \tfrac{\nabla_i u^\alpha \nabla_j u^\beta}{\delta+|\nabla u|^2}) \nabla^2_{i j} v^\beta =:A^{\alpha \beta}_{i j}(u) \nabla^2_{i j} v^\beta
\end{equation}
satisfies for all $u \in W^{1,\infty}$ and $\delta>0$ the Cordes condition with constant $\e$ independent of $u$ and $\delta$.
\end{lemma}
\begin{proof}
Firstly, we compute an upper bound for the left hand side of \eqref{eq.cordes} with the explicit coefficients of $\mathcal{L}_{p,u}$
\begin{align*}
&\sum_{i,j=1}^{n} \sum_{\alpha,\beta=1}^{N} (A^{\alpha \beta}_{i j})^2=\sum_{i,j=1}^{n} \sum_{\alpha,\beta=1}^{N} (\delta_{i j} \delta^{\alpha \beta}+(p-2) \tfrac{\nabla_i u^\alpha \nabla_j u^\beta}{\delta+|\nabla u|^2})^2 =\sum_{i=1}^{n} \sum_{\alpha=1}^{N} (1+(p-2) \tfrac{|\nabla_i u^\alpha|^2}{\delta+|\nabla u|^2})^2\\
&+\sum_{i=1}^{n} \sum_{\alpha \neq \beta=1}^{N} (p-2)^2 (\tfrac{\nabla_i u^\alpha \nabla_i u^\beta}{\delta+|\nabla u|^2})^2+\sum_{i\neq j=1}^{n} \sum_{\alpha=1}^{N} (p-2)^2 (\tfrac{\nabla_i u^\alpha \nabla_j u^\alpha}{\delta+|\nabla u|^2})^2+\sum_{i\neq j=1}^{n} \sum_{\alpha \neq \beta=1}^{N} (p-2)^2 (\tfrac{\nabla_i u^\alpha \nabla_j u^\beta}{\delta+|\nabla u|^2})^2\\
&=nN+2(p-2)\sum_{i=1}^{n} \sum_{\alpha=1}^{N} \tfrac{|\nabla_i u^\alpha|^2}{\delta+|\nabla u|^2}+(p-2)^2 \sum_{i, j=1}^{n} \sum_{\alpha, \beta=1}^{N} (\tfrac{\nabla_i u^\alpha \nabla_j u^\beta}{\delta+|\nabla u|^2})^2= nN+2(p-2)\tfrac{|\nabla u|^2}{\delta+|\nabla u|^2}\\
&+(p-2)^2 \tfrac{|\nabla u|^4}{(\delta+|\nabla u|^2)^2}.
\end{align*}
Here we have used that $|\nabla u \otimes \nabla u|= |\nabla u|^2$. For what regards the right hand side, we take some $\e\in (0,1]$ and compute
\begin{align*}
\tfrac{1}{nN-1+\e} \Big( \sum_{i=1}^{n} \sum_{\alpha=1}^{N} A^{\alpha \alpha}_{i i}\Big)^2=\tfrac{1}{nN-1+\e} \Big( \sum_{i=1}^{n} \sum_{\alpha=1}^{N} 1+(p-2) \tfrac{|\nabla_i u^\alpha|^2}{\delta+|\nabla u|^2}\Big)^2= \tfrac{1}{nN-1+\e} \Big( n N+(p-2)\tfrac{|\nabla u|^2}{\delta+|\nabla u|^2} \Big)^2.
\end{align*}
Therefore, verifying the Cordes condition is equivalent to
\begin{equation}\label{eq.cordes1}
\begin{aligned}
&(nN-1+\e) \Big( nN+2(p-2)\tfrac{|\nabla u|^2(x)}{\delta+|\nabla u|^2(x)}+(p-2)^2 \tfrac{|\nabla u|^4(x)}{(\delta+|\nabla u|^2(x))^2}\Big) \le \Big( n N+(p-2)\tfrac{|\nabla u|^2}{\delta+|\nabla u|^2} \Big)^2 \iff\\
&-nN(1-\e)-(1-\e)2(p-2)\tfrac{|\nabla u|^2(x)}{\delta+|\nabla u|^2(x)}+(nN-2+\e)(p-2)^2 \tfrac{|\nabla u|^4(x)}{(\delta+|\nabla u|^2(x))^2} \le 0.
\end{aligned}
\end{equation}
If we set $t:=\tfrac{|\nabla u|^2(x)}{\delta+|\nabla u|^2(x)} \in [0,1]$, we see that we want to prove that the inequality
\begin{equation}\label{eq.cordes2}
-nN(1-\e)-(1-\e)2(p-2)t+(nN-2+\e)(p-2)^2 t^2 \le 0,
\end{equation}
holds uniformly in $t \in [0,1]$. We will distinguish between several cases. If $n N<2$, that is if $n=N=1$, then this inequality is verified for all $p$ and all $\e>0$. This inequality holds true also for $p=2$ and any choice of $\e$. In the complementary case, we notice that the left hand side is a polynomial of second degree in $t$, with positive leading term, strictly negative at $t=0$ for any choice of $\e \in (0,1)$. So we need to prove that its positive root is greater or equal than $1$; for $p>2$, which amounts to show
\begin{align*}
&t_1:=\frac{(1-\e)2(p-2) + \sqrt{4(1-\e)^2(p-2)^2+4(nN-2+\e)(p-2)^2 nN(1-\e)}}{2 (nN-2+\e)(p-2)^2} \ge 1\\
&\iff (1-\e) + \sqrt{(1-\e)^2+(nN-2+\e) nN(1-\e)}\ge (nN-2+\e)(p-2).
\end{align*}
If $nN=2$, then for all $p \in (2,+\infty)$, we can choose $\e=\e(p)$ small enough so that this inequality holds (for example we can choose $\e=(p-1)^{-1}$). Otherwise, we use the upper bound $p<3+\tfrac{2}{nN-2}$ to get that for some small $\eta>0$ we have $p\le 3+\tfrac{2}{nN-2}-\eta$ and solve for the more strict inequality
\begin{align*}
&(1-\e) + \sqrt{(1-\e)^2+(nN-2+\e) nN(1-\e)}> (nN-2+\e)(\tfrac{nN}{nN-2}-\eta) \underbrace{\iff}_{\e \rightarrow 0} \\
&nN=1 + \sqrt{1+(nN-2) nN}> nN-(nN-2)\eta,
\end{align*}
which is clearly satisfied. Since the inequality at the limit is strict, the same remains true for $\e>0$ small enough depending on $\eta$ so on $p$ and $nN$.

Finally, for $p\in (1,2)$ we can use $p \ge 1$ to prove once again an inequality more strict than what we need
\begin{align*}
&t_1:=\frac{(1-\e)2(p-2) + \sqrt{4(1-\e)^2(p-2)^2+4(nN-2+\e)(p-2)^2 nN(1-\e)}}{2 (nN-2+\e)(p-2)^2} \ge 1\\
&\iff -(1-\e) + \sqrt{(1-\e)^2+(nN-2+\e) nN(1-\e)}\ge (nN-2+\e)(2-p)\\
&\Leftarrow -(1-\e) + \sqrt{(1-\e)^2+(nN-2+\e) nN(1-\e)}\ge (nN-2+\e).
\end{align*}
After rearranging and taking the square to both sides we get
\begin{equation*}
(1-\e)^2+(nN-2+\e) nN(1-\e)\ge (nN-1)^2 \iff \e \le \tfrac{(nN)^2-3 nN+2}{nN-1}.
\end{equation*}
Since we are in the case $nN>2$ we can find a small $\e>0$ depending only on $nN$ satisfying this inequality. This finishes the proof.
\end{proof}
In our case we have $\mathcal{L}_u=\mathcal{A}_{n,u}$, so Cordes' condition is satisfied only if $n=2,3$, for any dimension of the target Euclidean space $\R^N$. Appealing to Theorems $1.2.1-1.2.3$ and Remark $1.6.11$ in \cite{mau} we recover partially Calderon-Zygmund's regularity theory for the operator of our interest.
\begin{corollary}\label{cor.Cordesdirichlet}
Suppose the parameters $(n,N,p)$ satisfy \eqref{eq.pCordes}. Let $B$ be a ball. Then there exists parameters $q_0<2<q_1$ and a constant $C_1$, depending on the dimensions $n$, $N$ and $p$, such that for all $q \in (q_0,q_1)$ and data $f \in \mathbb{L}^q(B;\R^N)$, the Dirichlet problem
\begin{equation}
\begin{cases}
\mathcal{A}_{p,u} v= f \quad \text{in } B;\\
v=0 \quad \text{in } \partial B,
\end{cases}
\end{equation}
admits a unique solution $v \in W_0^{2,q}(B;\R^N)$, which enjoys the estimate
\begin{equation}
\norm{\nabla^2 v}_{q,B} \le C_1 \norm{f}_{q,B}.
\end{equation}
In particular, for every function $v\in W_0^{2,q}(B;\R^K)$ we have Calderon-Zygmund's type inequality
\begin{equation}
\norm{\nabla^2 v}_{q,B} \le C_1 \norm{\mathcal{A}_{p,u} v}_{q,B}.
\end{equation}
\end{corollary}
In an effort to make this work self-contained, we recall the proof of this result in our restricted case.
\begin{proof}
We start by proving the result for $q=2$. Firstly, let us set
\begin{equation*}
\gamma(x):=\tfrac{\sum_{i=1}^{n} \sum_{\alpha=1}^{N} A^{\alpha \alpha}_{i i} }{\sum_{i,j=1}^{n} \sum_{\alpha,\beta=1}^{\N} (A^{\alpha \beta}_{i j})^2}
\end{equation*}
and rewrite the system as
\begin{equation*}
\Delta v=(\Delta - \gamma \mathcal{A}_{p,u}) v+\gamma f.
\end{equation*}
In order to see that this problem is uniquely solvable, with corresponding estimates, we define for all fixed $w \in W^{2,2}_0(B;\R^N)$ the map $U=:Tw \in W^{2,2}_0(B;\R^N)$ as the only solution of the system
\begin{equation*}
\begin{cases}
    \Delta U=(\Delta - \gamma \mathcal{A}_{p,u}) w+\gamma f \ \text{ on } B;\\
    U=0 \ \text{ on } \partial B.
\end{cases}
\end{equation*}
We are going to show that $T$ is a contraction. Recalling Miranda-Talenti's estimate (with optimal constant $1$ since the ball $B$ is bounded, convex and regular)
\begin{equation*}
\int_B |\nabla^2 v|^2 \le \int_B |\Delta v|^2, \quad \forall v \in W^{2,2}_0(B;\R^N),
\end{equation*}
we obtain
\begin{align*}
&\norm{Tw_1-Tw_2}^2_{W^{2,2}_0}=\norm{U_1-U_2}^2_{W^{2,2}_0} \le \norm{\Delta U_1-\Delta U_2}^2_{\mathbb{L}^2}=\norm{\Delta (w_1-w_2)-\gamma \mathcal{A}_{p,u}(w_1-w_2) }^2_{\mathbb{L}^2}\\
&\le \int_B \Big( \sum_{\alpha \beta} \sum_{i,j} |\delta^{\alpha \beta} \delta_{i j} -\gamma A^{\alpha \beta}_{i j}(u)|^2 \Big) \Big( \sum_{\alpha \beta} \sum_{i,j} |\nabla^2_{i j} (w_1^{\alpha \beta} -w_2^{\alpha \beta})|^2 \Big) \le (1-\e)\norm{w_1-w_2}^2_{W^{2,2}_0},
\end{align*}
where we have used Cordes condition to deduce
\begin{align*}
&\sum_{\alpha \beta} \sum_{i,j} |\delta^{\alpha \beta} \delta_{i j} -\gamma A^{\alpha \beta}_{i j}(u)|^2= \sum_{\alpha \beta} \sum_{i,j} (\delta^{\alpha \beta} \delta_{i j})^2 -2 \gamma \delta^{\alpha \beta} \delta_{i j} A^{\alpha \beta}_{i j}(u)+\gamma^2 (A^{\alpha \beta}_{i j}(u))^2=\\
&nN - 2 \sum_{\alpha} \sum_{i} \gamma A^{\alpha \alpha}_{i i}(u) + \gamma^2 \sum_{\alpha \beta} \sum_{i,j} A^{\alpha \beta}_{i j}(u) = nN -\tfrac{ \Big( \sum_{i=1}^{n} \sum_{\alpha=1}^{N} A^{\alpha \alpha}_{i i} \Big)^2 }{\sum_{i,j=1}^{n} \sum_{\alpha,\beta=1}^{\N} (A^{\alpha \beta}_{i j})^2} \le nN-(nN-1+\e)=1-\e.
\end{align*}
Hence $T$ is indeed a contraction and its unique fixed point $v$ satisfying the estimate
\begin{equation}
\norm{v}_{W^{2,2}_0} \le \norm{\Delta v}_{\mathbb{L}^2}=\norm{\gamma f+\Delta v-\gamma \mathcal{A}_{p,u}v }_{\mathbb{L}^2}\le \norm{\gamma}_{\infty} \norm{f}_{\mathbb{L}^2}+\sqrt{1-\e}\norm{v}_{W^{2,2}_0}.
\end{equation}
We can therefore reabsorb and conclude the case $q=2$. For all $q \in (1,+\infty)$ we define the optimal Calderon-Zygmund constant $C_{CZ}(q)$ as the smallest constant such that
\begin{equation}
\norm{v}_{W^{2,q}_0} \le C_{CZ}(q) \norm{\Delta v}_{\mathbb{L}^q} \quad \forall v \in W^{2,q}_0.
\end{equation}
By Miranda-Talenti's estimate we know $C_{CZ}(2)=1$. Moreover, for any $Q \in (2,+\infty)$, Riesz-Thorin's interpolation inequality gives
\begin{equation}
C_{CZ}(q) \le C_{CZ}(2)^{\alpha} C_{CZ}^{1-\alpha}(Q)=C_{CZ}(Q)^{\tfrac{Q(q-2)}{q(Q-2)}}, \quad \text{where } \tfrac{1}{q}= \tfrac{\alpha}{2}+\tfrac{1-\alpha}{Q}.
\end{equation}
Therefore, choosing $q_1>2$ close enough to $2$ we can assume $C_{CZ}(q) \sqrt{1-\e} < 1$ for all $q \in [2,q_1)$. We can argue analogously relatively to the H\"older conjugate $Q'$ to get the existence of $q_0<2$ such that $C_{CZ}(q) \sqrt{1-\e} < 1$ for all $q \in (q_0,2]$. Extending the definition of the operator $T$ to maps in $W^{2,q}_0$, we can prove it to be a contraction exactly as before:
\begin{align*}
&\norm{Tw_1-Tw_2}^q_{W^{2,q}_0}=\norm{U_1-U_2}^q_{W^{2,q}_0} \le C_{CZ}(q)^q \norm{\Delta U_1-\Delta U_2}^2_{\mathbb{L}^q}\\
&\le C_{CZ}(q)^q \norm{\Delta (w_1-w_2)-\gamma \mathcal{A}_{p,u}(w_1-w_2) }^q_{\mathbb{L}^q} \le C_{CZ}(q)^q (1-\e)^{\frac{q}{2}}\norm{w_1-w_2}^q_{W^{2,q}_0}.
\end{align*}
The estimate is proven similarly
\begin{equation}
\norm{v}_{W^{2,q}_0} \le C_{CZ}(q) \norm{\Delta v}_{\mathbb{L}^q}\le C_{CZ}(q)(\norm{\gamma}_{\infty} \norm{f}_{\mathbb{L}^q}+\sqrt{1-\e}\norm{v}_{W^{2,q}_0}),
\end{equation}
concluding the proof after reabsorbing.
\end{proof}

One could also consider interior regularity estimates obtaining the following Corollary.
\begin{corollary}\label{cor.cordesbound}
Suppose the parameters $(n,N,p)$ satisfy \eqref{eq.pCordes}. By possibly restricting the interval $(q_0,q_1)$, given by Corollary \ref{cor.Cordesdirichlet}, while still mantaining the condition $q_0<2<q_1$, for every $v \in W^{2,q}(\Omega;\R^N)$, where $q \in (q_0,q_1)$, and any ball $B_{2r}(x)\subseteq \Omega$ we have
\begin{equation}
r^{2-\frac{n}{q}}\norm{\nabla^2 v}_{q,B_r(x)} \le C_2 r^{2-\frac{n}{q}} \norm{\mathcal{A}_{p,u}[v]}_{q,B_r(x)}+ C_2 r^{1-\frac{n}{2}}\norm{\nabla v}_{2, B_{2 r}(x)}.
\end{equation}
\end{corollary}
\begin{proof}
Let $h$ be the harmonic function $\Delta h=0$ satisfying boundary conditions $h=v$ on $\partial B_{2 r}$. Then $w:=v-h$ has zero Dirichlet boundary data and is a weak solution of
\begin{equation}
\Delta w=\Delta v=(\Delta - \gamma \mathcal{A}_{p,u}) v+\gamma \mathcal{A}_{p,u}[v].
\end{equation}
Here we have defined $\gamma$ as in the previous proof.
Combining Calderon-Zygmund's inequality for the Laplacian, the triangular inequality and Cordes' condition we get
\begin{equation}
\begin{aligned}
&\norm{\nabla^2 v}_{q,B_r} -\norm{\nabla^2 h}_{q,B_r} \le \norm{\nabla^2 w}_{q,B_r} \le C_{CZ}(q) \norm{(\Delta - \gamma \mathcal{A}_{p,u}) v}_{q,B_r}+C_{CZ}(q) \norm{\gamma \mathcal{A}_{p,u}[v]}_{q,B_r} \\
&\le C_{CZ}(q) \sqrt{1-\e} \norm{\nabla^2 v}_{q,B_r}+C \norm{\mathcal{A}_{p,u}[v]}_{q,B_r}.
\end{aligned}
\end{equation}
As in the previous proof, we can assume that $C_{CZ}(q)\sqrt{1-\e}<1$. Using the derivative estimates for harmonic function, expressed in the particular form of \cite{eva2} Part I, Section 2.2 Theorem 7, applied to $\nabla h$, we obtain
\begin{equation}
\norm{\nabla^2 h}_{q,B_r} \le C r^{-1+n \frac{2-q}{2q}} \norm{\nabla h}_{2,B_{2r}},
\end{equation}
from which we can rearrange the inequality above to get
\begin{equation}
r^{2-\frac{n}{q}} \norm{\nabla^2 v}_{q,B_r} \le C r^{2-\frac{n}{q}} \norm{\mathcal{A}_{p,u}[v]}_{q,B_r}+ C r^{1-\frac{n}{2}} \norm{\nabla h}_{2,B_{2r}} \le C r^{2-\frac{n}{q}} \norm{\mathcal{A}_{p,u}[v]}_{q,B_r}+C r^{1-\frac{n}{2}} \norm{\nabla v}_{2,B_{2r}},
\end{equation}
where the last inequality follows from the minimality of $h$ with respect to the Dirichlet energy, under fixed boundary conditions.
\end{proof}
\begin{remark}
Let us explicitly remark that the constants in the previous Corollary degenerate only when Cordes' costant $\e$ approaches zero, which is never the case for our $\mathcal{L}_u=\mathcal{A}_{3,u}$. Also, if we happen to know that $\nabla v \in \mathbb{L}^3$, as in the case of our interest, then we can substitute the last term with $C_2 r^{1-\frac{n}{3}} \norm{\nabla v}_{3,B_{2r}}$ by a simple application of H\"older's inequality.
\end{remark}
\subsection{Elementary inequalities}\label{subsec.elementary}
We now list some elementary inequalities, some satisfied by the integrand defining $E^{(s)}_{p,\delta}$, with particular attention to the uniformity of the constants as $p\searrow n$ and $\delta \searrow 0$. We will implicitely assume that $p \in (n,P_0)$ for some $P_0 \in (n,n+1)$, $\delta \in (0,1)$ and $s \in [0,1]$.
\begin{equation}\label{eq.elementary0}
|x-y|^n \le 2^{n-1} (|x|^n+|y|^n), \quad |x-y|^p \le 2^{P_0-1}(|x|^p+|y|^p).
\end{equation}
Using that for every $r \in (2,p)$ and $t \ge 0$ one has $t^r \le t^2+t^p$, we arrive to
\begin{equation}\label{eq.elementary1}
\begin{aligned}
&\max \set{ |x|^n-a_0(n,P_0),|x|^p}\le (s+(\delta+|x|^2)^{\frac{n}{2}})^{\frac{p}{n}}-(s+\delta^{\frac{n}{2}})^{\frac{p}{n}} \le C_3(n,P_0)(1+|x|^p).
\end{aligned}
\end{equation}
Using the convexity of $(1+x)^\alpha$ for $\alpha\ge 1$, we get
\begin{equation}\label{eq.elementary3}
    \tfrac{1}{p} \big[ (1+(\delta+|x|^2)^{\frac{n}{2}})^{\frac{p}{n}}-(1+\delta^{\frac{n}{2}})^{\frac{p}{n}}\big] \le \tfrac{1}{n}|x|^n+ C_4(n,P_0) \max \set{|x|^p,1}(P_0-n+ \delta).
\end{equation}
Setting $V(X):=(s+(\delta+|X|^2)^{\frac{n}{2}})^{\frac{p-n}{2 n}}(\delta+|x|^2)^{\frac{n-2}{4}} X$, we have for all $X,Y \in \R^{nN}$ and for some constants $c_0=c_0(n,P_0)>0$ and $c_1=c_1(n,N,P_0)>0$ uniform as $p\searrow n$ and $\delta \rightarrow 0$
\begin{equation}\label{eq.elementary2}
\begin{aligned}
[(s+(\delta+|X|^2)^{\frac{n}{2}})^{\frac{p-n}{n}} (\delta+|X|^2)^{\frac{n-2}{2}} X &- (s+(\delta+|Y|^2)^{\frac{n}{2}})^{\frac{p-n}{n}}(\delta+|Y|^2)^{\frac{n-2}{2}} Y]\cdot (X-Y) \\
&\ge c_0 |V(X)-V(Y)|^2 \ge c_1 |X-Y|^p.
\end{aligned}
\end{equation}
The proof of this inequality is quite lenghty, so it is presented in the Appendix.

Finally, we prove a Gagliardo-Niremberg type inequality. This will be used in the proof of Theorem \ref{th.SUregularity}, point $(c)$, to bound the quadratic term in \eqref{eq.ELEpdeltaSU}. Since we will be interested in the case $2 q>n=3$, we cannot apply directly Proposition $3.2$ in Struwe \cite{stru2} but we need to adapt its proof to our simpler case. See also \cite{riv} for a related result.
\begin{lemma}[Gagliardo-Nirenberg type Lemma]
For any ball $B_{R}(x) \subseteq \R^n$, any $q \in [\tfrac{n}{2},n]$ and any function $f \in W^{2,q}(B_{R}(x))$, we have $\nabla f \in \mathbb{L}^{2q}(B_{R}(x))$ and the following estimate holds
\begin{equation}\label{eq.gagliardo}
R^{2-\frac{n}{q}}\norm{\nabla f}_{ \mathbb{L}^{2q}(B_{R}(x)) }^2 \le C_5(n) \norm{\nabla f}_{\mathbb{L}^{n}(B_{R}(x))}(\norm{\nabla f}_{\mathbb{L}^{n}(B_{R}(x))}+R^{2-\frac{n}{q}}\norm{\nabla^2 f}_{\mathbb{L}^{q}(B_{R}(x))}).
\end{equation}
\end{lemma}
\begin{proof}
The proof follows similar lines as in \cite{stru2}. Firstly, notice that by Sobolev's embedding we have $W^{2,q} \subseteq W^{1,2q} \subset BMO$ for $q$ in the assumed range. In order to prove the estimate \eqref{eq.gagliardo}, we start recalling the following Gagliardo-Niremberg type inequality due to Adams-Frazier (see \cite{ada}): for every $q>1$ and any $G \in W^{2,q} \cap BMO(\R^n)$ with compact support, we have
\begin{equation}
\norm{\nabla G}_{\mathbb{L}^{2q}(\R^n)}^2 \le C(q) [G]_{BMO(\R^n)} \norm{\nabla^2 G}_{\mathbb{L}^{q}(\R^n)}.
\end{equation}
Due to the scale invariance of the thesis, we can assume $B:=B_R(x)=B_1(0)$; moreover, without loss of generality we assume $f$ has mean value $0$ on $B$. For any such $f$ we can find an extension $F \in W^{2,q}(B_2)$ such that for some constant $C=C(n,q) \le C_5(n)$
\begin{equation*}
\norm{\nabla^2 F}_{\mathbb{L}^{q}(B_2)}  \le C \norm{\nabla^2 f}_{\mathbb{L}^{q}(B)} \quad \text{and} \quad \norm{\nabla F}_{\mathbb{L}^{n}(B_2)} \le  C \norm{\nabla f}_{\mathbb{L}^{n}(B)}.
\end{equation*}
Let $\eta \in C^\infty_c(B_2(0))$ be a cut-off function such that $\eta \in [0,1]$ and $\eta \equiv 1$ on $B$. Applying the inequality above to $G:=\eta (F-[F]_{B_2}) \in W^{2,q}$ we get
\begin{equation*}
\norm{\nabla f}_{\mathbb{L}^{2q}(B)}^2 \le \norm{\nabla G}_{\mathbb{L}^{2q}(\R^n)}^2 \le C [G]_{BMO(\R^n)} \norm{\nabla^2 G}_{\mathbb{L}^{q}(\R^n)}.
\end{equation*}
For what regards the first term we estimate it through Poincar\'e's inequality
\begin{equation*}
[G]_{BMO(\R^n)}\le C \norm{\nabla G}_{\mathbb{L}^{n}(\R^n)} \le C(\norm{F-[F]_{B_2}}_{\mathbb{L}^{n}(B_2)}+\norm{\nabla F}_{\mathbb{L}^{n}(B_2)})  \le C \norm{\nabla F}_{\mathbb{L}^{n}(B_2)}\le  C \norm{\nabla f}_{\mathbb{L}^{n}(B)}.
\end{equation*}
Combining also with H\"older's inequality we obtain for the second term
\begin{align*}
\norm{\nabla^2 G}_{\mathbb{L}^{q}(\R^n)}&\le C(\norm{F-[F]_{B_2}}_{\mathbb{L}^{q}(B_2)}+\norm{\nabla F}_{\mathbb{L}^{q}(B_2)}+\norm{\nabla^2 F}_{\mathbb{L}^{q}(B_2)}) \le C(\norm{\nabla F}_{\mathbb{L}^{n}(B_2)}+\norm{\nabla^2 F}_{\mathbb{L}^{q}(B_2)})\\
&\le C(\norm{\nabla f}_{\mathbb{L}^{n}(B)}+\norm{\nabla^2 f}_{\mathbb{L}^{q}(B)}).
\end{align*}
from which we conclude the proof.
\end{proof}
\subsection{Variational theory for the approximating functionals}
The main reason to introduce the functionals $E_{p,\delta}$ where $p>n$, is that we can find critical points to them by variational methods; we follow closely the argument in \cite{sac}.

We start by considering the Dirichlet problem associated to our energies. Morrey's theorem guarantees the lower semi-continuity of $E_{p,\delta}$ with respect to the weak convergence in $W^{1,p}$, thanks to which we can ensure the existence of at least one solution to \eqref{eq.ELEpdelta} under extensible Dirichlet boundary condition.
\begin{lemma}\label{lemma.dirichlet}
For any $p \in (n,P_0)$, $\delta \in [0,1]$ and function $u_0 \in W^{1,p}(B;\N)$, where $B \subset \M$ is a ball, there exists a function $\tilde{u}_0 \in W^{1,p}(B;\N)$ minimizing the functional $E_{p,\delta}$ amongst all maps in the set

\noindent $\mathcal{E}:= \left\{ u \in W^{1,p}(B;\N) \mid Tr(u)=Tr(u_0) \text{ on } \partial B \right\}$. In particular, $\tilde{u}_0$ solves the associated Euler-Lagrange system \eqref{eq.ELEpdelta} weakly.
\end{lemma}
\begin{proof}
We pick a minimizing sequence $(u_k) \subset \mathcal{E}$ ($\mathcal{E}\neq \emptyset$ since $u_0$ belongs to it); the minimizing property of $(u_k)$ ensures that $\norm{\nabla u_k}_p^p \le E_{p,\delta}(u_k) \le C<\infty$ for some constant $C$, and therefore also the boundedness of $(u_k)$ in $W^{1,p}(B;\R^N)$ (recall that the target $\N$ is closed and hence bounded). Extracting a subsequence (and relabeling) we can assume that $u_k \rightharpoonup \tilde{u}_0$ for some $\tilde{u}_0 \in W^{1,p}(B;\R^N)$ by weak compactness; moreover, possibly extracting a further subsequence, the convergence is also in $C^{0,\alpha}$ for some $\alpha<1-\tfrac{n}{p}$, as one can see by combining Sobolev's embedding and Arzel\`a-Ascoli's theorem. In particular, the target is preserved in the limit, i.e. $\tilde{u}_0 \in W^{1,p}(B;\N)$; the same holds for the boundary condition, in other words we have that $Tr(u_0)=Tr(u_k) \rightarrow Tr(\tilde{u}_0)=Tr(u_0)$, so ultimately $\tilde{u}_0 \in \mathcal{E}$. The lower semi-continuity of $E_{p,\delta}$ allows us to conclude the minimality of $\tilde{u}_0$, and therefore also the last statement follows.
\end{proof}
Later, we are gonna be interested in maps minimizing in the restricted class of maps with small image, whose existence is ensured by the following Proposition, adapted from Proposition $3.1$ in \cite{far}.
We set $\rho_0:=\min \set{i_N, \norm{A}_\infty^{-1} }$, where $i_N$ is the injectivity radius of $\N$ and $A$ is the second fundamental form of $\N$.
\begin{proposition}\label{prop.dirichlet}
Consider a radius $\rho \in (0,\rho_0)$, a ball $B \subset \R^n$ and a map $u_0 \in W^{1,p}(B;\N)$ satisfying $u_0(B) \subseteq B_{\rho}^\N(y_0)$ for some $y_0 \in \N$. Then there exists a map $\tilde{u}_0 \in W^{1,p}(B;\N)$ solving \eqref{eq.ELEpdelta} such that $\tilde{u}_0(B) \subseteq B_{\rho}^\N(y_0)$, $\tilde{u}_0=u_0$ along $\partial B$, and $\tilde{u}_0$ is minimizing the functional $E_{p,\delta}(\cdot;B)$ amongst all functions $v \in W^{1,p}(B;\N)$ such that $v=u_0$ along $\partial B$ and such that $v(B) \subseteq B_{\rho}^\N(y_0)$.
\end{proposition}
\begin{proof}
The proof follows closely the one of Propositon $3.1$ in \cite{far}. Fix a radius $\rho_1 \in (\rho,\rho_0)$. Consider initially the problem of minimizing $E_{p,\delta}(\cdot;B)$ amongst all the maps $v \in W^{1,p}(B;\N)$ with $v=u_0$ along $\partial B$ and $v(B) \subseteq \bar{B}_{\rho_1}^\N(y_0)$. Pick a minimizing sequence $(v_k)$. Since $E_{p,\delta}$ controls the $p$-norm of the gradient, we have that $(v_k)$ is bounded in $W^{1,p}(B;\R^N)$, and therefore $v_k \rightharpoonup v_\infty$ weakly in $W^{1,p}(B;\R^N)$ for some $v_\infty \in W^{1,p}(B;\R^N)$ (after possibly extract a sequence). As before, by Rellich's theorem, we have $v_k \rightarrow v_\infty$ in $\mathbb{L}^p$ and by combining Sobolev's embedding and Arzel\`a-Ascoli's theorem the convergence is also in $C^{0,\alpha}$ for some $\alpha<1-\tfrac{n}{p}$. Thus we deduce that $v_\infty \in W^{1,p}(B;\N)$, $v_\infty (B) \subseteq \bar{B}_{\rho_1}^\N(y_0)$, and also that $v_\infty=u_0$ along $\partial B$. Therefore by weak lower semi-continuity of the energy $E_{p,\delta}$, $v_\infty$ solves this preliminary minimization problem. We now want to show $v_\infty (B) \subseteq B_{\rho}^\N(y_0)$ and that $v_\infty$ solves \eqref{eq.ELEpdelta} in normal coordinates around $y_0$. We start by proving the inclusion, denoting $v:=v_\infty$ for brevity. In normal coordinates around $y_0$, this amounts to prove $|v|<\rho$. In order to do so, we take a function $\eta \in W_0^{1,p}(B;\R)$, $\eta \ge 0$, and consider the comparison functions $v_t:=exp_{y_0}((1-t \eta)v)$, for $t \in [0,\norm{\eta}_{\infty}^{-1}]$. Notice that $v_0=v$, $v_t \in W^{1,p}(B;\N)$, $v_t (B) \subseteq \bar{B}_{\rho_1}^\N(y_0)$ and $v_t=u_0$ along $\partial B$, so $v_t$ can be chosen as comparison in the minimization problem above, hence $E_{p,\delta}(v;B) \le E_{p,\delta}(v_t;B)$ for all $t$. Taking the right $t$-derivative at $t=0$, we know $\partial_t E_{p,\delta}(v_t;B) \mid_{t=0} \ge 0$, that is
\begin{equation}
\int_{B}(1+(\delta+|\nabla v|^2)^{\frac{n}{2}})^{\frac{p-n}{n}}(\delta+|\nabla v|^2)^{\frac{n-2}{2}} \nabla v \nabla(\eta v)-\int_{B} (1+(\delta+|\nabla v|^2)^{\frac{n}{2}})^{\frac{p-n}{n}}(\delta+|\nabla v|^2)^{\frac{n-2}{2}} A_v(\nabla v, \nabla v) \eta v \le 0.
\end{equation}
Choosing $\eta:= \max \set{|v|^2-\rho^2,0}$, we see that $\eta \in W_0^{1,p}(B_1;\R)$, with explicit expression for the gradient: $\nabla \eta=2 v \nabla v$ if $|v|\ge \rho$ and $\nabla \eta =0$ if $|v|<\rho$. Plugging into the inequality above, we deduce
\begin{equation*}
\footnotesize
\begin{aligned}
\int_{B} \tfrac{1}{2}(1+(\delta+|\nabla v|^2)^{\frac{n}{2}})^{\frac{p-n}{n}}(\delta+|\nabla v|^2)^{\frac{n-2}{2}} |\nabla \eta|^2+\int_{B} (1+(\delta+|\nabla v|^2)^{\frac{n}{2}})^{\frac{p-n}{n}}(\delta+|\nabla v|^2)^{\frac{n-2}{2}}(|\nabla v|^2- A_v(\nabla v, \nabla v) v) \eta \le 0.
\end{aligned}
\end{equation*}
Since $\N$ is closed and $|v|$ is small enough, we can assume that $|\nabla v|^2- A_v(\nabla v, \nabla v) v \ge 0$, and therefore
\begin{equation*}
\int_{B} (1+(\delta+|\nabla v|^2)^{\frac{n}{2}})^{\frac{p-n}{n}}(\delta+|\nabla v|^2)^{\frac{n-2}{2}} |\nabla \eta|^2 \le 0 \ \Rightarrow \ \nabla \eta \equiv 0 \ \Rightarrow \ |v|<\rho.
\end{equation*}
Lastly, we need to prove that $v$ is a solution to \eqref{eq.ELEpdelta}. Take an arbitrary map $\phi \in W_0^{1,p}(B;\R^N)$ and consider the comparison map $v_t:=exp_{y_0}(v+t \phi)$. By what we have just proved, we know that $v_t(B) \subseteq \bar{B}_{\rho_1}^\N(y_0)$ for $|t|<(\rho_1-\rho) \norm{\phi}_\infty^{-1}$; clearly, we also have $v_t=u_0$ along $\partial B$. Once again this implies that $v_t$ can be compared to $v$ in energy, and thus $E_{p,\delta}(v_t;B) \ge E_{p,\delta}(v;B)$ for all $t$. We conclude by taking the derivative at $t=0$, and setting $\Tilde{u}_0:=v_\infty$.
\end{proof}
 In the following, we are gonna use some classical results due to Palais, as done in \cite{sac}. Since $p>n$, the natural energy space $W^{1,p}(\M;\N)$ is a $C^2$ separable Banach manifold. Moreover, we know that $E_{p,\delta}$ is $C^2$ on this manifold, and satisfies the Palais-Smale condition in a complete Finsler metric on $W^{1,p}(\M;\N)$. In particular, $E_{p,\delta}$ attains its minimum on every connected component of $W^{1,p}(\M;\N)$, and whenever $E_{p,\delta}$ does not have critical points in the slab $W^{1,p}(\M;\N) \cap E_{p,\delta}^{-1}((a,b))$, then there exists a deformation retraction of the sub-level set $W^{1,p}(\M;\N) \cap E_{p,\delta}^{-1}((-\infty,b))$ into the smaller sub-level set $W^{1,p}(\M;\N) \cap E_{p,\delta}^{-1}((-\infty,a))$.
\vspace{0.3cm}

Now we want to develop a unified treatement for the variational analysis of all the functionals in the family $E_{p,\delta}$. As in \cite{sac}, we denote by $\N_0:=\set{u\in W^{1,p}(\M;\N) \mid \forall x \in \M, \ u(x) \equiv y_0, \text{ for some } y_0 \in \N}$ the set of constant maps. This is the common submanifold of minima to $E_{p,\delta}$, where all the functionals take the minimal value $0$. The homotopy types of the spaces $W^{1,p}(\M^n;\N)$ are equivalent for all $p>n$, and equivalent also to the homotopy types of $C^0(\M;\N)$ and $C^\infty(\M;\N)$. With this identification in mind, we can adapt Proposition $2.4$ in \cite{sac} to our case. More precisely, we fix a smooth map $u$ in some connected component of $W^{1,p}(\M^n;\N)$, and denote by $B:=\norm{\nabla u}_{\infty}$. Then we clearly see that the minimum of $E_{p,\delta}$ over that connected component in consideration satisfies for all $p\in(n,P_0)$ and $\delta \in [0,1]$
\begin{equation}\label{eq.boundmin}
    \min E_{p,\delta} \le \tfrac{1}{p}[(1+(\delta+B^2)^{\frac{n}{2}})^{\frac{p}{n}}-(1+\delta^{\frac{n}{2}})^{\frac{p}{n}} ] Vol(\M):=B_{p,\delta}.
\end{equation}
Exactly as in \cite{sac}, we can obtain a family of deformations of suitable sub-level sets on the common minima sub-manifold $\N_0$, as expressed in the following Theorem.
\begin{theorem}[Theorem $2.6$ in \cite{sac}]\label{th.lowestenergyretraction}
For any $p\in(n,P_0)$ and $\delta \in [0,1]$ there exist a constant $\eta:=\eta(p,\delta,n,\N)$ and a deformation retraction $\sigma:W^{1,p}(\M^n;\N) \cap E_{p,\delta}^{-1}([0,\eta)) \rightarrow \N_0$.
\end{theorem}
The idea of the proof is rather simple and completely analogous to the one in \cite{sac}, so we just sketch it here. Since the energy $E_{p,\delta}$ controls the $\mathbb{L}^p$-norm of the gradient by \eqref{eq.elementary1}, and the target $\N$ is closed, for all maps $u$ in the sub-level set, for $\eta$ small enough, their images must lie in small enough geodesic balls of the target manifold, which we can assume to be all included in the images through the exponential maps of $\N$ based at suitable centers $y \in \N$, of balls of some uniform radius in $T_y \N$ (of course the centers $y$ depend on the map $u$ in consideration). Without loss of generality, there exist retractions of these balls to the origin in $T_y \N$. Composing these with the respective exponential maps, and then internally with the maps $u$ in consideration, we can retract each map $u$ to its respective constant map to $y$. This means that we have a retraction of the full sub-level set to $\N_0$. See Theorem $2.6$ in \cite{sac} for details.

Fix base points $x_0 \in \M$ and $y_0 \in \N$ and denote by $\Omega(\M;\N)$ the space of continuous maps sending $x_0$ to $y_0$. There exists a natural fibration $\pi_{x_0}:C^0(\M;\N)\rightarrow \N$ given by the evaluation of any map at $x_0$, whose fiber is exactly $\Omega(\M;\N)$. This fibration induces the following splitting of the homotopy groups
\begin{equation}\label{eq.homotopysplitting}
\pi_k(C^0(\M;\N))=\pi_k(\Omega(\M;\N)) \oplus \pi_k(\N)
\end{equation}
and ultimately the following existence Theorem.
\begin{theorem}[Theorem 2.7 in \cite{sac}]\label{th.SU2.7}
    If $\Omega(\M;\N)$ is not contractible, then after possibly increasing the constant $B_{p,\delta}>0$ from \eqref{eq.boundmin}, for all $p\in(n,P_0)$ and $\delta \in [0,1]$, $E_{p,\delta}$ has a critical value in the interval $(\tfrac{1}{2}\eta,B_{p,\delta})$, where $\eta$ is the same of Theorem \ref{th.lowestenergyretraction}.
\end{theorem}
\begin{proof}
    The proof follows the same lines as in \cite{sac}. In the easier case, we assume that $C^0(\M;\N)$ is not connected. Then we can minimize $E_{p,\delta}$ in a connected component not containing $\N_0$ and from the discussion above we have \eqref{eq.boundmin}, that is we have the claimed upper bound. Theorem \ref{th.lowestenergyretraction} implies the claimed lower bound.

    Suppose instead that $C^0(\M;\N)$ is connected. By assumption, we can pick a homotopically non-trivial class $[\gamma] \in \pi_k(\Omega(\M;\N))$ for some $k\ge 0$, and by \eqref{eq.homotopysplitting} we must have $k \ge 1$. Fix any representative $\gamma \in [\gamma]$, and notice that $\gamma:S^k \rightarrow \Omega(\M;\N) \subset C^0(\M;\N)$ seen as a map into $C^0(\M;\N)$ is not homotopic to any map $\tilde{\gamma}:S^k \rightarrow \N_0$. Moreover, if we set
    \begin{equation}
        B:=\max_{\theta \in S^k} \max_{x \in \M} |\nabla_x [\gamma(\theta)](x)|,
    \end{equation}
    then we have the upper bound on the energies
    \begin{equation*}
        E_{p,\delta}(\gamma(\theta)) \le \tfrac{1}{p}[(1+(\delta+B^2)^{\frac{n}{2}})^{\frac{p}{n}}-(1+\delta^{\frac{n}{2}})^{\frac{p}{n}}] Vol(\M)=:B'_{p,\delta} \quad \forall \theta \in S^k.
    \end{equation*}
    Suppose by contradiction that no critical values in the energy range considered in the statement. Then by the general discussion above, there exists a deformation retraction $\rho$ between the level sets
    \begin{equation*}
        \rho:W^{1,p}(\M^n;\N) \cap E_{p,\delta}^{-1}( [0,B_{p,\delta}) ) \rightarrow E_{p,\delta}^{-1}((0,\tfrac{1}{2} \eta)).
    \end{equation*}
    Composing this with the retraction $\sigma$ from Theorem \ref{th.lowestenergyretraction}, we obtain a retraction
    \begin{equation*}
         \sigma \circ \rho:W^{1,p}(\M^n;\N) \cap E_{p,\delta}^{-1}([0,B_{p,\delta} ) ) \rightarrow \N_0.
    \end{equation*}
    Therefore, $\sigma \circ \rho \circ \gamma:S^k \rightarrow \N_0$ is homotopic to $\gamma$, a contradiction.
\end{proof}
The important Corollary of this result is the following analogue to Proposition $2.8$ in \cite{sac}.
\begin{theorem}[Proposition $2.8$ in \cite{sac}]\label{th.SU2.8}
If $\M=S^n$ and for some $k \in \mathbb{N}$ we have $\pi_{k+n}(\N)\neq 0$, then for the same constant $B_{p,\delta} >0$ as above, for all $p\in(n,P_0)$ and $\delta \in [0,1]$, $E_{p,\delta}$ has a critical value in the interval $(\tfrac{1}{2}\eta,B_{p,\delta} )$.
\end{theorem}
\begin{proof}
By assumption we know that $\pi_k(\Omega(S^n;\N))=\pi_{k+n}(\N)\neq 0$, hence is not contractible. Applying the Theorem \ref{th.SU2.7} above, we obtain the thesis.
\end{proof}
\subsection{Technical Lemmas for Regularity}
Once we have non-trivial critical maps for $E_{p,\delta}$, we need to analyze their regularity.

Already for fixed $p\in (n,P_0)$ and $\delta \in (0,1)$, the regularity of solutions to \eqref{eq.ELEpdelta} is not obvious since the right hand side is only in $\mathbb{L}^1$, not even when one drops the required uniformity of regularity. By Sobolev's embedding we know its a-priori $C^{0,\alpha}$-regularity, for $\alpha=1-\tfrac{n}{p}$. Assuming smallness of the $E_{p,\delta}$-energy, this can be used to deduce a Caccioppoli-type inequality and improve the integrability of the gradient of the solutions. However, no higher regularity can be deduced at this stage, indeed, in the easier case of $p$-harmonic maps, the only methods to deduce their $C^{1,\beta}$-regularity the authors are aware of, are the ones by Duzaar-Mingione \cite{duz3}, Fuchs \cite{fuc0} and Hardt-Lin \cite{har1}. Respectively, the analysis in \cite{duz3} is based on the assumption $\alpha \in (\tfrac{p}{p+1},1)$ and a excess decay estimate for a non-linear quantity of the solution (thus we need to have an exponent close enough to $1$, which the a-priori Sobolev regularity cannot guarantee); \cite{fuc0} follows similar arguments with a different notion of excess, not easily adaptable to our energies; in \cite{har1}, the authors develop a parametrisation procedure, writing locally the target $\N$ as a graph, around the small image of the $p$-harmonic map. This allows them to rewrite $D_p$ as a functional of scalar valued functions, and deduce that the critical map considered correspond to a critical function under this identification.
The method we choose, consists in showing that any solution of \eqref{eq.ELEpdelta} is actually locally minimizing, in a small enough ball, adapting \cite{far}, and then we combine \cite{luc1} and \cite{duz3} to get the not uniform smoothness.
\newline

With this aim in mind, we now prove some technical Lemmas adapting results by Fardoun-Regbaoui in \cite{far}.
\begin{lemma}[Lemma $2.1$ in \cite{far}]\label{lemma.uniquenessbounds}
Let $X,Y$ be two vectors in the Euclidean space $\R^K$, then for any $\delta \in [0,1]$ and $p\in (n,P_0)$ we have 
\begin{equation}\label{eq.uniquenesslowerbound}
\begin{aligned}
  &[(1+(\delta+|X|^2)^{\frac{n}{2}} )^{\frac{p-n}{n}} (\delta+|X|^2)^{\frac{n-2}{2}} X-(1+(\delta+|Y|^2)^{\frac{n}{2}} )^{\frac{p-n}{n}} (\delta+|Y|^2)^{\frac{n-2}{2}}  Y] \cdot (X-Y) \\
  &\ge \tfrac{1}{2} [(1+(\delta+|X|^2)^{\frac{n}{2}} )^{\frac{p-n}{n}} (\delta+|X|^2)^{\frac{n-2}{2}} +(1+(\delta+|Y|^2)^{\frac{n}{2}} )^{\frac{p-n}{n}} (\delta+|Y|^2)^{\frac{n-2}{2}} ] |X-Y|^2,
  \end{aligned}
\end{equation}
and
\begin{equation}\label{eq.uniquenessupperbound}
\begin{aligned}
 &|(1+(\delta+|X|^2)^{\frac{n}{2}} )^{\frac{p-n}{2 n}} (\delta+|X|^2)^{\frac{n-2}{4}} X-(1+(\delta+|Y|^2)^{\frac{n}{2}} )^{\frac{p-n}{2 n}} (\delta+|Y|^2)^{\frac{n-2}{4}}  Y| \\
 &\le  \tfrac{p}{2} [(1+(\delta+|X|^2)^{\frac{n}{2}} )^{\frac{p-n}{2 n}} (\delta+|X|^2)^{\frac{n-2}{4}}+(1+(\delta+|Y|^2)^{\frac{n}{2}} )^{\frac{p-n}{2 n}} (\delta+|Y|^2)^{\frac{n-2}{4}}] |X-Y|.
 \end{aligned}
\end{equation}
\end{lemma}
\begin{proof}
In order to prove the first inequality, we can argue as in the proof of \eqref{eq.elementary2} and assume $X=(1,0)$ and $Y=(r \cos(\theta), r \sin(\theta))$. Towards proving \eqref{eq.uniquenesslowerbound}, we develop its left hand side
\begin{align*}
    &[(1+(\delta+|X|^2)^{\frac{n}{2}} )^{\frac{p-n}{n}} (\delta+|X|^2)^{\frac{n-2}{2}} X-(1+(\delta+|Y|^2)^{\frac{n}{2}} )^{\frac{p-n}{n}} (\delta+|Y|^2)^{\frac{n-2}{2}}  Y] \cdot (X-Y)\\
    &= [(1+(\delta+1)^{\frac{n}{2}})^{\frac{p-n}{n}} (\delta+1)^{\frac{n-2}{2}} (1,0)-(1+(\delta+r^2)^{\frac{n}{2}})^{\frac{p-n}{n}} (\delta+r^2)^{\frac{n-2}{2}} (r \cos(\theta),r \sin{\theta})] \cdot\\
    &\cdot (1-r \cos(\theta),-r \sin{\theta})= f(1)-f(1) r \cos(\theta)-f(r) r \cos(\theta)+f(r) r^2 (\cos^2(\theta)+\sin^2(\theta))\\
    &=f(1)-f(1) r \cos(\theta)-f(r) r \cos(\theta)+f(r) r^2.
\end{align*}
Here we have set $f(t):=(1+(\delta+t^2)^{\frac{n}{2}} )^{\frac{p-n}{n}} (\delta+t^2)^{\frac{n-2}{2}}$.
The right hand side can be rewritten as
\begin{align*}
    &\tfrac{1}{2} [(1+(\delta+|X|^2)^{\frac{n}{2}} )^{\frac{p-n}{n}} (\delta+|X|^2)^{\frac{n-2}{2}} +(1+(\delta+|Y|^2)^{\frac{n}{2}} )^{\frac{p-n}{n}} (\delta+|Y|^2)^{\frac{n-2}{2}} ] |X-Y|^2\\
    &= \tfrac{1}{2} [f(1)+f(r)] (1-2r\cos(\theta)+r^2)= \tfrac{1}{2}f(1)+\tfrac{1}{2}f(r)-r f(1) \cos(\theta)-r f(r) \cos(\theta)+\tfrac{1}{2} r^2 f(1)+\tfrac{1}{2} r^2 f(r).
\end{align*}
Plugging in \eqref{eq.uniquenesslowerbound}, and rearranging the terms, we see that the claimed inequality is true if and only if
\begin{equation*}
    0\le \tfrac{1}{2} (f(1)+f(r)r^2)-\tfrac{1}{2} r^2 f(1)-\tfrac{1}{2}f(r)=\tfrac{1}{2}(r^2-1)(f(r)-f(1)).
\end{equation*}
which is clearly true for all values of $r\ge 0$.
\vspace{0.3cm}

Towards proving \eqref{eq.uniquenessupperbound}, we firstly set $F(Z):=(1+(\delta+|Z|^2)^{\frac{n}{2}} )^{\frac{p-n}{2 n}} (\delta+|Z|^2)^{\frac{n-2}{4}} Z$ for any $Z\in \R^K$, and then notice that the left-hand-side of \eqref{eq.uniquenessupperbound} is equal to $|F(X)-F(Y)|$. By the mean value theorem, we have
\begin{equation*}
    |F(X)-F(Y)| \le \sup_{t \in [0,1]} |\nabla F(tX+(1-t)Y)| |X-Y|.
\end{equation*}
Since we know that
\begin{align*}
    &|\nabla F(Z)|=| \tfrac{p-n}{2 n}(1+(\delta+|Z|^2)^{\frac{n}{2}})^{\frac{p-3 n}{2 n}} \tfrac{n}{2} (\delta+|Z|^2)^{\frac{n-2}{2}} 2 Z \otimes (\delta+|Z|^2)^{\frac{n-2}{4}} Z\\
    &+(1+(\delta+|Z|^2)^{\frac{n}{2}})^{\frac{p-n}{2 n}} \tfrac{n-2}{4} (\delta+|Z|^2)^{\frac{n-6}{4}} 2 Z \otimes Z+(1+(\delta+|Z|^2)^{\frac{n}{2}})^{\frac{p-n}{2 n}} (\delta+|Z|^2)^{\frac{n-2}{4}} Id | \\
    &\le (1+(\delta+|Z|^2)^{\frac{n}{2}})^{\frac{p-3 n}{2 n}} (\delta+|Z|^2)^{\frac{n-6}{4}} \Big[ \tfrac{p-n}{2}(\delta+|Z|^2)^{\frac{n}{2}}|Z|^2+\tfrac{n-2}{2} (1+(\delta+|Z|^2)^{\frac{n}{2}})|Z|^2\\
    &+ (1+(\delta+|Z|^2)^{\frac{n}{2}})(\delta+|Z|^2) \Big]\le (1+(\delta+|Z|^2)^{\frac{n}{2}})^{\frac{p-3 n}{2 n}} (\delta+|Z|^2)^{\frac{n-6}{4}} \tfrac{p}{2} (1+(\delta+|Z|^2)^{\frac{n}{2}})(\delta+|Z|^2)\\
    &=\tfrac{p}{2} (1+(\delta+|Z|^2)^{\frac{n}{2}})^{\frac{p-n}{2 n}} (\delta+|Z|^2)^{\frac{n-2}{4}}.
\end{align*}
Plugging in the inequality above, and using that $|t X+(1-t) Y|\le \max \set{|X|,|Y|}$ we conclude as follows
\begin{align*}
    &|(1+(\delta+|X|^2)^{\frac{n}{2}} )^{\frac{p-n}{2 n}} (\delta+|X|^2)^{\frac{n-2}{4}} X-(1+(\delta+|Y|^2)^{\frac{n}{2}} )^{\frac{p-n}{2 n}} (\delta+|Y|^2)^{\frac{n-2}{4}}  Y| \\
    &\le \tfrac{p}{2}  (1+(\delta+|(\max \set{|X|,|Y|})|^2)^{\frac{n}{2}})^{\frac{p-n}{2 n}} (\delta+|(\max \set{|X|,|Y|})|^2)^{\frac{n-2}{4}} |X-Y|\\
    &\le \tfrac{p}{2} \max \set{(1+(\delta+|X|^2)^{\frac{n}{2}} )^{\frac{p-n}{2 n}} (\delta+|X|^2)^{\frac{n-2}{4}},(1+(\delta+|Y|^2)^{\frac{n}{2}} )^{\frac{p-n}{2 n}} (\delta+|Y|^2)^{\frac{n-2}{4}} } |X-Y| \\
    &\le \tfrac{p}{2} [(1+(\delta+|X|^2)^{\frac{n}{2}} )^{\frac{p-n}{2 n}} (\delta+|X|^2)^{\frac{n-2}{4}}+(1+(\delta+|Y|^2)^{\frac{n}{2}} )^{\frac{p-n}{2 n}} (\delta+|Y|^2)^{\frac{n-2}{4}}] |X-Y|.
\end{align*}
\end{proof}
Recall a crucial inequality from \cite{far}.
\begin{lemma}[Lemma $2.2$ in \cite{far}]\label{lemma.secondfundamentalform}
There exists a constant $C_6=C_6(n,\N)$ such that for arbitrary $u,v \in \N$ and associated vectors $U \in (T_u \N)^n$ and $V \in (T_v \N)^n$, the second fundamental form satisfies
\begin{equation}\label{eq.secondfundamentalform}
    |A_u(U,U)-A_v(V,V)|\le C_6 (|U|^2+|V|^2) |u-v|+C_6(|U|+|V|)|U-V|,
\end{equation}
where $|\cdot|$ denotes the Euclidean norm on the tangent spaces $T_u \N$ and $T_v \N$.
\end{lemma}
In the following, we prove a weighted Poincar\'e inequality, assuming that the image of a solution to \eqref{eq.ELEpdelta} is contained in a small enough ball. This should be compared to the intrinsic Poincar\'e's inequalities typical of double-phase problems, for example Theorem $1.6$ in \cite{col0}.
\begin{lemma}[Proposition 2.1 in \cite{far}]\label{lemma.imagepoincare}
By possibly restricting the constant $\rho_0$ from Proposition \ref{prop.dirichlet}, and for all $p\in (n,P_0)$ and $\delta \in (0,1)$ the following statement holds true. If $u \in W^{1,p}(B_{R_0}(x_0);\N)$ is a solution to \eqref{eq.ELEpdelta} satisfying $u(B_{R_0}(x_0)) \subseteq B^\N_{\rho}(y_0)$ for some point $y_0 \in \N$ and radius $\rho \in (0,\rho_0]$, then for any test function $\phi \in  W^{1,p}_0(B_{R_0}(x_0);\R^N)$ we have
\begin{equation}\label{eq.imagepoincare}
   \footnotesize \int_{B_{R_0}(x_0)} (1+(\delta+|\nabla u|^2)^{\frac{n}{2}})^{\frac{p-n}{n}}(\delta+|\nabla u|^2)^{\frac{n-2}{2}}|\nabla u|^2 |\phi|^2 \le 16 \rho^2 \int_{B_{R_0}(x_0)} (1+(\delta+|\nabla u|^2)^{\frac{n}{2}})^{\frac{p-n}{n}}(\delta+|\nabla u|^2)^{\frac{n-2}{2}} |\nabla \phi|^2.
\end{equation}
\end{lemma}
\begin{proof}
By density, we can consider $\phi \in C^\infty_0(B_{R_0}(x_0);\R^N)$. Test the system \eqref{eq.ELEpdelta} with $|\phi|^2 (u-y_0)$, to obtain, using the hypothesis together with Cauchy-Schwartz's inequality
\begin{align*}
&\int_{B_{R_0}(x_0) } (1+(\delta+|\nabla u|^2)^{\frac{n}{2}})^{\frac{p-n}{n}}(\delta+|\nabla u|^2)^{\frac{n-2}{2}} |\nabla u|^2 |\phi|^2 \\
&\le \int_{B_{R_0}(x_0) }(1+(\delta+|\nabla u|^2)^{\frac{n}{2}})^{\frac{p-n}{n}}(\delta+|\nabla u|^2)^{\frac{n-2}{2}} |A_u(\nabla u, \nabla u)| |\phi|^2 |u-y_0|\\
&+\int_{B_{R_0}(x_0) } 2 |\phi| |\nabla \phi|(1+(\delta+|\nabla u|^2)^{\frac{n}{2}})^{\frac{p-n}{n}}(\delta+|\nabla u|^2)^{\frac{n-2}{2}} |\nabla u| |u-y_0| \\
&\le C(\N) \rho \int_{B_{R_0}(x_0) }(1+(\delta+|\nabla u|^2)^{\frac{n}{2}})^{\frac{p-n}{n}}(\delta+|\nabla u|^2)^{\frac{n-2}{2}} |\nabla u|^2 |\phi|^2 \\
&+ 2 \rho \Big(\int_{B_{R_0}(x_0) } (1+(\delta+|\nabla u|^2)^{\frac{n}{2}})^{\frac{p-n}{n}}(\delta+|\nabla u|^2)^{\frac{n-2}{2}} |\nabla \phi|^2 \Big)^{\frac{1}{2}} \cdot \\
&\cdot \Big(\int_{B_{R_0}(x_0) } (1+(\delta+|\nabla u|^2)^{\frac{n}{2}})^{\frac{p-n}{n}}(\delta+|\nabla u|^2)^{\frac{n-2}{2}} |\nabla u|^2 |\phi|^2 \Big)^{\frac{1}{2}}.
\end{align*}
Taking $\rho_0$ small enough so that $2 C(\N) \rho_0 < 1$, we conclude rearranging this inequality.
\end{proof}
We recall the following Proposition from Luckhaus \cite{luc1}, restricted to the particular autonomous case of interest for us (by the smoothness of the Riemannian metrics involved there is no loss of generality in doing so). This will guarantee the H\"older regularity with respect to arbitrary exponents $\alpha \in (0,1)$ for minimizers of the $E_{p,\delta}$-energy.
\begin{proposition}[Proposition 1 in \cite{luc1}]\label{prop.luckhaus}
Let $p>1$ and $u \in W^{1,p}(B_{2 R_0}(x_0);\N)$ be locally minimizing the energy $G(u):=\int_{B_{2 R_0}(x_0)} G(\nabla u)$, satisfying the following assumptions.
 \begin{itemize}
     \item[(A1)] $G(\cdot)$ is convex and verifies the inequality
     \begin{equation*}
         l_1^{-1} |\xi|^p-l_2 \le G(\xi) \le l_1 |\xi|^p+l_2;
     \end{equation*}
     \item[(A2)] Any possible blow-up function $F(\xi):= \lim_{i \rightarrow \infty} |\lambda_i|^{-p} G(\lambda_i \xi)$ along a sequence $\lambda_i \rightarrow +\infty$, is such that any solution $v \in W^{1,p}(B_{2 R_0}(x_0);\R^N)$ of the associated Euler-Lagrange system
     \begin{equation*}
         -\Div[\partial_\xi F(\nabla v)]=0,
     \end{equation*}
     is in $C^{0,\mu}(B_{R_0}(x_0);\R^N)$ with bound $[v]_{C^{0,\mu}} \le l_{3,\mu} \norm{v}_{W^{1,p}}$ for some $l_{3,\mu} >0$ and $\mu \in (0,1)$.
 \end{itemize}
 Then for every $\Tilde{\mu}<\mu$ there exist constants $L$, $L_{3}$, $\Tilde{\e}_0$ and $\Tilde{\tau}_\mu$, all of them depending on $\N$, $G$ and $\Tilde{\mu}$ such that, for all radii $r<\Tilde{\tau}_\mu R_0$ and balls $B_r(x)\subset B_{R_0}(x_0)$ such that $r^{p-n} D_p(u;B_r(x) )<\Tilde{\e}_0^p$, we have
 \begin{equation}
     r'^{p-n} \int_{B_{r'}(x)} |\nabla u|^p \le L \Big( \frac{r'}{r} \Big)^{\Tilde{\mu} p} r^{p-n} \int_{B_{r}(x)} |\nabla u|^p, \quad \forall r' \le r,
 \end{equation}
 or equivalently, $u \in C_{loc}^{0,\Tilde{\mu}}(B_{\Tilde{\tau}_\mu R_0};\N)$ with bound $[u]_{C^{0,\Tilde{\mu}}} \le L_{3} R_0^{1-\frac{n}{p} } D_p(u;B_{R_0}(x_0))^{\frac{1}{p}}$.
\end{proposition}
\begin{remark}
As implicitely highlighted in the Introduction, Proposition $1$ in \cite{luc1} applies also to some double-phase problems, as long as condition $(A2)$ is met.
\end{remark}
\begin{remark}
We will see that for minimizers as in Theorem \ref{th.SUregularity}, point $(a)$, this theorem can be applied with constants uniform in $\delta$ and $p$. Indeed the hypotheses $(A1)$ and $(A2)$ are gonna be verified uniformly. At the same time, the proof of this proposition in \cite{luc1} relies on a blow-up argument and a comparison map construction, Lemma $1$ there, which are both uniform in our parameters. The latter, is based on a decomposition of $\partial B_{R_0}(x_0)$, in cells of dimension $j=1,...,n-1$, and the extension is done linearly if $j \le \floor{p-1}$, with the use of Sobolev embedding in order to control the oscillations, and $0$-homogeneous for the higher dimensional cells. In our situation, the second case never appears, and the first extension leads to uniform estimates as $p \searrow n$, for any $\delta \in [0,1]$, since $W^{1,p}(\R^{n-1}) \subseteq C^{0,\frac{1}{n}}(\R^{n-1})$, as we shall see below. Here we will be assuming $P_0 < n+1$. Compare with \cite{luc1}.
\end{remark}
\begin{lemma}[Uniform Luckhaus' Lemma]
Suppose $p \in (n,n+1)$, $u,v \in W^{1,p}(S^{n-1};\R^N)$, $\lambda \in (0,\tfrac{1}{2}]$ and $\e \in (0,1)$. Set
\begin{equation*}
    \int_{S^{n-1}} |\nabla u|^p+|\nabla v|^p+\tfrac{|u-v|^p}{\e^p}=:K^p.
\end{equation*}
There exist constants $C_6(n)$, $C_7(n)$ and a map $\phi:A(1-\lambda,1) \longrightarrow \R^N$ such that
\begin{align*}
\phi(x)=u(x) \ \ \text{if } |x|=1,\\
\phi(x)=v(\tfrac{x}{1-\lambda}) \ \ \text{if } |x|=1-\lambda,\\
\int_{A(1-\lambda,1)} |\nabla \phi|^p \le C_6 K^p \Big(1+\big(\tfrac{\e}{\lambda} \big)^p \Big) \lambda,\\
\phi(A(1-\lambda,1)) \subset T_\rho(u(S^{n-1})) \cup T_\rho(v(S^{n-1}) ), 
\end{align*}
where $\rho=C_7 K\e^{1-\beta_0}$, and $\beta_0:=\tfrac{2n-1}{2n}$.
\end{lemma}
\begin{proof}
Decompose $S^{n-1}$ into cells of diameter $\lambda$ as follows (it is convenient to set $\lambda=:2^{-\nu}$). Since $B_1$ is bi-Lipschitz homeomorphic to the open unit cube, we firstly decompose the boundary of the latter into open cubes $Q^j_i$ of side length $\lambda$ and dimensions between $0$ and $n-1$. There exist bi-Lipschitz homeomorphisms $\Phi_i^j$ (with distortions bounded by some dimensional constant) from these cubes into the desired disjoint cells $\set{e_i^j}$, so we obtain a decomposition
\begin{equation*}
    S^{n-1}:=\bigcup_{j=0}^{n-1} \bigcup_{i=1}^{k_j} e_i^j=:\bigcup_{j=0}^{n-1} Q_j.
\end{equation*}
For any positive measurable function $f$ we have
\begin{equation*}
    \int_{SO(n)} \int_{\omega Q_j} f d \mu_{S^{n-1}} d\omega= \mathcal{H}^j(Q_j) \int_{S^{n-1}} f.
\end{equation*}
Thus one can choose a rotation $\omega \in SO(n)$ such that for all $j$ one has
\begin{equation*}
\int_{\omega Q_j} |\nabla u|^p+|\nabla v|^p+\tfrac{|u-v|^p}{\e^p} \le c(n) K^p \lambda^{j+1-n}.
\end{equation*}
Without loss of generality we assume $\omega=id$. Sobolev's embedding allows to control the oscillations of $u-v$ on $Q_j$ for any $j \le \floor{p-1}=n-1$, hence for all $j$. More precisely, since by our choice $\beta \in (\tfrac{n-1}{n},1)$, if a cell $e \in Q_j$ is bi-Lipschitz homeomorphic to a ball of radius $r$ in $\R^j$, then as in \cite{luc1}
\begin{equation*}
\text{osc}_e |u-v| \le c(n,\mu) r^\mu \Big( \int_{e} |\nabla (u-v)|^p \Big)^{\frac{\beta}{p}} \Big( \int_{e} |u-v|^p \Big)^{\frac{1-\beta}{p}},
\end{equation*}
where $\mu< \tfrac{1}{j}-\tfrac{1}{p}$, for example we can take $\mu:=\tfrac{1}{n(n-1)}$, hence uniform in $p$. By construction, every cell of diameter $2^{-\nu}$ can be joined to a cell of diameter $2^{-\nu+1}$ by at most $2^{n-j}$-cells of diameter $2^{-\nu}$.
Summing up over $\nu$ we get as in \cite{luc1}
\begin{equation*}
\text{osc}_{Q_j} |u-v| \le c(n) \Big( \int_{Q_j} |\nabla u|^p+|\nabla v|^p \Big)^{\frac{\beta}{p}} \Big( \int_{Q_j} |u-v|^p \Big)^{\frac{1-\beta}{p}} \le c(n)  K \e^{1-\beta} \lambda^{\frac{j+1-n}{p}},
\end{equation*}
for all $j$. The decomposition in cells extends to the whole annulus $A(1-\lambda,1)$ through cells
\begin{equation*}
    \hat{e}_i^j:= \set{x \mid \tfrac{x}{|x|} \in e_i^j, \ 1-\lambda < |x| < 1}.
\end{equation*}
We interpolate linearly
\begin{equation*}
\phi(x):=u(\tfrac{x}{|x|})+ \tfrac{1-|x|}{\lambda}\big( v(\tfrac{x}{|x|})-u(\tfrac{x}{|x|}) \big).
\end{equation*}
Therefore, we can bound
\begin{equation*}
\int_{\hat{e}_i^j} |\nabla \phi|^p \le c \int_{1-\lambda}^1 r^{(j+1)-1} \int_{\partial B_r(x_0) \cap \hat{e}_i^j} |\nabla \phi|^p \le c \lambda \int_{S^{n-1}} |\nabla u|^p+|\nabla v|^p+\lambda^{-p} |u-v|^p.
\end{equation*}
Let us denote by $\hat{Q}_j:=\bigcup_i \hat{e}_i^j$. We deduce
\begin{equation*}
\int_{\hat{Q}_{n-1}} |\nabla \phi|^p \le c(n) \lambda^{(n-1)+2-n} K^p \Big(1+\big(\tfrac{\e}{\lambda} \big)^p \Big)=c(n) \lambda K^p \Big(1+\big(\tfrac{\e}{\lambda} \big)^p \Big).
\end{equation*}
\end{proof}

Next, we prove an analogue of Lemma 3 in \cite{duz3}, allowing us to link the H\"older continuity of the gradient of the solutions of \eqref{eq.ELEpdelta} to that of their corresponding non-linear $V$ quantity defined in the previous Subsection \ref{subsec.elementary}.
\begin{lemma}\label{lemma.nonlinearholder}
For $p\in (n,P_0)$ and $\delta \in [0,1]$, consider a matrix field $F:B_r(x) \rightarrow \R^{n N}$ so that the non-linear function $V(F)(x):=(1+(\delta+|F(x)|^2)^{\frac{n}{2}})^{\frac{p-n}{2 n}}(\delta+|F(x)|^2)^{\frac{n-2}{4}} F(x)$ is $C^{0,\beta}$-H\"older continuous for some exponent $\beta \in (0,1)$ and some constant $C$, both uniform in $p$ and $\delta$. Then $F$ is $C^{0, \Tilde{\beta}}$-regular where $\Tilde{\beta}:= \tfrac{2 \beta}{p}$, with semi-norm bound $[F]_{C^{0,\Tilde{\beta_1}} } \le C_8(P_0,n,\N,C,\beta)[V(F)]_{C^{0,\beta} }^{\frac{1}{p} }$ uniform in $p \in (n,P_0)$ and $\delta \in [0,1]$.
\end{lemma}
\begin{proof}
From the second inequality in \eqref{eq.elementary2} we deduce that for all $x,y \in \bar{B}_s$
\begin{align*}
    |F(x)-F(y)|^p \le c_1^{-1} |V(F)(x)-V(F)(y)|^2 \le C [V(F)]_{C^{0,\beta} } |x-y|^{2 \beta}.
\end{align*}
\end{proof}
Finally, let us prove an analogue to Riesz' theorem, which we will use in Section \ref{sec.minmax}.
\begin{lemma}\label{lemma.riesz}
Let $p \in (n,P_0)$, $\delta \in [0,1]$, and suppose $(f_k)_{k \in \mathbb{N}},f \in \mathbb{L}^p$ are such that $f_k \rightarrow f$ almost everywhere and $E_{p,\delta}(f_k) \rightarrow E_{p,\delta}(f)$. Then $f_k \rightarrow f$ strongly in $\mathbb{L}^p$.
\end{lemma}
\begin{proof}
The functions
\begin{equation*}
 g_k:=2^p[((1+(\delta+|f_k|^2)^{\frac{n}{2}})^{\frac{p}{n}}-(1+\delta^{\frac{n}{2}})^{\frac{p}{n}})+((1+(\delta+|f|^2)^{\frac{n}{2}})^{\frac{p}{n}}-(1+\delta^{\frac{n}{2}})^{\frac{p}{n}})]-|f_k-f|^p
\end{equation*}
are non negative since by \eqref{eq.elementary1} we can bound
\begin{equation*}
    g_k \ge 2^p(|f_k|^p+|f|^p)-|f_k-f|^p \ge 0,
\end{equation*}
and are converging almost everywhere by assumption to
\begin{equation*}
2^{p+1}((1+(\delta+|f|^2)^{\frac{n}{2}})^{\frac{p}{n}}-(1+\delta^{\frac{n}{2}})^{\frac{p}{n}}).
\end{equation*}
By assumption, and using Fatou's Lemma, we obtain
\begin{align*}
&2^{p+1} E_{p,\delta}(f) \le \int 2^{p+1} ((1+(\delta+|f|^2)^{\frac{n}{2}})^{\frac{p}{n}}-(1+\delta^{\frac{n}{2}})^{\frac{p}{n}}) \le \liminf_k \int g_k \\
&\le \liminf_k \int 2^p[((1+(\delta+|f_k|^2)^{\frac{n}{2}})^{\frac{p}{n}}-(1+\delta^{\frac{n}{2}})^{\frac{p}{n}})+((1+(\delta+|f|^2)^{\frac{n}{2}})^{\frac{p}{n}}-(1+\delta^{\frac{n}{2}})^{\frac{p}{n}})]\\
&= \liminf_k 2^{p} E_{p,\delta}(f_k)+2^{p} E_{p,\delta}(f)= 2^{p+1} E_{p,\delta}(f).
\end{align*}
Therefore we must have a chain of identities, concluding the proof.
\end{proof}
\section{Uniform Regularity}\label{sec.regularity}
As described in the introduction, our proof of Theorem \ref{th.SUregularity} is split in three cases. 
\subsection{Locally minimizing maps}\label{subsec.minimizers}
Throughout this subsection, we will assume to have a family $(u_{p,\delta})$ of maps all locally minimizing $E_{p,\delta}$ on a fixed ball $B_{R_0}(x_0)$. The uniform H\"older regularity relies on Luckhaus' regularity result Proposition \ref{prop.luckhaus}. The uniform higher order regularity is then deduced through an excess decay estimate relative to non-linear functions of the gradients of the solutions, adapting \cite{duz3}.
\begin{proof}[Proof of Theorem \ref{th.SUregularity}, part (a)]
Firstly, we will prove the H\"older continuity relative to an arbitrary exponent $\tilde{\mu} \in (0,1)$. In order to do so, we are going to apply Proposition \ref{prop.luckhaus}. To start, we notice that from the inequality \eqref{eq.elementary1} we deduce
\begin{equation*}
R_0^{p-n} D_p(u_{p,\delta};B_{R_0}(x_0)) \le R_0^{p-n} E_{p,\delta}(u_{p,\delta};B_{R_0}) \le \e_0^p.
\end{equation*}
Moreover, $\e_0$ can be made as small as we wish, in particular less than the constant $\tilde{\e}$ in Proposition \ref{prop.luckhaus}, and we will fix it later. Observe that the hypotheses $(A1)$ and $(A_2)$ in that Proposition are clearly satisfied by our functionals, and more precisely, we prove below their uniformity for $p \in (n,P_0)$ and $\delta \in [0,1]$.

The functionals in consideration $E_{p,\delta}$ have integrands $G(\xi):= (1+(\delta+|\xi|^2)^{\frac{n}{2}})^{\frac{p}{n}}-(1+\delta^{\frac{n}{2}})^{\frac{p}{n}}$. Hypothesis $(A1)$ is verified for constants $l_1=C_3(n,P_0)$ and $l_2=a_0$ by \eqref{eq.elementary1}, once recalled that $G(\xi)$ is convex (already observed in Section \ref{sec.preliminary}). For what regards $(A2)$, that is the regularity property for the blow-up equation, this is satisfied for any $\mu \in (0,1)$ uniformly, since we have $F(\xi):=|\xi|^p$ (the standard $p$-energy integrand), and the $C^{0,\mu}$-regularity for arbitrary $\mu \in (0,1)$ comes from the classical result of Uhlenbeck \cite{uhl0}. Let us remark that this is uniform for $p \in (n,P_0)$ and $\delta \in [0,1]$.

Therefore, for any $\tilde{\mu} \in (0,1)$, we can fix $\mu \in (\tilde{\mu} ,1)$ and use Proposition \ref{prop.luckhaus} to deduce that $u_{p,\delta} \in C^{0,\tilde{\mu} }_{loc}(B_{ \tilde{\tau}_\mu R_0};\N)$ with uniform bounds. More precisely, for any $r_1 <\tilde{\tau}_\mu R_0$ and balls $B_{r_1}(x_1) \subset B_{R_0}(x_0)$ we have the uniform Morrey norm bound
\begin{equation}\label{eq.morreydelta}
\norm{\nabla u_{p,\delta} }_{\mathbb{L}^{p,q-n}(B_{r_1}(x_1))} \sim [u_{p,\delta}]_{C^{0,\Tilde{\mu} }(B_{r_1}(x_1))} \le C R_0^{\frac{p}{q}-\frac{n}{p} } E_{p,\delta}(B_{R_0}(x_0))^{\frac{1}{p}} \quad \text{for } q=\tfrac{p}{1-\Tilde{\mu} }.
\end{equation}

In particular, we can choose $\Tilde{\mu} \in (\tfrac{p}{p+1},1)$ and we now adapt the proof from Lemma $6$ in \cite{duz3} to get higher regularity as follows. Firstly, we fix a ball $B_{r_2}(x_2) \subset \subset B_{r_1}(x_1)$ (without loss of generality $r_2<1$) and let $h_{p,\delta} \in W^{1,p}(B_{r_2}(x_2);\R^N)$ be the unique map such that $Tr(h_{p,\delta})=Tr(u_{p,\delta})$ along $\partial B_{r_2}(x_2)$ solving weakly
\begin{equation}\label{eq.harmonicdelta}
-\Div[(1+(\delta+|\nabla h|^2)^{\frac{n}{2}})^{\frac{p-n}{n}}(\delta+|\nabla h|^2)^{\frac{n-2}{2}} \nabla h]=0.
\end{equation}
We choose the test map $\phi:= u_{p,\delta}-h_{p,\delta} \in W^{1,p}_0(B_{r_2}(x_2);\R^N)$ in both the systems \eqref{eq.ELEpdelta} and \eqref{eq.harmonicdelta} and take their difference to get (we drop the pedices $_{p,\delta}$)
\begin{equation}
\begin{aligned}
&\int_{B_{r_2}(x_2)} [(1+(\delta+|\nabla u|^2)^{\frac{n}{2}})^{\frac{p-n}{n}}(\delta+|\nabla u|^2)^{\frac{n-2}{2}} \nabla u -(1+(\delta+|\nabla h|^2)^{\frac{n}{2}})^{\frac{p-n}{n}}(\delta+|\nabla h|^2)^{\frac{n-2}{2}} \nabla h](\nabla u-\nabla h) \\
&=\int_{B_{r_2}(x_2)} (1+(\delta+|\nabla u|^2)^{\frac{n}{2}})^{\frac{p-n}{n}}(\delta+|\nabla u|^2)^{\frac{n-2}{2}} A_{u}(\nabla u, \nabla u)\cdot (u-h).
\end{aligned}
\end{equation}
Using the elementary inequality \eqref{eq.elementary2}, we can bound the left hand side from below as
\begin{equation*}
\begin{aligned}
&\int_{B_{r_2}(x_2)} [(1+(\delta+|\nabla u|^2)^{\frac{n}{2}})^{\frac{p-n}{n}}(\delta+|\nabla u|^2)^{\frac{n-2}{2}} \nabla u -(1+(\delta+|\nabla h|^2)^{\frac{n}{2}})^{\frac{p-n}{n}}(\delta+|\nabla h|^2)^{\frac{n-2}{2}} \nabla h](\nabla u-\nabla h) \\
&\ge c_0 \int_{B_{r_2}(x_2)} |V(\nabla u)-V(\nabla h)|^2,
\end{aligned}
\end{equation*}
where we recall $V(\xi):=(1+(\delta+|\xi|^2)^{\frac{n}{2}})^{\frac{p-n}{2 n}}(\delta+|\xi|^2)^{\frac{n-2}{4}} \xi$. For the right hand side, we argue exactly as in \cite{duz3}
\begin{align*}
&\int_{B_{r_2}(x_2)} (1+(\delta+|\nabla u|^2)^{\frac{n}{2}})^{\frac{p-n}{n}}(\delta+|\nabla u|^2)^{\frac{n-2}{2}} A_{u}(\nabla u, \nabla u)\cdot (u-h) \\
&\le c \sup_{B_{r_2}(x_2)} |u-h| \int_{B_{r_2}(x_2)} (1+|\nabla u_p|^n)^{\frac{p}{n}}\le c E_{p,\delta}(B_{r_2}(x_2))^{1+\frac{1}{p}} r_2^{\Tilde{\mu} +p\Tilde{\mu} +n-p},
\end{align*}
where we have used \eqref{eq.morreydelta}, $\tilde{\mu}<1$, as well as the $\tilde{\mu}$-H\"older regularity of $h$, see \cite{gia1}; let us explicitely remark that we can appeal to \cite{gia1} thanks to the convexity of the integrand of $E_{p,\delta}$, ensuring that any ($\R^N$-valued) critical map is a minimizer with respect to its own boundary values. Notice that $\gamma_1:=\Tilde{\mu} +p\Tilde{\mu} -p \in (0,1)$, so that after multiplying by $r_2^{-n}$ we arrive to
\begin{equation}
\fint_{B_{r_2}(x_2)} |V(\nabla u)-V(\nabla h)|^2 \le c  E_{p,\delta}(B_{r_2}(x_2))^{1+\frac{1}{p}} r_2^{\gamma_1}.
\end{equation}
The results in \cite{gia1} apply to $h$, ensuring that for some constants $c>0$ and $\tau, \gamma_2 \in (0,1)$
\begin{equation}
\Phi(u;\tau \rho,y):=\fint_{B_{\tau r_2}(x_2)} |V(\nabla h)-[V(\nabla h)]_{B_{\tau r_2}(x_2)}|^2 \le \tau^{2 \gamma_2} \fint_{B_{r_2}(x_2)} |V(\nabla h)-[V(\nabla h)]_{B_{r_2}(x_2)}|^2.
\end{equation}
This is enough to deduce the decay inequality $(64)$ in \cite{duz3}, which eventually implies $V(\nabla u) \in C^{0,\alpha}_{loc}$ for $\alpha:=\min \set{\gamma_1,\gamma_2}$ and hence $u \in C^{1,\Tilde{\alpha}}_{loc}(B_{\tau_0 R_0}(x_0);\N)$ by our Lemma \ref{lemma.nonlinearholder}, concluding the proof.
\end{proof}
\begin{remark}
Notice that from $p>n$ we deduce
\begin{equation*}
R_1^{p-n} D_p(u_{p,\delta};B_{R_1}(x_1)) \le R_0^{p-n} D_p(u_{p,\delta};B_{R_0}(x_0)) \le R_0^{p-n} E_{p,\delta}(u_{p,\delta};B_{R_0}) \le \e_0^p,
\end{equation*}
for all balls $B_{R_1}(x_1) \subset B_{R_0}(x_0)$, so we could deduce a uniform full $C^{1,\alpha_0}(B_{R_0}(x_0))$-regularity using a covering argument.
\end{remark}

\subsection{Homogeneous targeted maps}
In this subsection, we consider the second case treated in Theorem \ref{th.SUregularity}: the homogeneous targeted critical maps. The proof follows closely the one by Toro-Wang in \cite{tor}, therefore we start by recalling some common preliminary material.

Let $(\N,h)$ be a smooth compact homogeneous space with a left invariant metric $h$. We can identify it with a quotient of Lie groups $G/H$, where $G$ is connected and $H$ is a closed subgroup of it. The Lie algebra of $G$ will be denoted by $\mathfrak{g}$.

Consider a solution $(u_{p,\delta})$ of \eqref{eq.ELEpdelta} and a smooth Killing tangent vector field $X$ on $\N$. Then the tangent vector field $W:=(1+(\delta+|\nabla u|^2)^{\frac{n}{2}})^{\frac{p-n}{n}}(\delta+|\nabla u|^2)^{\frac{n-2}{2}} h(\nabla u,X(u))$ is divergence free in the weak sense: indeed for any test function $\zeta \in C^{\infty}_c$ we can test \eqref{eq.ELEpdelta} with $\zeta X(u)$ (which is tangent to $\N$ and therefore orthogonal to $A_u$)
\begin{align*}
0&=\int (1+(\delta+|\nabla u|^2)^{\frac{n}{2}})^{\frac{p-n}{n}}(\delta+|\nabla u|^2)^{\frac{n-2}{2}} \nabla u \nabla (\zeta X(u))\\
&= \int (1+(\delta+|\nabla u|^2)^{\frac{n}{2}})^{\frac{p-n}{n}}(\delta+|\nabla u|^2)^{\frac{n-2}{2}} \nabla u[ \nabla \zeta X(u) + \zeta (\nabla^h_{\nabla u} X)(u)] \\
&= \int (1+(\delta+|\nabla u|^2)^{\frac{n}{2}})^{\frac{p-n}{n}}(\delta+|\nabla u|^2)^{\frac{n-2}{2}} \nabla u X(u) \nabla \zeta.
\end{align*}
In the last step, we have used that for the Killing vector field, one must have $\scal{\nabla u}{(\nabla^h_{\nabla u} X)(u)}=0$. This is exactly the divergence-free property claimed above.

Helein in \cite{hel4} proved that in homogeneous manifolds $(\N,h)$, there exist $\ell:=dim(\mathfrak{g})$-smooth tangent vector fields $Y_1$,..., $Y_\ell$ and $\ell$ smooth Killing tangent vector fields $X_1$,..., $X_\ell$ generating in the following sense: for any tangent vector field $V$ one can decompose it as follows
\begin{equation*}
V=\sum_{i=1}^\ell h(V,X_i)Y_i.
\end{equation*}
Applying this to the vector field defining the equation we get \eqref{eq.ELEpdelta}
\begin{equation*}
    (1+(\delta+|\nabla u|^2)^{\frac{n}{2}})^{\frac{p-n}{n}}(\delta+|\nabla u|^2)^{\frac{n-2}{2}} \nabla u= (1+(\delta+|\nabla u|^2)^{\frac{n}{2}})^{\frac{p-n}{n}}(\delta+|\nabla u|^2)^{\frac{n-2}{2}} \sum_{i=1}^\ell h(\nabla u,X_i(u))Y_i(u),
\end{equation*}
and recalling that $W$ is divergence-free we obtain
\begin{equation*}
-\Div[(1+(\delta+|\nabla u|^2)^{\frac{n}{2}})^{\frac{p-n}{n}}(\delta+|\nabla u|^2)^{\frac{n-2}{2}}\nabla u^\alpha]= \sum_{i=1}^\ell W^\alpha \nabla ( Y_i(u)).
\end{equation*}
This system is equivalent to \eqref{eq.ELEpdelta} as every step so far can be run backwards, but now the right hand side enjoys the gradient-curl structure suitable to unlock the classical compensated compactness results a la Coifman-Meyer-Lions-Semmes. Combining this with the Hardy-BMO duality, we can now prove Theorem \ref{th.SUregularity}, point $(b)$; the argument is classical by now, but we will repeat it for sake of completeness, and because there is some novel minor difficulty due to the double-phase type nature of the problem.

We start by recalling the following result from Strzelecky, allowing to avoid the use of the local Hardy space $\mathcal{H}^1_{loc}$ (which would still produce a rigorous proof, see \cite{tor}).
\begin{lemma}[Corollary 3 in \cite{strz2}]
Let $B$ be a ball in $\R^n$. Suppose that for some $p\in(1,\infty)$, we have $Y \in W^{1,p}(B)$ and that
$W \in \mathbb{L}^{\frac{p}{p-1}}(B;\R^n)$. Assume further that $W$ is divergence free in the distributional sense. Then there exists a function $h \in \mathcal{H}^1(\R^n)$ such that
$h=\nabla Y \cdot W$ in $B$, and that satisfies the bound
\begin{equation*}
\norm{h}_{\mathcal{H}^1(\R^n)} \le C_9 \norm{\nabla Y}_{\mathbb{L}^{p} } \norm{W}_{\mathbb{L}^{\frac{p}{p-1}} }
\end{equation*}
The constant $C_9$ does not depend on the size of $B$, and can be also chosen uniformly for $p$ in a neighbourhood of $n$.
\end{lemma}
The claimed uniformity for $p\in (n,P_0)$ follows clearly from the proof in \cite{strz2}, since it relies on the extension operator in Sobolev spaces and on Proposition $4.1$ in \cite{iwa3}, both holding with uniform constants.
\begin{proof}[Theorem \ref{th.SUregularity}]
First of all we consider an arbitrary ball $B_{2r}(x) \subset B_{R_0}(x_0)$. From the Lemma above, there exist $h^1,...,h^N \in \mathcal{H}^1$ such that \eqref{eq.ELEpdelta} implies
\begin{equation*}
-\Div[(1+(\delta+|\nabla u|^2)^{\frac{n}{2}})^{\frac{p-n}{n}}(\delta+|\nabla u|^2)^{\frac{n-2}{2}}\nabla u^\alpha]=h^\alpha \quad \text{on } B_{2r}(x),
\end{equation*}
where we also have the uniform bound
\begin{equation*}
\norm{h^\alpha}_{\mathcal{H}^1} \le C \norm{\nabla (Y(u))}_{\mathbb{L}^{p} } \norm{W}_{\mathbb{L}^{\frac{p}{p-1}} } \le C \int_{B_{2r}(x)} (1+(\delta+|\nabla u|^2)^{\frac{n}{2}})^{\frac{p}{n}} \le C E_{p,\delta}(u;B_{2r}(x))+Cr^n.
\end{equation*}
Notice that we have used \eqref{eq.elementary1} in the last step. Consider a cut off function $\eta \in C^\infty_c(B_{2 r}(x))$ with $\eta \equiv 1$ on $B_{r}(x)$, $\eta \in [0,1]$ and $|\nabla \eta| \le 2 r^{-1}$; we choose as test map $\phi:= \eta^p (u-[u]_{A(r,2r)})$, where $[u]_{A(r,2r)}$ is the average of $u$ on the annulus $A(r,2r)$, so that from the equation above we deduce after some routine calculation
\begin{align*}
&\int_{B_{r}(x)} (1+(\delta+|\nabla u|^2)^{\frac{n}{2}})^{\frac{p-n}{n}}(\delta+|\nabla u|^2)^{\frac{n-2}{2}} \tfrac{|\nabla u|^2}{2} \le\int_{B_{2r}(x)} \eta^p (1+(\delta+|\nabla u|^2)^{\frac{n}{2}})^{\frac{p-n}{n}}(\delta+|\nabla u|^2)^{\frac{n-2}{2}} \tfrac{|\nabla u|^2}{2} \\
&\le C \int_{A(r,2r)} r^{-2} (1+(\delta+|\nabla u|^2)^{\frac{n}{2}})^{\frac{p-n}{n}}(\delta+|\nabla u|^2)^{\frac{n-2}{2}} |u-[u]_{A(r,2r)}|^2 + \int_{B_{2r}(x)} \eta^p (u-[u]_{A(r,2r)}) h.
\end{align*}
In order to treat the first term, we use Young's inequality with exponents $(\tfrac{p}{p-n}, \tfrac{p}{n-2},\tfrac{p}{2})$, the bound $(\delta+|\nabla u|^2)^{\frac{p}{2}} \le (1+(\delta+|\nabla u|^2)^{\frac{n}{2}})^{\frac{p}{n}}$ and then Poincare' inequality together with \eqref{eq.elementary1} to get
\begin{align*}
&\int_{A(r,2r)} r^{-2} (1+(\delta+|\nabla u|^2)^{\frac{n}{2}})^{\frac{p-n}{n}}(\delta+|\nabla u|^2)^{\frac{n-2}{2}} |u-[u]_{A(r,2r)}|^2 \le \tfrac{p-2}{p} \int_{A(r,2r)} (1+(\delta+|\nabla u|^2)^{\frac{n}{2}})^{\frac{p}{n}}\\
&+C\int_{A(r,2r)} |\nabla u|^p \le C \int_{A(r,2r)} (1+(\delta+|\nabla u|^2)^{\frac{n}{2}})^{\frac{p}{n}}-(1+\delta^{\frac{n}{2}})^{\frac{p}{n}}+C r^n \le C E_{p,\delta}(u;A(r,2r))+C r^n.
\end{align*}
For what regards the second term, we first claim that the test function $\phi= \eta^p (u-[u]_{A(r,2r)})$ belongs to BMO and can be bounded uniformly in this space in terms of the rescaled energy $r^{1-\frac{n}{p}} E_{p,\delta}(u;B_{2r}(x))$. Indeed, since $spt(\phi) \subset B_{2r}(x)$ and $p>n$, then several applications of Poincare's and H\"older's inequalities, on the cube or the ball, combined with \eqref{eq.elementary0} and the bounds on $\eta$, allow us to bound
\begin{align*}
&\fint_Q |\phi-[\phi]_Q|\le \Big( \fint_Q |\phi-[\phi]_Q|^n \Big)^{\frac{1}{n}}\le C \Big(\int_Q |\nabla \phi|^n \Big)^{\frac{1}{n}} \le C \Big(\int_{B_{2r}(x)} |\nabla \phi|^n \Big)^{\frac{1}{n}} \le C r^{1-\frac{n}{p}} \Big(\int_{B_{2r}(x)} |\nabla \phi|^p \Big)^{\frac{1}{p}} \\
&\le C r^{1-\frac{n}{p}} \Big(\int_{B_{2r}(x)} |\nabla u|^p+\int_{A(r,2r)} |\tfrac{u-[u]_{A(r,2r)}}{r}|^p \Big)^{\frac{1}{p}} \le C r^{1-\frac{n}{p}} E_{p,\delta}(u;B_{2r}(x)),
\end{align*}
as claimed. Notice that we used once again \eqref{eq.elementary1}, and the constant is independent of $\delta$ and $p$.

Therefore the second term can be bounded via duality
\begin{align*}
\int_{B_{2r}(x)} \eta^p (u-[u]_{A(r,2r)}) h &\le \norm{\eta^p (u-[u]_{A(r,2r)})}_{BMO} \cdot \norm{h}_{\mathcal{H}^1} \\
&\le C r^{1-\frac{n}{p}} E_{p,\delta}(u;B_{2r}(x))( E_{p,\delta}(u;B_{2r}(x))+r^n)\le C \e_0 E_{p,\delta}(u;B_{2r}(x))+Cr^n.
\end{align*}
In the last step, the monotonicity of the rescaled energy and the smallness assumption are used.

The left hand side instead, controls the $E_{p,\delta}$-energy up to some good factors, as we see by some formal manipulations, the inequality $(\delta+|\nabla u|^2)^{\frac{n}{2}} \le (1+(\delta+|\nabla u|^2)^{\frac{n}{2}})$, and weighted Young's inequality with exponents $(\tfrac{p}{p-n},\tfrac{p}{n})$ and $(\tfrac{p}{p-2},\tfrac{p}{2})$:
\begin{align*}
&\tfrac{1}{2}\int_{B_{r}(x)} (1+(\delta+|\nabla u|^2)^{\frac{n}{2}})^{\frac{p-n}{n}}(\delta+|\nabla u|^2)^{\frac{n-2}{2}} |\nabla u|^2 = \tfrac{1}{2}\int_{B_{r}(x)} (1+(\delta+|\nabla u|^2)^{\frac{n}{2}})^{\frac{p-n}{n}}(\delta+|\nabla u|^2)^{\frac{n}{2}} \\
&-\delta(1+(\delta+|\nabla u|^2)^{\frac{n}{2}})^{\frac{p-n}{n}}(\delta+|\nabla u|^2)^{\frac{n-2}{2}}=\tfrac{1}{2}\int_{B_{r}(x)} (1+(\delta+|\nabla u|^2)^{\frac{n}{2}})^{\frac{p}{n}}-(1+(\delta+|\nabla u|^2)^{\frac{n}{2}})^{\frac{p-n}{n}} \\
&-\delta(1+(\delta+|\nabla u|^2)^{\frac{n}{2}})^{\frac{p-n}{n}}(\delta+|\nabla u|^2)^{\frac{n-2}{2}}\\
&\ge (\tfrac{1}{2}-\tfrac{p-n}{2 p})\int_{B_{r}(x)} (1+(\delta+|\nabla u|^2)^{\frac{n}{2}})^{\frac{p}{n}}-(1+(\delta+|\nabla u|^2)^{\frac{n}{2}})^{\frac{p-2}{n}}-Cr^n \\
&\ge (\tfrac{1}{2}-\tfrac{p-n}{2 p}-\eta \tfrac{p-2}{2p})\int_{B_{r}(x)} (1+(\delta+|\nabla u|^2)^{\frac{n}{2}})^{\frac{p}{n}} -C_\eta r^n
\ge \tfrac{1}{4} E_{p,\delta}(u,B_{r}(x)) -C r^n.
\end{align*}
In the last step we chose $P_0$ close enough to $n$, for example such that $\tfrac{p-n}{p}\le\tfrac{1}{8}$, and the weight $\eta$ small enough. Summarising, we have obtained
\begin{equation*}
E_{p,\delta}(u,B_{r}(x))\le C E_{p,\delta}(u;A(r,2r))+ C \e_0 E_{p,\delta}(u;B_{2r}(x))+ C r^n.
\end{equation*}
A standard hole-filling technique leads us to (for $\e_0$ small enough)
\begin{equation*}
E_{p,\delta}(u,B_{r}(x))\le \theta E_{p,\delta}(u;B_{2r}(x)) + C r^n,
\end{equation*}
for some uniform constant $\theta \in (0,1)$, and iterating we arrive to
\begin{equation*}
r^{p-n-p \alpha_0} E_{p,\delta}(u,B_{r}(x)) \le (2R_0)^{p-n-p \alpha_0} E_{p,\delta}(u,B_{2 R_0}(x))+C (2R_0)^{p-p \alpha_0},
\end{equation*}
for a uniform exponent $\alpha_0 \in (0,1)$, which corresponds to the uniform H\"older regularity $u_{p,\delta} \in C^{0,\alpha_0}(B_{R_0})$. Let us explicitly remark that the second term comes from the fact that we avoided an intrinsic approach a la Colombo-Mingione \cite{col0}.

Notice that, even though this norm can be made arbitrary small, a priori we cannot reabsorb the second summand of the right hand side in the first, in order to conclude our estimate in Theorem \ref{th.SUregularity}. However, once unlocked the uniform H\"older regularity on a fixed scale, we can show the uniqueness of the solutions $u_{p,\delta}$ on a uniform scale (see next Subsection), and therefore their minimality, and conclude by appealing to point $(a)$ in Theorem \ref{th.SUregularity}.
\end{proof}

\subsection{Uniqueness of solutions with small range}
In this subsection, we want to show the uniqueness of the solutions to \eqref{eq.ELEpdelta} under fixed boundary condition, assuming their images are small; in particular, any weak solution to this equation coincides with a locally energy minimizing solution (which always exists thanks to Lemma \ref{lemma.dirichlet}) on small enough scales. Later, we will see how this newly gained local minimality will allow us to prove the a-priori (non-uniform) smoothness of weak solutions to \eqref{eq.ELEpdelta} needed to justify the derivation of the equation \eqref{eq.ELEpdeltaSU}.
\begin{proposition}\label{prop.uniqueness1}
There exists a constant $\rho_1=\rho_1(P_0,n,\N)$ such that, if $u, v \in W^{1,p}(B_{R_0}(x_0);\N)$ are two solutions to \eqref{eq.ELEpdelta} under the same boundary condition $u=v$ along $\partial B_{R_0}(x_0)$, whose images satisfy $u(B_{R_0}(x_0)),v(B_{R_0}(x_0)) \subseteq B_{\rho}^\N(y_0)$ for some $y_0 \in \N$ and $\rho \in (0,\rho_1]$, then $u=v$ on $B_{R_0}(x_0)$.
\end{proposition}
\begin{proof}
Set $B_{R_0}(x_0)=:B$. We adapt the proof of the uniqueness for $p$-harmonic maps given by Fardoun-Regbaoui, see Theorem $1.1$ in \cite{far}. By assumptions, for all test functions $\phi \in W^{1,p}_0 \cap \mathbb{L}^\infty$ the following two equations are verified
\begin{equation*}
\int_{B}(1+(\delta+|\nabla u|^2)^{\frac{n}{2}})^{\frac{p-n}{n}}(\delta+|\nabla u|^2)^{\frac{n-2}{2}} \nabla u \nabla \phi=\int_{B} (1+(\delta+|\nabla u|^2)^{\frac{n}{2}})^{\frac{p-n}{n}}(\delta+|\nabla u|^2)^{\frac{n-2}{2}} A_u(\nabla u, \nabla u) \phi,
\end{equation*}
and
\begin{equation*}
\int_{B}(1+(\delta+|\nabla v|^2)^{\frac{n}{2}})^{\frac{p-n}{n}}(\delta+|\nabla v|^2)^{\frac{n-2}{2}} \nabla v \nabla \phi=\int_{B} (1+(\delta+|\nabla v|^2)^{\frac{n}{2}})^{\frac{p-n}{n}}(\delta+|\nabla v|^2)^{\frac{n-2}{2}} A_v(\nabla v, \nabla v) \phi.
\end{equation*}
We take the difference of these equations and choose $\phi=u-v$ as test function to get
\begin{equation}\label{eq.SUdifferenceuv}
\footnotesize \begin{aligned}
&\int_{B} [(1+(\delta+|\nabla u|^2)^{\frac{n}{2}})^{\frac{p-n}{n}}(\delta+|\nabla u|^2)^{\frac{n-2}{2}} \nabla u -(1+(\delta+|\nabla v|^2)^{\frac{n}{2}})^{\frac{p-n}{n}}(\delta+|\nabla v|^2)^{\frac{n-2}{2}} \nabla v](\nabla u-\nabla v)=\\
&\int_{B}[ (1+(\delta+|\nabla u|^2)^{\frac{n}{2}})^{\frac{p-n}{n}}(\delta+|\nabla u|^2)^{\frac{n-2}{2}} A_u(\nabla u, \nabla u) -(1+(\delta+|\nabla v|^2)^{\frac{n}{2}})^{\frac{p-n}{n}}(\delta+|\nabla v|^2)^{\frac{n-2}{2}} A_v(\nabla v, \nabla v)] (u-v).
\end{aligned}
\end{equation}
Inequality \eqref{eq.uniquenesslowerbound} in our technical Lemma \ref{lemma.uniquenessbounds} gives the lower bound on the left hand side of \eqref{eq.SUdifferenceuv}
\begin{align*}
&\int_{B} [(1+(\delta+|\nabla u|^2)^{\frac{n}{2}})^{\frac{p-n}{n}}(\delta+|\nabla u|^2)^{\frac{n-2}{2}} \nabla u -(1+(\delta+|\nabla v|^2)^{\frac{n}{2}})^{\frac{p-n}{n}}(\delta+|\nabla v|^2)^{\frac{n-2}{2}}\nabla v](\nabla u-\nabla v)\ge \\
&\tfrac{1}{2} \int_{B} [(1+(\delta+|\nabla u|^2)^{\frac{n}{2}})^{\frac{p-n}{n}}(\delta+|\nabla u|^2)^{\frac{n-2}{2}} +(1+(\delta+|\nabla v|^2)^{\frac{n}{2}})^{\frac{p-n}{n}}(\delta+|\nabla v|^2)^{\frac{n-2}{2}}]|\nabla u-\nabla v|^2.
\end{align*}
In order to estimate from above the right hand side of \eqref{eq.SUdifferenceuv}, we use the inequality \eqref{eq.secondfundamentalform} from Lemma \ref{lemma.secondfundamentalform} with $U:=(1+(\delta+|\nabla u|^2)^{\frac{n}{2}})^{\frac{p-n}{2 n}}(\delta+|\nabla u|^2)^{\frac{n-2}{4}}\nabla u$ and $V:=(1+(\delta+|\nabla v|^2)^{\frac{n}{2}})^{\frac{p-n}{2 n}}(\delta+|\nabla v|^2)^{\frac{n-2}{4}}\nabla v$
\begin{equation*}
\footnotesize \begin{aligned}
&\int_{B}[ (1+(\delta+|\nabla u|^2)^{\frac{n}{2}})^{\frac{p-n}{n}}(\delta+|\nabla u|^2)^{\frac{n-2}{2}} A_u(\nabla u, \nabla u) -(1+(\delta+|\nabla v|^2)^{\frac{n}{2}})^{\frac{p-n}{n}}(\delta+|\nabla v|^2)^{\frac{n-2}{2}} A_v(\nabla v, \nabla v)] (u-v)\\
&\le C \int_{B}[ (1+(\delta+|\nabla u|^2)^{\frac{n}{2}})^{\frac{p-n}{n}}(\delta+|\nabla u|^2)^{\frac{n-2}{2}} |\nabla u|^2 +(1+(\delta+|\nabla v|^2)^{\frac{n}{2}})^{\frac{p-n}{n}}(\delta+|\nabla v|^2)^{\frac{n-2}{2}} |\nabla v|^2] |u-v|^2\\
&+[ (1+(\delta+|\nabla u|^2)^{\frac{n}{2}})^{\frac{p-n}{2 n}}(\delta+|\nabla u|^2)^{\frac{n-2}{4}} |\nabla u| +(1+(\delta+|\nabla v|^2)^{\frac{n}{2}})^{\frac{p-n}{2 n}}(\delta+|\nabla v|^2)^{\frac{n-2}{4}} |\nabla v|] \cdot \\
&\cdot \Big|(1+(\delta+|\nabla u|^2)^{\frac{n}{2}})^{\frac{p-n}{2 n}}(\delta+|\nabla u|^2)^{\frac{n-2}{4}}\nabla u- (1+(\delta+|\nabla v|^2)^{\frac{n}{2}})^{\frac{p-n}{2 n}}(\delta+|\nabla v|^2)^{\frac{n-2}{4}}\nabla v\Big| |u-v|=:I+II.
\end{aligned}
\end{equation*}
For what regards the term $II$, the second inequality \eqref{eq.uniquenessupperbound} of Lemma \ref{lemma.uniquenessbounds} gives the following estimate
\begin{equation*}
\footnotesize
\begin{aligned}
&II \le C \int_{B} [ (1+(\delta+|\nabla u|^2)^{\frac{n}{2}})^{\frac{p-n}{2 n}}(\delta+|\nabla u|^2)^{\frac{n-2}{4}} |\nabla u| +(1+(\delta+|\nabla v|^2)^{\frac{n}{2}})^{\frac{p-n}{2 n}}(\delta+|\nabla v|^2)^{\frac{n-2}{4}} |\nabla v|] |u-v| \cdot\\
&\cdot [(1+(\delta+|\nabla u|^2)^{\frac{n}{2}})^{\frac{p-n}{2 n}}(\delta+|\nabla u|^2)^{\frac{n-2}{4}}+(1+(\delta+|\nabla v|^2)^{\frac{n}{2}})^{\frac{p-n}{2 n}}(\delta+|\nabla v|^2)^{\frac{n-2}{4}}] |\nabla u-\nabla v| \\
&\le C I+ \frac{1}{4}  \int_{B} [(1+(\delta+|\nabla u|^2)^{\frac{n}{2}})^{\frac{p-n}{n}}(\delta+|\nabla u|^2)^{\frac{n-2}{2}} +(1+(\delta+|\nabla v|^2)^{\frac{n}{2}})^{\frac{p-n}{n}}(\delta+|\nabla v|^2)^{\frac{n-2}{2}}]|\nabla u-\nabla v|^2.
\end{aligned}
\end{equation*}
In the last step we have used a weighted Young $(2,2)$-inequality. The first term $I$ can be estimated through Lemma \ref{lemma.imagepoincare}, so that we have
\begin{equation*}
I \le C \rho^2 \int_{B}[(1+(\delta+|\nabla u|^2)^{\frac{n}{2}})^{\frac{p-n}{n}}(\delta+|\nabla u|^2)^{\frac{n-2}{2}} +(1+(\delta+|\nabla v|^2)^{\frac{n}{2}})^{\frac{p-n}{n}}(\delta+|\nabla v|^2)^{\frac{n-2}{2}}] |\nabla u-\nabla v|^2,
\end{equation*}
therefore, plugging the last two inequalities in the calculation above we see that the right hand side is ultimately bounded by
\begin{equation*}
(C \rho^2+ \frac{1}{4}) \int_{B}[(1+(\delta+|\nabla u|^2)^{\frac{n}{2}})^{\frac{p-n}{n}}(\delta+|\nabla u|^2)^{\frac{n-2}{2}} +(1+(\delta+|\nabla v|^2)^{\frac{n}{2}})^{\frac{p-n}{n}}(\delta+|\nabla v|^2)^{\frac{n-2}{2}}] |\nabla u-\nabla v|^2.
\end{equation*}
Combined with the lower bound for the left hand side, this gives
\begin{equation*}
(\frac{1}{4}-C \rho^2) \int_{B}[(1+(\delta+|\nabla u|^2)^{\frac{n}{2}})^{\frac{p-n}{n}}(\delta+|\nabla u|^2)^{\frac{n-2}{2}} +(1+(\delta+|\nabla v|^2)^{\frac{n}{2}})^{\frac{p-n}{n}}(\delta+|\nabla v|^2)^{\frac{n-2}{2}}] |\nabla u-\nabla v|^2 \le 0,
\end{equation*}
so for $\rho_1=\rho_1(P_0,n,\N)$ small enough we obtain $\nabla u=\nabla v$, and therefore $u=v$ by the common boundary condition as required.
\end{proof}
As a corollary, any solution with small range is a local energy minimizer, so we have a direct analogue of Theorem $1.2$ in \cite{far}.
\begin{corollary}\label{cor.uniqueness1}
There exist (not-uniform) constants $\alpha=\alpha(P_0,p,n,\N), \tau_1=\tau_1(P_0,p,n,\N) \in (0,1)$ and $\rho_2=\rho_2(P_0,n,\N)>0$ such that, if $u \in W^{1,p}(B_{R_0}(x_0);\N)$ is a solution to \eqref{eq.ELEpdelta} satisfying $u(B_{R_0}(x_0)) \subseteq B_{\rho_2}^\N(y_0)$ for some $y_0 \in \N$, then $u \in C^{1,\alpha}_{loc}(B_{R_0}(x_0);\N)$. Moreover, $u$ is minimizing the functional $E_{p,\delta}(\cdot;B_{\tau_1 R_0}(y))$ for all $B_{\tau_1 R_0}(y) \subset B_{R_0}(x_0)$ amongst all functions $v \in W^{1,p}(B_{\tau_1 R_0}(y);\N)$ such that $v=u$ along $\partial B_{\tau_1 R_0}(y)$ and such that $v(B_{\tau_1 R_0}(y)) \subseteq B_{\rho_2}^\N(P_1)$.
\end{corollary}
\begin{proof}
Firstly, we prove the second assertion, choosing $\rho_2<\min \set{\rho_1,\rho_0}$, an arbitrary $B:=B_{\tau_1 R_0}(y) \subset B_{R_0}(x_0)$, $\tau_1$ to be fixed later, and apply Proposition \ref{prop.dirichlet} above with radius $\rho_2$ to get a solution $\Tilde{u}$ of \eqref{eq.ELEpdelta}, which minimizes $E_{p,\delta}(\cdot;B)$ amongst all functions $v \in W^{1,p}(B;\N)$ such that $v=u$ along $\partial B$ and such that $v(B) \subseteq B_{\rho_2}^\N(y_0)$. Proposition \ref{prop.uniqueness1} above then implies that $u=\Tilde{u}$. It remains to show that $u \in C^{1,\alpha}_{loc}(B;\N)$. Notice that, since $p>n$, we can choose $\tau_1$ small enough so that Sobolev's embedding implies that for all $v \in W^{1,p}(B;\N)$ with $v=u$ along $\partial B$ and $\norm{\nabla u-\nabla v}_{p}$ small enough, we also have $v(B)\subseteq B_{\rho_2}^\N(y_0)$, and so $E_{p,\delta}(u;B) \le E_{p,\delta}(v;B)$ by the construction above. This means that $u$ is a local minimizer of $E_{p,\delta}$ (without any range restriction), and this is enough to deduce its regularity by what we proved in the Subsection \ref{subsec.minimizers}.
\end{proof}
Finally, we can deduce the (non-uniform) local regularity of arbitrary solutions to \eqref{eq.ELEpdelta} thanks to Sobolev's embedding.
\begin{corollary}\label{cor.uniqueness2}
Any solution $u_{p,\delta} \in W^{1,p}(B_{R_0}(x_0);\N)$ to \eqref{eq.ELEpdelta} is of class $C^{\infty}_{loc}(B_{R_0}(x_0);\N)$ and a local minimizer of $E_{p,\delta}$.
\end{corollary}
\begin{proof}
For any point $x \in B_{R_0}(x_0)$, there exists a small enough constant $\tau_1$ such that the ball $B:=B_{\tau_1 R_0}(x) \subset B_{R_0}(x_0)$ and $(\tau_1 R_0)^{p-n}E_{p,\delta}(u_{p,\delta};B)<\e_3$, where $\e_3$ is as small as we wish. By Sobolev's embedding, the image $u_{p,\delta}(B_{\tau_1 R_0}(x))$ is contained in $B^N_{\rho_2}(u(x))$ for small enough $\e_3$ (this $\tau_1$ degenerates as $p \searrow n$). In particular, $u_{p,\delta}$ is locally minimizing and $C^{1,\alpha_1}_{loc}$ around $x$. Hence, the system is non-degenerate and uniformly elliptic (not uniformly in $\delta$), with $C^{0,\alpha_1}$-coefficients, so we can appeal to standard Schauder's theory to bootstrap this regularity to a smooth one.
\end{proof}
As a Corollary of the uniqueness, solutions with uniformly small enough energies are trivial.
\begin{corollary}\label{cor.trivial}
Suppose that $E_{p,\delta}(u_{p,\delta};\M) <\e_0$, for $\e_0$ as in \ref{th.SUregularity}, and that Theorem \ref{th.SUregularity} applies to them. Then $u_{p,\delta}\equiv y_{p,\delta}$ for points $y_{p,\delta}\in \N$.
\end{corollary}
\begin{proof}
By Theorem \ref{th.SUregularity} and Corollary \ref{cor.uniqueness2}, we know that $u:=u_{p,\delta}$ are locally minimizing in $\M$. A covering argument implies a uniform oscillation inequality (or H\"older continuity)
\begin{equation*}
\norm{u-[u]_{\M} }_{\infty} \le C E_{p,\delta}(u;\M)^{\frac{1}{p}} \le C \e_0.
\end{equation*}
Testing the system \eqref{eq.ELEpdelta} with $u-[u]_{\M}$ and combining \eqref{eq.elementary1} with the bound above we get
\begin{align*}
& \int_{\M} (1+(\delta+|\nabla u|^2)^{\frac{n}{2}})^{\frac{p-n}{n}}(\delta+|\nabla u|^2)^{\frac{n-2}{2}} |\nabla u |^2 \le \int_{\M} (1+(\delta+|\nabla u|^2)^{\frac{n}{2}})^{\frac{p-n}{n}}(\delta+|\nabla u|^2)^{\frac{n-2}{2}} A_{u}(\nabla u,\nabla u)\cdot\\
&\cdot (u-[u]_{\M}) \le C \e_0 \int_{\M} (1+(\delta+|\nabla u|^2)^{\frac{n}{2}})^{\frac{p-n}{n}}(\delta+|\nabla u|^2)^{\frac{n-2}{2}} |\nabla u |^2,
\end{align*}
from which we conclude by choosing $\e_0$ small enough (uniformly).
\end{proof}

\subsection{Three-dimensional critical maps}
In this subsection we treat the case of arbitrary critical maps to the functionals $E_{p,\delta}$ in case the domain has dimension $n=3$. Recall that, thanks to our Corollary \ref{cor.uniqueness2}, we are allowed to recast equation \eqref{eq.ELEpdelta} as in \eqref{eq.ELEpdeltaSU}, which we recall
\begin{equation*}
-\mathcal{L}_{u} [u]=(p-3) \tfrac{(\delta+|\nabla u|^2)^{\frac{1}{2}}\nabla^2 u (\nabla u,\nabla u)}{1+(\delta+|\nabla u|^2)^{\frac{3}{2}} } +A_u(\nabla u, \nabla u).
\end{equation*}
Here we have set 
\begin{equation*}
    (\mathcal{L}_{u} [v])^\alpha:=(\delta_{i j} \delta^{\alpha \beta}+ \tfrac{\nabla_i u^\alpha \nabla_j u^\beta}{\delta+|\nabla u|^2}) \nabla^2_{i j} v^\beta =:L^{\alpha \beta}_{i j}(u) \nabla^2_{i j} v^\beta,
\end{equation*}
and by the inequality \eqref{eq.pCordes} we know that this operator satisfies Cordes' condition. We will prove a result stronger than Theorem \ref{th.SUregularity}, point $(c)$, more resemblant of Proposition $3.1$ in \cite{sac}.
\begin{theorem}\label{th.regularityD3}
Suppose $n=3$. There exist constants $\alpha_2$, $\e_1>0$ and a threshold $P_1>3$, all depending only on the data $n, N, (\M,g)$ and $(\N,h)$, such that if $u$ is a solution to \eqref{eq.ELEpdeltaSU} for some $p \in (3,P_1)$ and $\delta \in (0,1)$, with $D_3(u,B_{R_0}(x_0))<\e_1^3$, then $u \in W^{2,q}(B_{\tau_0 R_0}(x_0))$ for all $q \in [2,q_1)$, where $q_1$ is as in Corollary \ref{cor.cordesbound}. Furthermore, $u \in C_{loc}^{1,\alpha_2}(B_{R_0}(x_0);\N)$ and is a local minimizer of $E_{p,\delta}$.
\end{theorem}
\begin{proof}
Fix an arbitrary ball $B_{2r}(x) \subset B_{R_0}(x_0)$. Take the $\R^N$-norm of the system above and use some elementary bounds to get
\begin{equation*}
|\mathcal{L}_{u}[u]| \le (p-3) |\nabla^2 u|+c(\mathcal{\N})|\nabla u|^2,
\end{equation*}
so taking the $\mathbb{L}^q(B_r(x))$-norm we get thanks to our Corollary \ref{cor.cordesbound}
\begin{equation*}
\begin{aligned}
&c r^{2-\frac{3}{q}} \norm{\nabla^2 u}_{q,B_r(x) } - c \norm{\nabla u}_{3, B_{2r}(x)} \le r^{2-\frac{3}{q}} \norm{\mathcal{L}_u[u]}_{q,B_r(x)} \\
&\le (p-3) r^{2-\frac{3}{q}} \norm{\nabla^2 u}_{q,B_r(x)}+c(\mathcal{\N}) r^{2-\frac{3}{q}} \norm{|\nabla 
 u|^2}_{q,B_r(x)}.
 \end{aligned}
\end{equation*}
Choosing $P_1$ close enough to $3$, we can absorb the first summand of the right hand side, and then use the inequality \eqref{eq.gagliardo}
\begin{equation}
 (c-P_1+3) r^{2-\frac{3}{q}} \norm{\nabla^2 u}_{q,B_r(x) } \le c \norm{\nabla u}_{3, B_{2r}(x) }+c(n,\N,q,\mathcal{\N}) \norm{\nabla u}_{3, B_{r}(x) } [\norm{\nabla u}_{3, B_{r}(x) }+ r^{2-\frac{3}{q}} \norm{\nabla^2 u}_{q,B_r(x)}].
\end{equation}
Choosing $\e_0$ small enough we can rearrange this inequality as
\begin{equation}
r^{2-\frac{3}{q}} \norm{\nabla^2 u}_{q,B_r(x)} \le \tfrac{c+c\e_0}{c-P_1+3-c \e_0} \norm{\nabla u}_{3, B_{2r}(x)}.
\end{equation}
Sobolev's embedding gives the uniform H\"older continuity of the solutions relative to some exponent $\alpha_2>\tfrac{1}{2}$, which we combine with the uniqueness Corollary \ref{cor.uniqueness1}, to deduce their minimality, and finally with Theorem \ref{th.SUregularity}, point $(a)$, concluding the proof.
\end{proof}
\begin{corollary}[Theorem \ref{th.SUregularity}, point $(c)$]
Suppose $n=3$. There exist constants $\alpha_0, \ \e_0>0$ and $P_0>3$, depending only on the data $n, N, (\M,g)$ and $(\N,h)$, such that if $u_{p,\delta}$ is a solution to \eqref{eq.ELEpdeltaSU} for some $p \in (3,P_0)$ and $\delta \in (0,1)$, with $R_0^{p-n} E_{p,\delta}(u,B_{R_0}(x_0))<\e_0$, then $u \in C^{1,\alpha_0}_{loc}(B_{R_0}(x_0);\N)$ and is a local minimizer of $E_{p,\delta}$.
\end{corollary}

\section{Quantization of the energy}\label{sec.quantization}
In this section, we aim to prove Theorem \ref{th.quantization}. The argument follows partially the one by the second author in \cite{lam1}, however we need to adapt it to the weaker small energy regularity Theorem \ref{th.SUregularity}, leading to non-trivial modifications.
\subsection{Bubbles construction and concentration radii bound}
Let $\e_0$ be given by Theorem \ref{th.SUregularity} and consider throughout this entire Section, a family of solutions of \eqref{eq.ELEpdelta}, $(u_k:=u_{p_k,\delta_k})_{k \in \mathbb{\N}}$, to which Theorem \ref{th.SUregularity} applies, and where $p_k\in (n,P), \delta_k \in (0,1)$ and $p_k \searrow n$ and $\delta_k \searrow 0$ (relative to scales $s_k \in [0,1]$ possibly dependent on $k$), with uniformly bounded energies $E_{p_k,\delta_k}^{(s_k)}(u_k;\M) \le \Lambda_0$. We can extract a subsequence (not relabeled) such that $u_k$ converges weakly in $W^{1,n}(\M;\R^N)$ to a map $u$ thanks to \eqref{eq.elementary1}. If on a ball $B_r(x) \subseteq \M$ we have $E_{p_k,\delta_k}^{(s_k)}(u_k;B_r(x)) \le \e_0^{p_k}$ for all $k$ large enough, then Theorem \ref{th.SUregularity} guarantees a uniform $C^{1,\alpha_0}$-bound on the $u_k$'s on a slightly smaller ball $B_{\tau_0 r}(x)$, and thus by Arzel\`a-Ascoli's Theorem for any $\alpha_1<\alpha_0$ we can possibly extract a further subsequence converging in $C^{1,\alpha_1}(B_{\tau_0 r}(x))$ to $u$. For this reason, it is convenient to define the energy concentration set $\Sigma$ as
\begin{equation*}
\Sigma:= \set{x \in \M \mid \limsup_{k\rightarrow +\infty} E_{p_k,\delta_k}^{(s_k)}(u_k,B_r(x))^{\frac{1}{p_k}} \ge \e_0 \ \forall r>0}.
\end{equation*}
From the uniform boundedness of the $E_{p_k,\delta_k}$-energies of the sequence $(u_k)$, we see that $\Sigma$ consists of finitely many points $x_1$,...,$x_K$, where $K=K(\e_0,\Lambda_0,n)$. Indeed, given $K$ distinct points $x_1$,..., $x_K$ in $\Sigma$, we can find a radius $R$ smaller than $R_0:=\tfrac{1}{2} \min \set{d(x_i,x_j)}$ so that the balls $B_R(x_i)$ are mutually disjoint. Therefore we obtain by the definition of $\Sigma$ (along a sequence realising the $\limsup$)
\begin{align*}
\e_0^{p_k} K &\le \sum_{i=1}^K E^{(s)}_{p_k,\delta_k}(u_k,B_R(x_i)) \le E^{(s)}_{p_k,\delta_k}(u_k,\M) \le \Lambda_0,
\end{align*}
so that $K=K(\e_0,\Lambda_0,n)$ as claimed.

By the discussion above, outside the set $\Sigma$, $u_k$ converge ``smoothly" to a weakly $n$-harmonic map $u \in C^{1,\alpha_1}(\M\setminus \set{x_1,...,x_K};\mathcal{N})$ with finite $D_n$-energy. By Duzaar-Fuchs' removability of singularity Theorem in \cite{duz0}, we see that $u \in C^{1,\alpha_1}(\M;\mathcal{N})$. In order to construct the bubbles, along with some energy control, we adopt the classical method of maximal concentration function. Recall the definition of $R_0$ from above. Introduce for all $i=1,...,K$ and $k \in \mathbb{N}$ the maximal concentration functions $F^i_k(R):= \max_{y \in B_{R_0}(x_i)} E^{(s_k)}_{p_k,\delta_k}(u_k,B_R(y))$ for $R \in (0,R_0)$. By definition of $\Sigma$ we know that for all $R>0$, $F^i_k(R) \ge E^{(s_k)}_{p_k,\delta_k}(u_k,B_R(x_i)) \ge \e_0^{p_k} /2$ for all $k$ (after possibly extract a subsequence). On the other hand, for all $i=1,...,K$ and for all $k \in \mathbb{N}$ fixed, $F^i_k(R) \rightarrow 0$ as $R \rightarrow 0$, so there exists $x^i_k(R) \in B_{R_0}(x_i)$ such that $F^i_k(R)=E^{(s_k)}_{p_k,\delta_k}(u_k,B_R(x^i_k(R)))$ and, by diagonal argument, we can also ensure $x^i_k \rightarrow x^i$ and find radii $r^i_k \rightarrow 0$ (without loss of generality $r_k^i \le 1$) such that
\begin{equation}\label{eq.maxpropertybubble}
\max_{y \in B_{R_0}(x_i)} E^{(s_k)}_{p_k,\delta_k}(u_k,B_{r^i_k}(y))=E^{(s_k)}_{p_k,\delta_k}(u_k,B_{r^i_k}(x^i_k))=\tfrac{\e_0^{p_k}}{2}.
\end{equation}
We fix our attention on $x_1$. By possibly making $R_0$ smaller, we will consider the maps $u_k$ as defined on an Euclidean ball $B_{R_0}(0)$ thanks to the exponential map of $\M$ centered at $x_1$ and ignore the lower order terms coming from this identification; notice that we then have $x_k^1 \rightarrow 0$. Drop the upper index $\cdot^1$ to simplify the notation. We define rescaled maps $v_k:B_{r_k^{-1} R_0} \rightarrow \N$ as $v_k(x):=u_k( x_k+ r_k x)$, and notice that for any $k$ large enough, $v_k$ is a weak solution of the rescaled system \eqref{eq.ELEpdelta} with $s_k':=r_k^n s_k \in (0,1)$ and $\delta'_k:= r_k^2 \delta_k\in [0,1]$
\begin{equation}
\footnotesize
-\Div[(s_k'+(\delta'_k+|\nabla v_k|^2)^{\frac{n}{2}})^{\frac{p-n}{n}}(\delta'_k+|\nabla v_k|^2)^{\frac{n-2}{2}} \nabla v_k]=(s_k'+(\delta'_k+|\nabla v_k|^2)^{\frac{n}{2}})^{\frac{p-n}{n}}(\delta'_k+|\nabla v_k|^2)^{\frac{n-2}{2}} A_{v_k}(\nabla v_k, \nabla v_k).
\end{equation}
Moreover, equation \eqref{eq.maxpropertybubble} implies
\begin{equation}\label{eq.rescaledmaxpropertybubble}
\max_{y \in B_{r_k^{-1} R_0}(0)} r_k^{p_k-n} E^{(s'_k)}_{p_k,\delta'_k}(v_k,B_1(y))=r_k^{p_k-n} E^{(s'_k)}_{p_k,\delta'_k}(v_k,B_1(0))=\tfrac{\e_0^{p_k}}{2}.
\end{equation}
The same rescaling shows that $v_k$ have uniformly bounded $E^{(s'_k)}_{p_k,\delta'_k}$-energy for $k$ large enough since
\begin{equation}\label{eq.uniformboundrescale}
\begin{aligned}
    E^{(s'_k)}_{p_k,\delta'_k}(v_k;B_{R_0 r_k^{-1}}(0)) &= \tfrac{1}{p_k} \int_{ B_{R_0 r_k^{-1}}(0) } (s'_k+(\delta'_k+|\nabla v_k|^2)^{\frac{n}{2}})^{\frac{p_k}{n}}- (s'_k+\delta_k'^{ \frac{n}{2}})^{\frac{p_k}{n}} \\
    &\le r_k^{p_k-n} \tfrac{1}{p_k} \int_M (s+(\delta_k+|\nabla u_k|^2)^{\frac{n}{2}})^{\frac{p_k}{n}}- (s+\delta_k^{\frac{n}{2}})^{\frac{p_k}{n}} \le E^{(s)}_{p_k,\delta_k}(u_k;\M) \le \Lambda_0.
    \end{aligned}
\end{equation}
Finally, combining the inequalities \eqref{eq.elementary1} and \eqref{eq.rescaledmaxpropertybubble} we deduce the uniform smallness of the rescaled $E^{(s'_k)}_{p_k,\delta'_k}$-energies of the $v_k$'s on balls of radius $1$ centered at points in bigger and bigger domains. In particular, we can apply Theorem \ref{th.SUregularity} on any ball $B_1(y)$ with $y \in B_{r_k^{-1} R_0}(0)$, obtaining uniform $C^{1,\alpha_0}$-bound on larger and larger balls. Arzel\'a-Ascoli's theorem ensures $v_k \rightarrow \omega$ in $C^{1,\alpha_1}_{loc}(\R^n;\mathcal{N})$ for some $n$-harmonic map $\omega \in C^{1,\alpha_1}(\R^n;\mathcal{N})$ (hence locally minimizing by the uniqueness result of Fardoun-Regaboui \cite{far}). By conformal invariance of the $n$-harmonic map equation, and Duzaar-Fuchs' removability of singularity Theorem \cite{duz0}, we can extend $\omega$ to a map from $S^n$ into $\mathcal{N}$, which we will call a \emph{bubble} at the point $x^1$. We can repeat the argument to the rescaled functions $v_k$. Iterating this procedure we get a scheme of convergence to an object called \emph{bubble tree} (we use the terminology from \cite{par}). Since any bubble (at any step) must "carry" a non-trivial amount of energy $E^{(s'_k)}_{p_k,\delta'_k}$ at least $\tfrac{\e_0^{p_k}}{2}$, and the energies are uniformly bounded (at any step), the procedure ends after a finite amount of time repeating the argument given above for the cardinality of $\Sigma$.

In what follows, we want to find a precise asymptotic of the concentration radii with respect to the parameters $p_k$. We split the argument in two lemmas, inspired respectively by Li-Zhu \cite{liz} and by Lamm \cite{lam1}.

First of all we need to refine a bit the setting in order to get better bounds: as we will see in the proof of Theorem \ref{th.quantization}, from the classical induction argument of Ding and Tian \cite{din}, it will suffice to prove the energy identity in the presence of one single bubble. Therefore, thanks to the local nature of the problem, we can focus on a family of solutions $u_k:B_1\rightarrow \N$ with only one concentration point $x_1=0$. We denote by $\omega:S^n \rightarrow \N$ the bubble (sometimes identified with its composition with the stereographic projection $\omega:\R^n\rightarrow \N$), by $r_k$ the concentration radii associated to the bubble, without loss of generality $r_k \in (0,1]$, and by $u_\infty \in C^{1,\alpha}(B_1;\N)$ the $n$-harmonic map arising as a limit in $C^{1,\alpha}(B_1 \setminus \set{0};\N)$ of the family $u_k$. From the $C^{1,\alpha}$-convergence we see that for every $R_0 \in (0,1)$, we have
\begin{equation*}
\lim_{k \rightarrow +\infty} E_{p_k,\delta_k}(u_k,B_1 \setminus B_{R_0})=D_{n}(u_\infty,B_1 \setminus B_{R_0}).
\end{equation*}
By the definition of the bubble $\omega$, the rescalings $v_k(y):=u_k(r_k y)$ converge in $C^{1,\alpha}$ to $\omega$, so that for every $R>0$
\begin{equation*}
\lim_{k \rightarrow +\infty} E_{p_k,\delta_k}(u_k,B_{R r_k})=D_{n}(\omega,B_{R}).
\end{equation*}
Furthermore, for any fixed $M>1$, the same reasoning leads to the identities above guarantees
\begin{equation*}
\lim_{k \rightarrow +\infty} E_{p_k,\delta_k}(u_k,B_{R_0} \setminus B_{\frac{R_0}{M}})+E_{p_k,\delta_k}(u_k,B_{M R r_k} \setminus B_{R r_k}) \longrightarrow 0,
\end{equation*}
as $R_0 \rightarrow 0$ and $R \rightarrow +\infty$. Therefore, the proof of the energy identity will be reduced to show
\begin{equation}
\lim_{R \rightarrow +\infty} \lim_{R_0 \rightarrow 0} \lim_{k \rightarrow +\infty} E_{p_k,\delta_k}(u_k,A(r_k R,R_0))=0.
\end{equation}
Notice that the first limit to take is as $k \rightarrow +\infty$. As we will see, we have for $k$ large enough, under these hypotheses, the uniform bound
\begin{equation}\label{eq.gradientdecay}
|\nabla u_k|<\tfrac{C_{10} \e_0}{|x|}, \quad x \in A(r_k R,R_0).
\end{equation}
\begin{lemma}
In the situation above, assuming the bound \eqref{eq.gradientdecay}, the concentration radii verify
\begin{equation}\label{eq.concentrationradii}
1 \le \limsup_{k \rightarrow \infty} (r_k)^{n-p_k} \le C_{11}(n,\e_0,\Lambda_0) < \infty.
\end{equation}
\end{lemma}
\begin{proof}
From the regularity of the bubble $\omega$ and compactness of $S^n$, we have $\norm{\nabla \omega}_{\infty} \le C$, so by the smooth convergence of the rescalings $v_k$ we have for $k$ large enough
\begin{equation*}
|\nabla v_k| \le C, \ \text{on } B_{R} \ \Rightarrow |\nabla u_k| \le \tfrac{C}{r_k}, \ \text{on } B_{R r_k}.
\end{equation*}
Arguing similarly, the smoothness of $u_{\infty}$ implies (for $R_0$ small enough depending only on $u_\infty$, and for $k$ large enough)
\begin{equation*}
|\nabla u_{\infty}| \le \tfrac{C}{R_0}, \ \text{on } A(R_0,1) \ \Rightarrow |\nabla u_k| \le \tfrac{C}{R_0}, \ \text{on } A(R_0,1).
\end{equation*}
Therefore, combining with \eqref{eq.gradientdecay}, we deduce $\norm{\nabla u_k}_{\infty,B_1} \le C r_k^{-1}$ for $k$ large enough depending only on $R$ and $R_0$, so $\nabla v_k$ is uniformly locally bounded by the constant $C$. Lebesgue's dominated convergence theorem imply, together with the regularity of $\omega$ and of the convergence, that
\begin{align*}
\liminf_{R \rightarrow \infty} \liminf_{k \rightarrow \infty} \int_{B_R} |\nabla v_k|^{p_k}= \liminf_{R \rightarrow \infty} \int_{B_R} \lim_{k \rightarrow \infty} |\nabla v_k|^{p_k}= \lim_{R \rightarrow \infty} \int_{B_R} |\nabla \omega|^{n}=\int_{\R^n} |\nabla \omega|^{n} \ge \tfrac{1}{2}\e_0^n.
\end{align*}
However, for any $R>0$, we can use the energy bound and \eqref{eq.elementary1} to get
\begin{align*}
\limsup_{R \rightarrow \infty} \limsup_{k \rightarrow \infty} \int_{B_R} |\nabla v_k|^{p_k}=\limsup_{R \rightarrow \infty} \limsup_{k \rightarrow \infty} r_k^{p_k-n} \int_{B_{R r_k}} |\nabla u_k|^{p_k} \le \Lambda_0\limsup_{k \rightarrow \infty} r_k^{p_k-n}.
\end{align*}
Hence, we conclude
\begin{equation*}
    \limsup_{k \rightarrow \infty} r_k^{n-p_k} \le 2 \Lambda_0 \e_0^{-n}<\infty. \qedhere
\end{equation*}
\end{proof}
\begin{lemma}
In the situation above, assume the bound \eqref{eq.gradientdecay} and that the maps $(u_k)$ verify the assumption \eqref{eq.entropy}, which we recall here for convenience
\begin{equation}\label{eq.entropycondition}
\lim_{k \rightarrow \infty} (p_k-n)\int_M (1+(\delta_k+|\nabla u_k|^2)^{\tfrac{n}{2}})^{\frac{p_k}{n}}\log(1+(\delta_k+|\nabla u_k|^2)^{\tfrac{n}{2}}) =0.
\end{equation}
Then the concentration radii verify
\begin{equation}\label{eq.concentrationradiientropy}
\lim_{k \rightarrow \infty} (r_k)^{n-p_k} =1.
\end{equation}
\end{lemma}
\begin{proof}
For every $k \in \mathbb{N}$ we define
\begin{equation*}
\Omega_k := \set{x \in B_{r_k}(x_k) \mid |\nabla u_k(x)| \ge \e_0 (4 C_3 C_{11} \omega_n)^{\frac{-1}{p_k}} r_k^{-1}}.
\end{equation*}
Here $C_3$ is the constant in \eqref{eq.elementary1} and $C_{11}$ is the constant in \eqref{eq.concentrationradii}.
We claim that $Vol(\Omega_k) \ge c r_k^{p_k}$ for some constant $c>0$ and for every $k$. Otherwise, there exists a subsequence $k_m$ such that for every $m \in \mathbb{\N}$ we have
\begin{equation*}
Vol(\Omega_{k_m}) \le \tfrac{c}{m} r_{k_m}^{p_{k_m}}.
\end{equation*}
From Theorem \ref{th.SUregularity} and equation \eqref{eq.maxpropertybubble}, we deduce the uniform gradient bound
\begin{equation*}
\norm{\nabla u_{k_m}}_{\mathbb{L}^\infty(B_{r_{k_m}}(x_k))} \le \frac{C \e_0}{r_{k_m}}.
\end{equation*}
Notice that we used the full power of \eqref{eq.maxpropertybubble} to get the bound above on the full ball $B_{r_{k_m}}(x_{k_m})$.
From the definition of $\Omega_k$, we can refine this estimate for points $x \in B_{r_{k_m}}(x_{k_m}) \setminus \Omega_{k_m}$ as follows
\begin{equation*}
|\nabla u_{k_m}|(x) \le \e_0 (4 C_3 C_{11} \omega_n)^{\frac{-1}{p_k}} r_{k_m}^{-1}.
\end{equation*}
By construction of $r_k$, $x_k$ and $\Omega_{k}$, and applying the inequalities obtained above, as well as inequality \eqref{eq.elementary1}, we obtain
\begin{equation*}
\begin{aligned}
&\tfrac{\e_0^{p_{k_m}}}{2}=E_{p_{k_m},\delta_{k_m}}(u_{k_m},B_{r_{k_m}}(x_{k_m})) = \int_{\Omega_{k_m}} (1+(\delta_{k_m}+|\nabla u_{k_m}|^2)^{\frac{n}{2}})^{\frac{p_{k_m}}{n}}-(1+\delta_{k_m}^{\frac{n}{2}})^{\frac{p_{k_m}}{n}} +\\
&+\int_{B_{r_{k_m}}(x_k) \setminus \Omega_{k_m}} (1+(\delta_{k_m}+|\nabla u_{k_m}|^2)^{\frac{n}{2}})^{\frac{p_{k_m}}{n}}-(1+\delta_{k_m}^{\frac{n}{2}})^{\frac{p_{k_m}}{n}}\\
&\le \tfrac{c}{m} r_{k_m}^{p_{k_m}} C_3 (1+C\e_0^{p_{k_m}} r_{k_m}^{-p_{k_m}} )+ \omega_n r_{k_m}^n C_3 (1+\e_0^{p_{k_m}}  (4 C_3 C_{11} \omega_n r_{k_m}^{p_{k_m}})^{-1} )\\
&\le (\tfrac{1}{4 C_{11}}r_{k_m}^{n-p_{k_m}}+ \tfrac{c}{m})\e_0^{p_{k_m}} +\mathfrak{o}_m(1) \le \tfrac{1}{3} \e_0^{p_{k_m}}, \ \text{for large enough } m.
\end{aligned}
\end{equation*}
This gives the desired contradiction, proving finally the claim. We are now ready to prove \eqref{eq.concentrationradiientropy}, appealing to the volume bound just proven, the entropy assumption \eqref{eq.entropycondition}, inequality \eqref{eq.concentrationradii} and the monotonicity of the logarithm and of the function $y \mapsto y^{\frac{p_k}{n}}\log(y)$ (in what follows, the constant $c$ varies from line to line but always remains positive)
\begin{align*}
0&=\lim_{k\rightarrow \infty} (p_k-n)\int_M (1+(\delta_k+|\nabla u_k|^2)^{\tfrac{n}{2}})^{\frac{p_k}{n}}\log(1+(\delta_k+|\nabla u_k|^2)^{\tfrac{n}{2}})  \\
&\ge \limsup_{k \rightarrow \infty} (p_k-n)\int_{\Omega_k} |\nabla u_k|^{p_k}\log(|\nabla u_k|^n)\\
&\ge \limsup_{k\rightarrow \infty} n (p_k-n) c r_k^{p_k} \e_0^{p_k} (4 C_3 C_{11} \omega_n)^{-1} r_{k}^{-p_k} \log(\e_0 (4 C_3 C_{11} \omega_n)^{\frac{-1}{p_k}} r_{k}^{-1} )\\
&=\limsup_{k\rightarrow \infty} c (n-p_k) (\log(r_k)-\log(c))=\limsup_{k\rightarrow \infty} c \log(r_k^{n-p_k}) \ge 0,
\end{align*}
from which we can easily deduce the thesis.
\end{proof}
\begin{remark}
We remark explicitly that, after this Lemma, the entropy bound passes naturally to the rescalings $v_k$ introduced above, and we can iterate the procedure since the proof works identically.
\end{remark}
\subsection{Hopf differential type estimate}
In the following Lemma, we show that the rotational component of the gradient contributes more than the radial one to the $E_{p,\delta}$-energy, up to an infinitesimal term as $p$ goes to $n$. The proof follows a similar argument to the one in Lemma $2.4$ in \cite{lam1}.
\begin{lemma}\label{lemma.hopf1}
In the situation above, we have for every $r \in (0,1)$, every $p \in (n,P_0)$ and $\delta \in [0,1]$, the following inequality
\begin{equation}\label{eq.hopf1}
\begin{aligned}
&\int_{\partial B_r} (1+(\delta+|\nabla u|^2)^{\frac{n}{2}})^{\frac{p-n}{n}}(\delta+|\nabla u|^2)^{\frac{n-2}{2}} |\partial_r u_p|^2\le \tfrac{p-n}{(p-1) r}\int_{B_r} (1+(\delta+|\nabla u|^2)^{\frac{n}{2}})^{\frac{p}{n}}\\
&+\tfrac{1}{p-1}\int_{\partial B_r} (1+(\delta+|\nabla u|^2)^{\frac{n}{2}})^{\frac{p-n}{n}} [1+(\delta+|\nabla u|^2)^{\frac{n-2}{2}} (\delta+ r^{-2} |\nabla_S u_p|^2)].
\end{aligned}
\end{equation}
\end{lemma}
\begin{proof}
Test the system \eqref{eq.ELEpdelta} with $\phi^\alpha:= x \nabla u^\alpha$: by $C^{1,\alpha_0}$-regularity of $u_{p,\delta}$ we are allowed to do this, since the integrand is of the form $\mathbb{L}^q \cdot \mathbb{L}^\infty$, for some $q>1$, and a cut-off and limit procedure allows to consider non-trivial boundary values in the test function. Using also the orthogonality between the second fundamental form and the tangent bundle $T\N$ we get
\begin{equation*}
\footnotesize
\begin{aligned}
0&=\int_{B_r} -\Div[(1+(\delta+|\nabla u|^2)^{\frac{n}{2}})^{\frac{p-n}{n}}(\delta+|\nabla u|^2)^{\frac{n-2}{2}} \nabla u^\alpha] x \nabla u^\alpha=\int_{B_r} (1+(\delta+|\nabla u|^2)^{\frac{n}{2}})^{\frac{p-n}{n}}(\delta+|\nabla u|^2)^{\frac{n-2}{2}}|\nabla u|^2\\
&-\int_{\partial B_r} r (1+(\delta+|\nabla u|^2)^{\frac{n}{2}})^{\frac{p-n}{n}}(\delta+|\nabla u|^2)^{\frac{n-2}{2}} |\partial_r u_p|^2+\tfrac{1}{n}\int_{B_r} (1+(\delta+|\nabla u|^2)^{\frac{n}{2}})^{\frac{p-n}{n}} x \nabla[1+(\delta+|\nabla u|^2)^{\frac{n}{2}}].
\end{aligned}
\end{equation*}
In order to treat the last term, we integrate by parts, obtaining
\begin{equation*}
\begin{aligned}
&\tfrac{1}{n}\int_{B_r} (1+(\delta+|\nabla u|^2)^{\frac{n}{2}})^{\frac{p-n}{n}} x \nabla[1+(\delta+|\nabla u|^2)^{\frac{n}{2}}]=\tfrac{1}{p}\int_{B_r} \nabla[(1+(\delta+|\nabla u|^2)^{\frac{n}{2}})^{\frac{p}{n}}] x\\
&=\tfrac{1}{p}\int_{\partial B_r} r (1+(\delta+|\nabla u|^2)^{\frac{n}{2}})^{\frac{p}{n}}-\tfrac{n}{p}\int_{B_r} (1+(\delta+|\nabla u|^2)^{\frac{n}{2}})^{\frac{p}{n}}.
\end{aligned}
\end{equation*}
Combining the inequalities above and using the decomposition $|\nabla u|^2=|\partial_r u_p|^2+ r^{-2}|\nabla_S u_p|^2$, we deduce
\begin{equation*}
\begin{aligned}
&\int_{\partial B_r} r (1+(\delta+|\nabla u|^2)^{\frac{n}{2}})^{\frac{p-n}{n}}(\delta+|\nabla u|^2)^{\frac{n-2}{2}} |\partial_r u_p|^2=\int_{B_r}  (1+(\delta+|\nabla u|^2)^{\frac{n}{2}})^{\frac{p-n}{n}}(\delta+|\nabla u|^2)^{\frac{n-2}{2}}|\nabla u|^2\\
&+\tfrac{1}{p}\int_{\partial B_r} r (1+(\delta+|\nabla u|^2)^{\frac{n}{2}})^{\frac{p}{n}}-\tfrac{n}{p}\int_{B_r} (1+(\delta+|\nabla u|^2)^{\frac{n}{2}})^{\frac{p}{n}}\le \tfrac{p-n}{p}\int_{B_r} (1+(\delta+|\nabla u|^2)^{\frac{n}{2}})^{\frac{p}{n}}\\
&+ \tfrac{1}{p}\int_{\partial B_r} r (1+(\delta+|\nabla u|^2)^{\frac{n}{2}})^{\frac{p-n}{n}}[1+(\delta+|\nabla u|^2)^{\frac{n-2}{2}}(\delta+|\partial_r u_p|^2+ r^{-2}|\nabla_S u_p|^2)],
\end{aligned}
\end{equation*}
which can be rearranged as
\begin{equation*}
\begin{aligned}
&\int_{\partial B_r} (1+(\delta+|\nabla u|^2)^{\frac{n}{2}})^{\frac{p-n}{n}}(\delta+|\nabla u|^2)^{\frac{n-2}{2}} |\partial_r u_p|^2\le \tfrac{p-n}{(p-1) r}\int_{B_r} (1+(\delta+|\nabla u|^2)^{\frac{n}{2}})^{\frac{p}{n}}\\
&+\tfrac{1}{p-1}\int_{\partial B_r} (1+(\delta+|\nabla u|^2)^{\frac{n}{2}})^{\frac{p-n}{n}} [1+(\delta+|\nabla u|^2)^{\frac{n-2}{2}} (\delta+ r^{-2} |\nabla_S u_p|^2)].
\end{aligned}
\end{equation*}
\end{proof}
Thanks to the Lemma above, it suffices to control (in energy) the tangential component of the gradient of the solutions in order to get decay of the energy along the necks. This observation leads a great advantage, as one can extend the control of the energy decay from the fixed scale given by Theorem \ref{th.SUregularity} to arbitrary large scales, with an explicit bound. The argument differs from the original argument of Sacks-Uhlenbeck, Theorem $3.6$ in \cite{sac}, or Lamm, Lemma $2.5$ in \cite{lam1}, since we cannot control the $\mathbb{L}^\infty-$norm of the Hessian of the solutions. This will cause a slightly worse dependence on the radii of the annulus in consideration, which however will still be enough to prove the energy identity later on in the Section.
\begin{lemma}\label{lemma.hopf2}
There exist uniform $\e_2>0$, $C_{12}>0$ such that for all $\e<\e_2$ and for all solutions $u_{p,\delta}$ of \eqref{eq.ELEpdelta} the following statement holds. Assume $1>R_2\ge 8 R_1>0$. If for every $r \in (R_1,\tfrac{R_2}{2})$ we have
$E_{p,\delta}(u_{p,\delta},A(r,2 r))<\e$, then for every $p\in (n,P_0)$ and $\delta \in (0,1)$ we have
\begin{equation}\label{eq.hopf2}
\int_{A(2 R_1,\tfrac{R_2}{4})} r^{-p}|\nabla_S u|^p \le c \e^{\frac{1}{p}}(1+R_1^{n-p}+R_2^{n-p}).
\end{equation}
\end{lemma}
\begin{proof}
First of all, let us assume that $\e_2<3^{-n}\e_0$, where $\e_0$ comes from Theorem \ref{th.SUregularity}, and let $y \in A(2 R_1,\tfrac{R_2}{4})$. Then we have $\tfrac{2|y|}{3},\tfrac{4|y|}{3} \in (R_1,\tfrac{R_2}{2})$ and $B_{\tfrac{|y|}{3}}(y) \subset A(\tfrac{2|y|}{3},\tfrac{4|y|}{3}) \subset A(R_1,\tfrac{R_2}{2})$, so we can use Theorem \ref{th.SUregularity} to obtain the pointwise bound
\begin{equation}\label{eq.gradientannulus}
|y| |\nabla u_{p,\delta}(y)| \le 3 \tfrac{|y|}{3} \norm{\nabla u_{p,\delta} }_{\infty,B_{\tfrac{\tau_0 |y|}{3}}(y)} \le 3 C E_{p,\delta}(u_{p,\delta},B_{\tfrac{|y|}{3}}(y))^{\frac{1}{p}} \le 3C \e^{\frac{1}{p}}.
\end{equation}
Notice that this bound holds also on the closure of the annulus $A(2 R_1,\tfrac{R_2}{4})$. We will now approximate the solutions $u_{p,\delta}$ with piecewise "$p$-harmonic" functions. In order to do so, split $A(2 R_1,\tfrac{R_2}{4})$ in (tilted-)dyadic annuli $A_k:=A(2^k R_1,2^{k+1}R_1)$, $A_l:=A(2^l R_1,\tfrac{R_2}{4})$, where $l:= \floor{\log_2 \big(\tfrac{R_2}{4R_1} \big)} \in \mathbb{\N}$.

Define the approximant $h_{p,\delta}=h_{p,\delta}(r)$ piecewise, as follows: we minimize on each annulus $A_k$ the $E_{p,\delta}$-energy of functions $v \in W^{1,p}(A_k;\R^\N)$ with constant Dirichlet boundary data $v =\fint_{\partial B_{2^k R_1}} u_{p,\delta}$ on $\partial B_{2^k R_1}$ and $v =\fint_{\partial B_{2^{k+1} R_1}} u_{p,\delta}$ on $\partial B_{2^{k+1} R_1}$ for $k=1,...,l-1$, and analogously for $k=l$. Notice that $h_{p,\delta}$ is a radial function. Clearly, the following equation is satisfied in the weak sense (we drop the pedices $\cdot_{p,\delta}$ to shorten the equation)
\begin{equation*}
\begin{aligned}
 &-\Div[(1+(\delta+|\nabla u|^2)^{\frac{n}{2}})^{\frac{p-n}{n}}(\delta+|\nabla u|^2)^{\frac{n-2}{2}} \nabla u]+\Div[(1+(\delta+|\nabla h|^2)^{\frac{n}{2}})^{\frac{p-n}{n}}(\delta+|\nabla h|^2)^{\frac{n-2}{2}} \nabla h] \\
 &=(1+(\delta+|\nabla u|^2)^{\frac{n}{2}})^{\frac{p-n}{n}}(\delta+|\nabla u|^2)^{\frac{n-2}{2}} A_u(\nabla u, \nabla u).
\end{aligned}
\end{equation*}
As done in the proof of Lemma \ref{lemma.hopf1}, we can justify testing the system by $u-h$ on $A_k$, so that integrating by parts and using the elementary inequality \eqref{eq.elementary2} we obtain for a uniform constant $c=c(n,\N,P_0)$
\begin{equation*}
\footnotesize
\begin{aligned}
&c \int_{A_k} |\nabla u-\nabla h|^p \le \int_{A_k} ((1+(\delta+|\nabla u|^2)^{\frac{n}{2}})^{\frac{p-n}{n}}(\delta+|\nabla u|^2)^{\frac{n-2}{2}} \nabla u-(1+(\delta+|\nabla h|^2)^{\frac{n}{2}})^{\frac{p-n}{n}}(\delta+|\nabla h|^2)^{\frac{n-2}{2}} \nabla h)\\
&(\nabla u-\nabla h) = \int_{A_k} (u-h) (1+(\delta+|\nabla u|^2)^{\frac{n}{2}})^{\frac{p-n}{n}}(\delta+|\nabla u|^2)^{\frac{n-2}{2}} A_u(\nabla u, \nabla u) \\
&+ \int_{\partial B_{2^{k+1} R_1}} (u-h) [(1+(\delta+|\nabla u|^2)^{\frac{n}{2}})^{\frac{p-n}{n}}(\delta+|\nabla u|^2)^{\frac{n-2}{2}} \partial_r u-(1+(\delta+|\nabla h|^2)^{\frac{n}{2}})^{\frac{p-n}{n}}(\delta+|\nabla h|^2)^{\frac{n-2}{2}} \partial_r h]\\
&-\int_{\partial B_{2^k R_1}} (u-h) [(1+(\delta+|\nabla u|^2)^{\frac{n}{2}})^{\frac{p-n}{n}}(\delta+|\nabla u|^2)^{\frac{n-2}{2}} \partial_r u-(1+(\delta+|\nabla h|^2)^{\frac{n}{2}})^{\frac{p-n}{n}}(\delta+|\nabla h|^2)^{\frac{n-2}{2}} \partial_r h]\\
&= \int_{A_k} (u-h) (1+(\delta+|\nabla u|^2)^{\frac{n}{2}})^{\frac{p-n}{n}}(\delta+|\nabla u|^2)^{\frac{n-2}{2}} A_u(\nabla u, \nabla u) \\
&+ \int_{\partial B_{2^{k+1} R_1}} (u-h) (1+(\delta+|\nabla u|^2)^{\frac{n}{2}})^{\frac{p-n}{n}}(\delta+|\nabla u|^2)^{\frac{n-2}{2}} \partial_r u\\
&-\int_{\partial B_{2^k R_1}} (u-h) (1+(\delta+|\nabla u|^2)^{\frac{n}{2}})^{\frac{p-n}{n}}(\delta+|\nabla u|^2)^{\frac{n-2}{2}}\partial_r u.
\end{aligned}
\end{equation*}
In the last step, we have used the Dirichlet boundary data satisfied by $h$ as well as that $|\nabla h|=|h_r|$ is a radial function. From the boundary data of $h$ and the $C^{1,\alpha}$-bounds for $u$ and $h$ on $A_k$, we deduce
\begin{equation}\label{eq.bounduminush}
|u-h|(y) \le |u(y)-\fint_{\partial B_{2^k R_1} } u|+|h(y)-\fint_{\partial B_{2^k R_1}} h| \le \norm{u}_{C^0(A_k)} + \norm{h}_{C^0(A_k)} \le C \e^{\frac{1}{p}} \quad \forall y \in A_k.
\end{equation}
Applying this to the inequality above we deduce
\begin{equation*}
\footnotesize
\begin{aligned}
&\int_{A_k} |\nabla u-\nabla h|^p \le C \e^{\frac{1}{p}} \int_{A_k} (1+(\delta+|\nabla u|^2)^{\frac{n}{2}})^{\frac{p}{n}} + \int_{\partial B_{2^{k+1} R_1}} (u-h) (1+(\delta+|\nabla u|^2)^{\frac{n}{2}})^{\frac{p-n}{n}}(\delta+|\nabla u|^2)^{\frac{n-2}{2}} \partial_r u\\
&-\int_{\partial B_{2^{k+1} R_1}} (u-h) (1+(\delta+|\nabla u|^2)^{\frac{n}{2}})^{\frac{p-n}{n}}(\delta+|\nabla u|^2)^{\frac{n-2}{2}}\partial_r u.
\end{aligned}
\end{equation*}
We will sum these over all $k=1,...,l$ and exploit the telescopic nature of the sum. We can combine inequality \eqref{eq.elementary1}, with the gradient bound from \eqref{eq.gradientannulus} and the uniform bound \eqref{eq.bounduminush} to estimate this sum as
\begin{equation*}
\begin{aligned}
&\int_{A(2 R_1,\tfrac{R_2}{4})} |\nabla u-\nabla h|^p \le C \e^{\frac{1}{p}} \int_{A(2 R_1,\tfrac{R_2}{4})} (1+(\delta+|\nabla u|^2)^{\frac{n}{2}})^{\frac{p}{n}}\\
&+  C \e^{\frac{1}{p}} \int_{\partial B_{\tfrac{R_2}{4}}} (1+(\delta+|\nabla u|^2)^{\frac{n}{2}})^{\frac{p-n}{n}}(\delta+|\nabla u|^2)^{\frac{n-2}{2}} |\partial_r u|\\
&+C \e^{\frac{1}{p}} \int_{\partial B_{2R_1}} (u-h)(1+(\delta+|\nabla u|^2)^{\frac{n}{2}})^{\frac{p-n}{n}}(\delta+|\nabla u|^2)^{\frac{n-2}{2}} |\partial_r u| \le C \e^{\frac{1}{p}}(1+R_1^{n-p}+R_2^{n-p}).
\end{aligned}
\end{equation*}
Notice that we have also used $R_1<R_2\le 1$ and H\"older's inequality for the boundary terms. Using the fact that $h_p$ is radial, we can bound
\begin{equation*}
r^{-2}|\nabla_S u|^2=r^{-2}|\nabla_S u-\nabla_S h|^2 \le r^{-2}|\nabla_S u-\nabla_S h|^2+|\partial_r u-\partial_r h|^2=|\nabla u-\nabla h|^2,
\end{equation*}
from which we deduce the thesis by rising to the $p/2$-power
\begin{equation*}
\int_{A(2 R_1,\tfrac{R_2}{4})} r^{-p}|\nabla_S u|^p \le C \e^{\frac{1}{p}}(1+R_1^{n-p}+R_2^{n-p}).
\end{equation*}
\end{proof}
\subsection{Proof of the energy identity}
\begin{proof}[Proof of Theorem \ref{th.quantization}]
First of all, we recall the setting for convenience of the reader. From the classical induction argument of Ding and Tian \cite{din}, it suffices to prove the energy identity in the presence of one single bubble. Therefore, thanks to the local nature of the problem, we focus on a family of solutions $u_k:B_1\rightarrow \N$ relative to parameters $p_k$, $\delta_k$ and $s=1$, admitting only one energy concentration point $x_1=0$. According to what done in this section so far, we denote by $\omega:S^n \rightarrow \N$ the bubble, by $r_k$ the concentration radii associated to the bubble, and by $u_\infty \in C^{1,\alpha}(B_1;\N)$ the $n$-harmonic map arising as a limit in $C^{1,\alpha}(B_1 \setminus \set{0};\N)$ of the family $u_k$. From the $C^{1,\alpha}$-convergence we see that for every $R_0 \in (0,1)$, we have
\begin{equation*}
\lim_{k \rightarrow +\infty} E_{p_k,\delta_k}(u_k,B_1 \setminus B_{R_0})=D_n(u_\infty,B_1 \setminus B_{R_0}).
\end{equation*}
By the definition of the bubble $\omega$, the rescalings $v_k(y):=u_k(r_k y)$ converge in $C^{1,\alpha}$ to $\omega$, so that for every $R>0$
\begin{equation*}
\lim_{k \rightarrow +\infty} E_{p_k,\delta_k}(u_k,B_{R r_k})=D_{n}(\omega,B_{R}).
\end{equation*}
Furthermore, for any fixed $M>1$, the same reasoning leading to the identities above guarantees
\begin{equation}\label{eq.decayM}
\lim_{k \rightarrow +\infty} E_{p_k,\delta_k}(u_k,B_{R_0} \setminus B_{\frac{R_0}{M}})+E_{p_k,\delta_k}(u_k,B_{M R r_k} \setminus B_{R r_k}) \longrightarrow 0,
\end{equation}
as $R_0 \rightarrow 0$ and $R \rightarrow +\infty$. Therefore, the proof of the energy identity is reduced to show
\begin{equation}\label{eq.limitquantization}
\lim_{R \rightarrow +\infty} \lim_{R_0 \rightarrow 0} \lim_{k \rightarrow +\infty} E_{p_k,\delta_k}(u_k,A(r_k R,R_0))=0
\end{equation}
To begin, we need to work towards proving the pointwise gradient bound \eqref{eq.gradientdecay}, assumed throughout the entire Section. We claim that for any $\e>0$ there exists $k_0 \in \mathbb{\N}$ such that for all $k\ge k_0$ and for all $r \in [R r_k,\tfrac{R_0}{2}]$ we have
\begin{equation*}
E_{p_k,\delta_k}(u_k,B_{2 r} \setminus B_r)<\e.
\end{equation*}
The proof relies on the fact that $\omega$ is the unique bubble. Indeed, suppose by contradiction that for some $\e>0$ and for every $k$ there exists a radius $t_k \in (R r_k,\tfrac{R_0}{2})$ such that
\begin{equation*}
E_{p_k,\delta_k}(u_k,B_{2 t_k} \setminus B_{t_k})=\max_{[R r_k,\tfrac{R_0}{2}]}  E_{p_k,\delta_k}(u_k,B_{2 r} \setminus B_r)\ge \e.
\end{equation*}
Notice that we have used \eqref{eq.decayM} to ensure that $t_k$ is not equal to $R r_k$ or $\tfrac{R_0}{2}$. Even more precisely, exploiting the full power of \eqref{eq.decayM}, we get the stronger relations
\begin{equation*}
\tfrac{R_0}{t_k}\longrightarrow +\infty \quad \text{and} \quad \tfrac{R r_k}{t_k}\longrightarrow 0.
\end{equation*}
In order to construct a different bubble, we let $\tilde{v}_k:A(\tfrac{R_0}{t_k},\tfrac{R r_k}{t_k}) \rightarrow \N$ be the rescaling $\tilde{v}_k(x):=u_k(t_k x)$. Notice that the $\tilde{v}_k$ solve \eqref{eq.ELEpdelta} with $\delta':= \delta_k \cdot t_k^2$ and $s':=1 \cdot t_k^n$, and they satisfy $t_k^{n-p_k} E^{(s')}_{p_k,\delta'}(\tilde{v}_k,B_2 \setminus B_1) \ge \e$ by rescaling, and $ t_k^{n-p_k} E^{(s')}_{p_k,\delta'}(\tilde{v}_k,A(\tfrac{R_0}{t_k},\tfrac{R r_k}{t_k})) \le 4 \Lambda_0$ (we can just repeat the argument used in inequality \eqref{eq.uniformboundrescale}). Notice that $t_k^{n-p_k} \rightarrow 1$ since $R_0$ and $R$ are fixed, and we can use \eqref{eq.concentrationradiientropy}. Given this uniform energy bound, we can assume without loss of generality that $\tilde{v}_k \rightarrow \tilde{\omega}$ weakly in $W^{1,n}_{loc}(\R^n \setminus \set{0};\N)$, for some $n$-harmonic map $\tilde{\omega}:\R^n \rightarrow \N$ with finite $n$-energy.

We now distinguish two cases. In the first case, we assume that for some $\tilde{r}>0$ we have
\begin{equation*}
\sup_{k \in \mathbb{\N}} \sup_{x \in B_4 \setminus B_{\frac{1}{4}} } E^{(s')}_{p_k,\delta'}(\tilde{v}_k,B_{\tilde{r} }(x))<\e_0.
\end{equation*}
Thus, we can use Theorem \eqref{th.SUregularity} together with a covering argument, to deduce $\tilde{v}_k \rightarrow \tilde{\omega}$ in $C^{1,\alpha_1}(B_2 \setminus B_{\frac{1}{2}})$ for some $\alpha_1<\alpha_0$. The limit procedure ensures that $E_n(\tilde{\omega},B_2 \setminus B_{\frac{1}{2}})\ge \e$, so that $\tilde{\omega}$ is a non-constant map. We now observe that $B_2 \setminus B_{\frac{1}{2}}$ is conformally equivalent to $S^n \setminus \set{S,\N}$, so we can lift $\tilde{\omega}$ and extend it to the whole $S^n$ by Duzaar-Fuchs' removability of the singularity Theorem \cite{duz0}, that is we have constructed a new (non-trivial) bubble, contradicting the uniqueness of $\omega$ assumed from the beginning of the proof.

In the second and complementary case, there exists at least one energy accumulation point (i.e. at a energy level higher than $\e_0$) for the sequence $\tilde{v}_k$ in $B_4 \setminus B_{\frac{1}{4}}$. We can therefore argue as in the beginning of this section, namely through the use of the maximal concentration function, to show the existence of a non-trivial bubble, contradicting once again the uniqueness of the bubble $\omega$. This concludes the proof of the claim.
\vspace{0.5cm}

Combine Lemma \ref{lemma.hopf1}, the orthogonal decomposition of $\nabla u$ and weighted Young's inequality with exponents $(\tfrac{p}{p-2},\tfrac{p}{2})$ to get
\begin{equation*}
\footnotesize
\begin{aligned}
&E_{p_k,\delta_k}(u_k,A(r_k R,R_0))=\int_{A(r_k R,R_0)} (1+(\delta_k+|\nabla u_k|^2)^{\frac{n}{2}})^{\frac{p_k}{n}} - (1+\delta_k^{\frac{n}{2}})^{\frac{p}{n}} \le \int_{A(r_k R,R_0)} (1+(\delta_k+|\nabla u_k|^2)^{\frac{n}{2}})^{\frac{p_k-n}{n}}\\
&+(1+(\delta_k+|\nabla u_k|^2)^{\frac{n}{2}})^{\frac{p_k-n}{n}} [\delta_k (\delta_k+|\nabla u_k|^2)^{\frac{n-2}{2}}+(\delta_k+|\nabla u_k|^2)^{\frac{n-2}{2}}(|\partial_r u_k|^2+|x|^{-2} |\nabla_S u_k|^2)]\\
&= \int_{A(r_k R,R_0)}(1+(\delta_k+|\nabla u_k|^2)^{\frac{n}{2}})^{\frac{p_k-n}{n}}+(1+(\delta_k+|\nabla u_k|^2)^{\frac{n}{2}})^{\frac{p_k-n}{n}} (\delta_k+|\nabla u_k|^2)^{\frac{n-2}{2}}\\
&+ \int_{r_k R}^{R_0} \int_{\partial B_r} (1+(\delta_k+|\nabla u_k|^2)^{\frac{n}{2}})^{\frac{p_k-n}{n}} (\delta_k+|\nabla u_k|^2)^{\frac{n-2}{2}} (|\partial_r u_k|^2+r^{-2} |\nabla_S u_k|^2)\\
&\le \tfrac{p_k}{p_k-1} \int_{A(r_k R,R_0)} (1+(\delta_k+|\nabla u_k|^2)^{\frac{n}{2}})^{\frac{p_k-n}{n}}+(1+(\delta_k+|\nabla u_k|^2)^{\frac{n}{2}})^{\frac{p_k-n}{n}} (\delta_k+|\nabla u_k|^2)^{\frac{n-2}{2}}\\
&+\int_{r_k R}^{R_0} \tfrac{p_k-n}{(p_k-1)r} \int_{B_r} (1+(\delta_k+|\nabla u_k|^2)^{\frac{n}{2}})^{\frac{p_k}{n}}+ \int_{r_k R}^{R_0} \tfrac{1}{(p_k-1)r^2}\int_{\partial B_r} (1+(\delta_k+|\nabla u_k|^2)^{\frac{n}{2}})^{\frac{p_k-n}{n}} (\delta_k+|\nabla u_k|^2)^{\frac{n-2}{2}} |\nabla_S u_k|^2 \\
&\le \tfrac{n}{n-1} \int_{A(r_k R,R_0)}  (1+(\delta_k+|\nabla u_k|^2)^{\frac{n}{2}})^{\frac{p_k-n}{n}}+(1+(\delta_k+|\nabla u_k|^2)^{\frac{n}{2}})^{\frac{p_k-n}{n}} (\delta_k+|\nabla u_k|^2)^{\frac{n-2}{2}}\\
&+\int_{r_k R}^{R_0} \tfrac{p_k-n}{(n-1)r} \int_{B_r} (1+(\delta_k+|\nabla u_k|^2)^{\frac{n}{2}})^{\frac{p_k}{n}}+\tfrac{1}{4}\int_{A(r_k R,R_0)} (1+(\delta_k+|\nabla u_k|^2)^{\frac{n}{2}})^{\frac{p_k}{n}}+c(P_0,n) \int_{A(r_k R,R_0)} r^{-p_k}|\nabla_S u_k|^{p_k}.
\end{aligned}
\end{equation*}
Further applications of Young's inequality (relative to $(\tfrac{p_k}{p_k-n},\tfrac{p_k}{n})$ and weighted $(\tfrac{p}{p-2},\tfrac{p}{2})$ ) allow to estimate the first two terms on the RHS:
\begin{equation*}
\footnotesize
\begin{aligned}
&\int_{A(r_k R,R_0)}(1+(\delta_k+|\nabla u_k|^2)^{\frac{n}{2}})^{\frac{p_k-n}{n}} \le \int_{A(r_k R,R_0)} \tfrac{p_k-n}{p_k} (1+(\delta_k+|\nabla u_k|^2)^{\frac{n}{2}})^{\frac{p_k}{n}} +C (R_0^n-(r_k R)^n),\\
&\int_{A(r_k R,R_0)} (1+(\delta_k+|\nabla u_k|^2)^{\frac{n}{2}})^{\frac{p_k-n}{n}} (\delta_k+|\nabla u_k|^2)^{\frac{n-2}{2}} \le \int_{A(r_k R,R_0)} \tfrac{p_k-2}{p_k} \eta (1+(\delta_k+|\nabla u_k|^2)^{\frac{n}{2}})^{\frac{p_k}{n}} + \tfrac{C}{\eta} (R_0^n-(r_k R)^n).
\end{aligned}
\end{equation*}
For $k$ large enough, we have $\tfrac{n}{n-1}\tfrac{p_k-n}{p}\le \tfrac{1}{8}$, so choosing $\eta=\tfrac{n-1}{8 n}$ we can absorb the $E_{p_k,\delta_k}$-energy terms back to the left hand side; overall, we can finally rearrange, using Lemma \ref{lemma.hopf2}, the concentration radii estimate \eqref{eq.concentrationradiientropy} and the global $E_{p_k,\delta_k}$-energy bound to obtain
\begin{equation*}
\begin{aligned}
&E_{p_k}(u_k,A(r_k R,R_0))\le \int_{r_k R}^{R_0} \tfrac{2(p_k-n)}{(n-1)r} \int_{B_r}  (1+(\delta_k+|\nabla u_k|^2)^{\frac{n}{2}})^{\frac{p_k}{n}}+ C \int_{A(r_k R,R_0)} r^{-p_k}|\nabla_S u_k|^{p_k} \\
&+C (R_0^n-(r_k R)^n)\le \int_{r_k R}^{R_0} \tfrac{p_k-n}{(n-1)r} (2\Lambda_0+C r^n) + C\e^{\frac{1}{p_k}}(1+(r_k R)^{n-p_k}+R_0^{n-p_k})+C(R_0^n-(r_k R)^n)\\
&=C(p_k-n) \log \Big( \tfrac{R_0}{r_k R} \Big)+ C\e^{\frac{1}{p_k}}(1+(r_k R)^{n-p_k}+R_0^{n-p_k}) +C(R_0^n-(r_k R)^n)\\
&\underbrace{\longrightarrow}_{k\rightarrow \infty} C \log(1)+ C\e^{\frac{1}{n}}(1+1+1)+C R_0^n=C \e^{\frac{1}{n}}+C R_0^n.
\end{aligned}
\end{equation*}
Finally, we can let first $R_0 \rightarrow 0$ and $R\rightarrow +\infty$, then $\e \rightarrow 0$ to obtain \eqref{eq.limitquantization} and therefore conclude the proof.
\end{proof}

\section{Min-max construction}\label{sec.minmax}
In this section, we show how Theorem \ref{th.quantization} allows to solve min-max problems for the $E_n$-energy up to bubbles. Throughout the entire Section, we will assume that
\begin{equation*}
E_{p,\delta}(u):=\int_M (1+(\delta+|\nabla u|^2)^{\frac{n}{2}})^{\frac{p}{n}},
\end{equation*}
since this affine transformation does not change the critical points as well as the monotonicity of the functionals in the parameter $p$. Let us start with a Lemma, guaranteeing the existence of a family of critical points to $E_{p,\delta}$ satisfying the entropy condition in Theorem \ref{th.quantization} in quite general min-max problems. 
\begin{lemma}\label{lemma.minmax}
Let $p\in (n,P_0)$, $\delta \in [0,1]$ and $\mathcal{F} \subset \mathcal{P}(W^{1,p}(\M;\N))$ be a collection of sets. Consider a continuous semi-flow $\Phi:[0,+\infty)\times W^{1,p}(\M;\N) \rightarrow W^{1,p}(\M;\N)$ such that $\Phi(0,\cdot)=id$, $\Phi(t,\cdot)$ is a homeomorphism of $W^{1,p}(\M;\N)$ for any $t\ge 0$, and $E_{p,\delta}(\Phi(t,u))$ is non-increasing in $t$ for any $u \in W^{1,p}(\M;\N)$. Suppose further that for every $F \in \mathcal{F}$ and
all $t \in [0,+\infty)$ we have $\Phi(t,F) \subset F$. Define the min-max value
\begin{equation*}
\beta_{p,\delta}:= \inf_{F \in \mathcal{F}} \sup_{u\in F} E_{p,\delta}(u),
\end{equation*}
and assume that $\beta_{p,\delta}<\infty$ for all $p \in (n,P_0]$ and $\delta \in [0,1]$. Then for almost every $p \in (n,P_0)$ and every $\delta_0 \in [0,1]$ there exists a critical point $u_{p,\delta_0} \in C^{\infty}(\M;\N)$ of $E_{p,\delta_0}$ at the energy level $E_{p,\delta_0}(u_{p,\delta_0})=\beta_{p,\delta_0}$, satisfying the entropy condition
\begin{equation}\label{eq.entropyminmax}
\lim_{p\rightarrow n} (p-n)\int_M (1+(\delta_0+|\nabla u_{p,\delta_0}|^{2})^{\frac{n}{2}})^{\frac{p}{n}}\log(1+(\delta_0+|\nabla u_{p,\delta_0}|^{2})^{\frac{n}{2}}) =0.
\end{equation}
\end{lemma}
\begin{remark}
As said in Section \ref{sec.preliminary} we know that, for all $p>n$ and $\delta \in (0,1)$, we can apply some standard min-max principle to obtain the existence of critical points $u_{p,\delta}$ of $E_{p,\delta}$ at energy level $E_{p,\delta}(u_{p,\delta})=\beta_{p,\delta}$. However, there is no guarantee a-priori, that the sequence satisfies the entropy condition \eqref{eq.entropyminmax}.

Moreover, we notice that the entropy condition gives some sort of asymptotic bound on the sets where the gradients of the solutions are high enough, therefore the uniformity of the statement in the parameter $\delta_0$ should not be surprising.

Finally, the proof is a bit harder than that of Lemma $1.4$ in \cite{lam1} exactly due to the presence of the parameter $\delta_0$.
\end{remark}
\begin{proof}
The proof is based on the powerful Struwe's monotonicity trick from \cite{stru3}. First of all, we see that $p \mapsto \beta_{p,\delta_0}$ is non-decreasing for any fixed $\delta_0$, and so it is differentiable almost everywhere, with non-negative derivative $\tfrac{d \beta_{p,\delta_0}}{d p}\in \mathbb{L}^1([n,P])$. Therefore we must have
\begin{equation}\label{eq.minmax0}
\liminf_{p \rightarrow n} (p-n)\log(\tfrac{1}{p-n})\tfrac{d \beta_{p,\delta_0} }{d p}=0.
\end{equation}
Let us pick a point of differentiability $p>n$ of $\beta_{p,\delta_0}$ and a non-increasing sequence $p_{k+1} \le p_k$, with $p_k \rightarrow p$. By definition of $\beta_{p_k,\delta_0}$, for every $k$ there exists an element $F_k:=F_k(\delta_0) \in \mathcal{F}$ such that 
\begin{equation*}
\sup_{u \in F_k} E_{p_k,\delta_0}(u) \le \beta_{p_k,\delta_0}+(p_k-p).
\end{equation*}
On the other hand, $p$ is a differentiability point for $\beta_{p,\delta_0}$, and hence for $k$ large enough
\begin{equation*}
\beta_{p_k,\delta_0} \le \beta_{p,\delta_0} + (\tfrac{d \beta_{p,\delta_0} }{d p}+1)(p_k-p),
\end{equation*}
which combined with the above inequality and the monotonicity of $p\mapsto E_{p,\delta_0}$ gives
\begin{equation*}
\beta_{p,\delta_0} \le \sup_{u \in F_k} E_{p,\delta_0}(u) \le \sup_{u \in F_k} E_{p_k,\delta_0}(u)\le \beta_{p_k,\delta_0}+(p_k-p)\le \beta_{p,\delta_0} + (\tfrac{d \beta_{p,\delta_0}}{d p}+2)(p_k-p).
\end{equation*}
Overall, we can choose a map $v_k \in F_k$ almost realising the $\sup E_{p,\delta_0}$ on $F_k$, in the sense that we have
\begin{equation}\label{eq.minmax1}
\beta_{p,\delta_0} -(p_k-p)\le E_{p,\delta_0}(v_k) \le E_{p_k,\delta_0}(v_k) \le \beta_{p_k,\delta_0}+(p_k-p)\le \beta_{p,\delta_0} + (\tfrac{d \beta_{p,\delta_0}}{d p}+2)(p_k-p).
\end{equation}
As in \cite{lam1}, we split now the proof in several steps.

\begin{itemize}
\item[Step 1]: For every $v \in W^{1,p}(\M;\mathcal{\N})$ satisfying \eqref{eq.minmax1} in place of $v_k$, we can bound the $p$-derivative of the energy as
\begin{equation*}
\partial_p E_{p,\delta_0}(v) \le \tfrac{d \beta_{p,\delta_0}}{d p}+3.
\end{equation*}
Indeed, \eqref{eq.minmax1} gives
\begin{equation*}
\tfrac{E_{p_k,\delta_0}(v)-E_{p,\delta_0}(v)}{p_k-p} \le \tfrac{d \beta_{p,\delta_0}}{d p}+3,
\end{equation*}
and therefore the mean value Theorem yields the existence of a point of differentiability $p' \in [p,p_k]$ such that
\begin{equation*}
\partial_p E_{p,\delta_0} \mid_{p=p'}(v) \le \tfrac{d \beta_{p,\delta_0}}{d p}+3
\end{equation*}
In order to conclude this step, it is sufficient to notice that
\begin{equation*}
\begin{aligned}
&\partial_p E_{p,\delta_0}(v)=\tfrac{p}{n}\int_M (1+(\delta_0+|\nabla v|^2)^{\frac{n}{2}})^{\frac{p}{n}}\log(1+(\delta_0+|\nabla v|^2)^{\frac{n}{2}}) \\
&\le \tfrac{p'}{n}\int_M (1+(\delta_0+|\nabla v|^2)^{\frac{n}{2}})^{\frac{p'}{n}}\log(1+(\delta_0+|\nabla v|^2)^{\frac{n}{2}}) = \partial_p E_{p,\delta_0} \mid_{p=p'}(v).
\end{aligned}
\end{equation*}
\item[Step 2]: We have the following uniform $C^1$-bound valid for all solution $v$ of \eqref{eq.minmax1} (in place of $v_k$)
\begin{equation}\label{eq.minmax2}
\sup_{v \in W^{1,p}, v \sim \eqref{eq.minmax1}} \sup_{\norm{\phi}_{W^{1,p_k}(v^*T\mathcal{\N})}\le 1} |\scal{d E_{p_k,\delta_0}(v)}{\phi}-\scal{d E_{p,\delta_0}(v)}{\phi}| \longrightarrow 0.
\end{equation}
In fact, for every $\phi \in W^{1,p_k}(v^*T\mathcal{\N})$ with norm less than or equal to $1$, we compute with the help of H\"older's inequality, Young's weighted inequality with exponents $(\tfrac{n}{n-1},n)$, $p_k-p\le 1$ and \eqref{eq.minmax1} several times
\begin{equation*}
\footnotesize
\begin{aligned}
&|\scal{d E_{p_k,\delta_0}(v)}{\phi}-\scal{d E_{p,\delta_0}(v)}{\phi}| \le \int_M [p_k(1+(\delta_0+|\nabla v|^2)^{\frac{n}{2}})^{\frac{p_k-n}{n}}-p(1+(\delta_0+|\nabla v|^2)^{\frac{n}{2}})^{\frac{p-n}{n}}] (\delta_0+|\nabla v|^2)^{\frac{n-2}{2}} \\
&|\nabla v| |\nabla \phi|\le (p_k-p) \Big( \int_M (1+(\delta_0+|\nabla v|^2)^{\frac{n}{2}})^{\frac{p_k}{n}}\Big)^{\frac{p_k-n}{p_k}} \Big( \int_M (\delta_0+|\nabla v|^2)^{\frac{p_k}{2}} \Big)^{\frac{n-1}{p_k}} \Big( \int_M |\nabla \phi|^{p_k}\Big)^{\frac{1}{p_k}}\\
&+p \int_M [(1+(\delta_0+|\nabla v|^2)^{\frac{n}{2}})^{\frac{p_k-n}{n}}-(1+(\delta_0+|\nabla v|^2)^{\frac{n}{2}})^{\frac{p-n}{n}}](\delta_0+|\nabla v|^2)^{\frac{n-2}{2}} |\nabla v| |\nabla \phi| \\
&\le \Big( \beta_{p,\delta_0} + (\tfrac{d \beta_{p,\delta_0}}{d p}+2)(p_k-p)+c \Big)^{\frac{p_k-1}{p_k}} (p_k-p)+p \int_M [(1+(\delta_0+|\nabla v|^2)^{\frac{n}{2}})^{\frac{p_k-n}{n}}-(1+(\delta_0+|\nabla v|^2)^{\frac{n}{2}})^{\frac{p-n}{n}}]\cdot\\
&\cdot [\tfrac{n-1}{n}\nu^{\tfrac{n}{1-n}}(\delta_0+|\nabla v|^2)^{\frac{n}{2}}+ \tfrac{1}{n}\nu^n|\nabla \phi|^n]\\
&\le c(p,P,\beta_{p,\delta_0},\tfrac{d \beta_{p,\delta_0}}{d p}) (p_k-p)+p\tfrac{n-1}{n}\nu^{\tfrac{n}{1-n}}(E_{p_k,\delta_0}(v)-E_{p,\delta_0}(v))+\tfrac{p}{n}\nu^n\int_M (1+(\delta_0+|\nabla v|^2)^{\frac{n}{2}})^{\frac{p_k-n}{n}}|\nabla \phi|^n\\
&\le c (p_k-p)(1+\nu^{\tfrac{n}{1-n}})+\tfrac{p}{n}\nu^n[\tfrac{p_k-n}{p_k}E_{p_k,\delta_0}(v)+c(p_k-n)+\tfrac{n}{p_k}\int_M |\nabla \phi|^{p_k}] \le c (p_k-p)^{\tfrac{n-1}{n}} \longrightarrow 0.
\end{aligned}
\end{equation*}
In the last step we have set $\nu:=(p_k-p)^{\tfrac{n-1}{n^2}}$. We explicitly remark that the constant can be chosen uniformly in $\delta_0$. This concludes the proof of the claim.

\item[Step 3]: We have the following limiting condition: for any sequence $\delta_k \rightarrow \delta \in [0,1]$ and any $p\in (n,P_0)$
\begin{equation}\label{eq.minmax2delta}
\sup_{v \in W^{1,p}, E_{p,\delta}(v) \le \beta_{P_0,1} } \sup_{\norm{\phi}_{W^{1,p}(v^*T\mathcal{\N})}\le 1} |\scal{d E_{p,\delta_k}(v)}{\phi}-\scal{d E_{p,\delta}(v)}{\phi}| \longrightarrow 0.
\end{equation}
Indeed, for every $v \in W^{1,p}$, such that $E_{p,\delta}(v) \le \beta_{P_0,1}$ and $\phi \in W^{1,p}(v^*T\mathcal{\N})$ with norm less than or equal to $1$, we use the following elementary inequalities
\begin{align*}
&(x+y)^{\frac{p-n}{n}} \le (x)^{\frac{p-n}{n}}+(y)^{\frac{p-n}{n}}, \quad \forall x,y \ge 0,\\
&(x+y)^{\frac{n-2}{n}} \le (x)^{\frac{n-2}{n}}+(y)^{\frac{n-2}{n}}, \quad \forall x,y \ge 0,\\
&|(a_1+b)^{\frac{n}{2}}-(a_2+b)^{\frac{n}{2}}| \le c(n)|a_1-a_2|[|a_1-a_2|^{\frac{n-2}{2}}+|b|^{\frac{n-2}{2}}],\\
&|\nabla v| \le(\delta_k+|\nabla v|^2)^{\frac{1}{2}}\le (1+(\delta+|\nabla v|^2)^{\frac{n}{2}})^{\frac{1}{n}},\\
&\tfrac{p}{p-1}(\tfrac{p-n}{n}(n-2)+1+n-2)\le p,\\
&\tfrac{p}{p-1}(p-n+\tfrac{n-2}{n}(n-2)+1)\le p
\end{align*}
together with H\"older's inequality with exponents $p$ and $\tfrac{p}{p-1}$ and $2 \le n < p$ to obtain
\begin{equation*}
\begin{aligned}
&|\scal{d E_{p,\delta_k}(v)}{\phi}-\scal{d E_{p,\delta}(v)}{\phi}| \le p \int_M |(1+(\delta_k+|\nabla v|^2)^{\frac{n}{2}})^{\frac{p-n}{n}}(\delta_k+|\nabla v|^2)^{\frac{n-2}{2}}\\
&-(1+(\delta+|\nabla v|^2)^{\frac{n}{2}})^{\frac{p-n}{n}}(\delta+|\nabla v|^2)^{\frac{n-2}{2}} \pm (1+(\delta+|\nabla v|^2)^{\frac{n}{2}})^{\frac{p-n}{n}}(\delta_k+|\nabla v|^2)^{\frac{n-2}{n} \frac{n}{2}} | |\nabla v| |\nabla \phi| \\
&\le p \int_M |(\delta_k+|\nabla v|^2)^{\frac{n}{2}}-(\delta+|\nabla v|^2)^{\frac{n}{2}}|^{\frac{p-n}{n}} (\delta_k+|\nabla v|^2)^{\frac{n-2}{2}}|\nabla v| |\nabla \phi|+(1+(\delta+|\nabla v|^2)^{\frac{n}{2}})^{\frac{p-n}{n}}\cdot\\
&\cdot|(\delta_k+|\nabla v|^2)^{\frac{n}{2}}-(\delta+|\nabla v|^2)^{\frac{n}{2}}|^{\frac{n-2}{n}}|\nabla v| |\nabla \phi| \le c \int_M |\delta_k-\delta|^{\frac{p-n}{n}} [|\delta_k-\delta|^{\frac{n-2}{2}}+|\nabla v|^{n-2}]^{\frac{p-n}{n}} \cdot\\
&\cdot (\delta_k+|\nabla v|^2)^{\frac{n-2}{2}}|\nabla v| |\nabla \phi|+|\delta_k-\delta|^{\frac{n-2}{n}} (1+(\delta+|\nabla v|^2)^{\frac{n}{2}})^{\frac{p-n}{n}} [|\delta_k-\delta|^{\frac{n-2}{2}}+|\nabla v|^{n-2}]^{\frac{n-2}{n}}|\nabla v| |\nabla \phi|\\
&\le c \int_M |\delta_k-\delta|^{\frac{p-n}{n}+\frac{n-2}{2}} (\delta_k+|\nabla v|^2)^{\frac{n-1}{2}} |\nabla \phi| + |\delta_k-\delta|^{\frac{p-n}{n}} (\delta_k+|\nabla v|^2)^{\frac{n-2}{2}} |\nabla v|^{\frac{p-n}{n}(n-2)+1} |\nabla \phi|\\
&+|\delta_k-\delta|^{\frac{n^2-4}{2n}} (1+(\delta+|\nabla v|^2)^{\frac{n}{2}})^{\frac{p-n}{n}} |\nabla v| |\nabla \phi| +|\delta_k-\delta|^{\frac{n-2}{n}} (1+(\delta+|\nabla v|^2)^{\frac{n}{2}})^{\frac{p-n}{n}}|\nabla v|^{\frac{n-2}{n}(n-2)+1} |\nabla \phi|\\
&\le c(P_0,n,\beta_{P,1}) [|\delta_k-\delta|^{\frac{p-n}{n}+\frac{n-2}{2}} +|\delta_k-\delta|^{\frac{p-n}{n}}+|\delta_k-\delta|^{\frac{n^2-4}{2n}}+|\delta_k-\delta|^{\frac{n-2}{n}}]\longrightarrow 0.
\end{aligned}
\end{equation*}
Notice that the bound found is not infinitesimal as $p \searrow n$ contrary to the case treated above; however, this lack of uniformity will not influence the rest of the proof, as Step $3$ will be used only in Step $4$ when $p>n$ is fixed.

\item[Step 4]: We refine the sequence $v_k$ to an asymptotic critical sequence $u_k \in W^{1,p_k}(\M;\mathcal{\N})$ satisfying similar properties; more precisely, we want to find $u_k$ satisfying not only \eqref{eq.minmax1} in place of $v_k$, but also
\begin{equation*}
   \sup_{\delta_0 \in [0,1]} \norm{d E_{p_k,\delta_0}(u_k)}_{(W^{1,p_k}(\M;\mathcal{\N}))^*} \longrightarrow 0.
\end{equation*}
Suppose instead by contradiction that
\begin{equation*}
\norm{d E_{p_k,\delta_k}(u)}_{(W^{1,p_k}(\M;\mathcal{\N}))^*} \ge 4 \eta,
\end{equation*}
for some $\eta>0$ and for all $u$ verifying \eqref{eq.minmax1} in place of $u_k$, with $\delta_k$ in place of $\delta_0$, and for all $k$ large enough.

For these $k$, consider pseudo-gradient vector fields $e_k \in C^{0,1}_{loc}(W^{1,p_k}(\M;\mathcal{\N});W^{1,p_k}(\M;u^*T\mathcal{\N}))$ for the functionals $E_{p_k,\delta_k}$ with $\norm{e_k(u)}_{W^{1,p_k}(u^*T\mathcal{\N})} \le 1$ and
\begin{equation}\label{eq.minmax3}
\scal{d E_{p_k,\delta_k}(u)}{e_k(u)} \le -\tfrac{1}{2}\norm{d E_{p_k,\delta_k}(u)}_{(W^{1,p_k}(\M;\mathcal{\N}))^*} \le -2 \eta,
\end{equation}
for all $u$ satisfying \eqref{eq.minmax1}. Without loss of generality we can assume that $\delta_k \rightarrow \delta_\infty \in [0,1]$. Consider a cut-off function $\psi \in C^\infty (\R)$ such that $0 \le \psi \le 1$, $\psi \equiv 1$ on $(-\infty,0]$, $\psi \equiv 0$ on $[1,+\infty]$, and set
\begin{equation*}
\psi_k(u):= \psi \big( \tfrac{E_{p,\delta_\infty}(u)-(\beta_{p_k,\delta_\infty}-(p_k-p))}{p_k-p}\big).
\end{equation*}
Introduce the cut-off pseudo-gradient vector fields $\tilde{e}_k(\cdot):=\psi_k(\cdot) e_k(\cdot)$. These fields generate flows $\Phi_k:\R_0^{+} \times W^{1,p_k}(\M;\mathcal{\N}) \rightarrow W^{1,p_k}(\M;\mathcal{\N})$ defined by
\begin{equation*}
\begin{cases}
&\tfrac{d}{dt} \Phi_k(t,u)=\tilde{e}_k(\Phi_k(t,u)), \quad\text{for } t>0,\\
&\Phi_k(0,u)=u.
\end{cases}
\end{equation*}
Choose $F_k:= F_k(\delta_\infty) \in \mathcal{F}$ as above, that is almost realising the infimum with $\delta_\infty$, meaning that
\begin{equation*}
\sup_{u \in F_k} E_{p_k,\delta_\infty}(u) \le \beta_{p_k,\delta_\infty}+(p_k-p),
\end{equation*}
and notice that by the dynamic closedness of the family assumed, for all $v \in F_k$ and $t \in \R_0^{+}$, $v_t:=\Phi_k(t,v) \in F_k$. By construction, we also have that
\begin{equation*}
M(t):=\sup_{v \in F_k} E_{p,\delta_\infty}(v_t) \le \sup_{v \in F_k} E_{p_k,\delta_\infty}(v_t)\le \sup_{u \in F_k} E_{p_k,\delta_\infty}(u)\le \beta_{p_k,\delta_\infty}+(p_k-p) \quad \forall t \ge 0.
\end{equation*}
This implies that the quantity $M(t)$, which by definition satisfies $M(t) \ge \beta_{p,\delta_\infty}$, is attained at points $v'_t$ satisfying \eqref{eq.minmax1} with $\delta_\infty$ in place of $\delta_0$, so that we must have $\psi_k(v'_t)=1$. Applying this identity, the definition of flow, together with \eqref{eq.minmax2} (and its uniformity in $\delta_0$), \eqref{eq.minmax2delta} and \eqref{eq.minmax3} we compute
\begin{equation*}
\begin{aligned}
&\tfrac{d}{dt} E_{p,\delta_\infty}(v'_t)=\scal{d E_{p,\delta_\infty}(v'_t)}{\tfrac{d}{dt} v'_t}=\psi_k(v'_t) \scal{d E_{p,\delta_\infty}(v'_t)}{e_k(v'_t)}=\scal{d E_{p,\delta_\infty}(v'_t)}{e_k(v'_t)}\\
&\le \scal{d E_{p_k,\delta_k}(v'_t)}{e_k(v'_t)}+|\scal{d E_{p,\delta_k}(v'_t)}{e_k(v'_t)}-\scal{d E_{p_k,\delta_k}(v'_t)}{e_k(v'_t)}|\\
&+|\scal{d E_{p,\delta_\infty}(v'_t)}{e_k(v'_t)}-\scal{d E_{p,\delta_k}(v'_t)}{e_k(v'_t)}| \le -2 \eta+ \mathfrak{o}_k(1).
\end{aligned}
\end{equation*}
Therefore $\tfrac{d}{dt} M(t) \le -\eta<0$, and hence for $t$ large enough $M(t)<\beta_{p,\delta_\infty}$, in contradiction to the inequality observed above, concluding the proof of this step. Let us explicitly remark that we are using the uniformity of \eqref{eq.minmax2} with respect to $\delta_k$ as well as that $p>n$ is fixed in this contradiction argument to use Step $3$.
\item[Step 5]: In the final step of the proof, for any given $\delta_0\in [0,1]$, we show that the sequence $u_k$ "constructed" in Step $4$ satisfies the entropy bound \eqref{eq.entropyminmax}.
In order to do that, we remark that the sequence $u_k$ verifies the bound $\norm{u_k}^{p_k}_{W^{1,p_k}(\M;\mathcal{\N})} \le E_{p_k,\delta_0}(u_{p_k,\delta_0}) \le \beta_{P,1}+2 <+\infty$, so it is uniformly bounded in $W^{1,p}(\M;\mathcal{\N})$; applying Rellich's and Kondrachov's theorems we may assume, after possibly extracting a subsequence, that $u_k$ converges weakly in $W^{1,p}(\M;\mathcal{\N})$, strongly in $\mathbb{L}^{p}(\M;\mathcal{\N})$, strongly in $W^{1,n}(\M;\mathcal{\N})$ to some $u_{p,\delta_0} \in W^{1,p}(\M;\mathcal{\N})$, and that their gradients converge almost everywhere to the gradient of $u_{p,\delta_0}$. We are gonna show that the convergence is strong in $W^{1,p}(\M;\mathcal{\N})$, and hence deduce that
 $u_{p,\delta_0}$ is a critical point at the right energy level. First of all,  since $C^\infty(\M;\N)$ is dense in $W^{1,p}(\M;\mathcal{\N})$ we can find a sequence $u^\ell\in C^\infty(\M;\N)$ such that $u^\ell \rightarrow u_{p,\delta_0}$ in $W^{1,p}(\M;\mathcal{\N})$. The Sublemma $(3.26)$ in Chapter $3$ of \cite{ura}, ensures that
\begin{equation*}
    \norm{(Id-P_{u_k})[u_k-u^\ell] }_{W^{1,p_k}} \longrightarrow 0,
\end{equation*}
where $P_u:W^{1,p}(\M;\R^N) \longrightarrow W^{1,p}(\M;u^*T\N)$ is at every point $x\in \M$ the orthogonal projection to the subspace $T_{u(x)}\N$.

After considering $E_{p_k,\delta_0}$ as a functional defined on $W^{1,p_k}(\M;\R^N)$, i.e. relaxing the $\N$-targeted assumption, we can combine this convergence with the one in Step $4$, to deduce
\begin{equation*}
\scal{d E_{p_k,\delta_0}(u_k)}{u_k-u^\ell } \underbrace{\longrightarrow}_{k,\ell \longrightarrow \infty} 0.
\end{equation*}
The convexity of the functional $E_{p_k,\delta_0}$ implies
\begin{align*}
&\int_M (1+(\delta_0+|\nabla u^\ell|^2)^{\frac{n}{2}})^{\frac{p_k}{n}} \ge \int_M ((1+(\delta_0+|\nabla u_k|^2)^{\frac{n}{2}})^{\frac{p_k}{n}} \\
&+p_k \int_M (1+(\delta_0+|\nabla u_k|^2)^{\frac{n}{2}})^{\frac{p_k-n}{n}} (\delta_0+|\nabla u_k|^2)^{\frac{n-2}{2}} \nabla u_k \nabla(u^\ell-u_k).
\end{align*}
Combining this inequality with the limit condition above we get
\begin{equation*}
\begin{aligned}
&\mathfrak{o}_k(1)= \scal{d E_{p_k,\delta_0}(u_k)}{u_k-u^\ell} \ge \int_M ((1+(\delta_0+|\nabla u_k|^2)^{\frac{n}{2}})^{\frac{p_k}{n}}-(1+(\delta_0+|\nabla u^\ell|^2)^{\frac{n}{2}})^{\frac{p_k}{n}} \\
&\ge \int_M ((1+(\delta_0+|\nabla u_k|^2)^{\frac{n}{2}})^{\frac{p}{n}}-(1+(\delta_0+|\nabla u^\ell|^2)^{\frac{n}{2}})^{\frac{p}{n}}-c \norm{\nabla u^\ell}_{\infty} (p_k-n).
\end{aligned}
\end{equation*}
We remark that the reason for the smooth approximation is exactly given by the last term. We can now let first $k \rightarrow \infty$ and use the weak lower-semicontinuity of $E_{p,\delta}$ granted by Morrey's theorem
\begin{equation*}
\begin{aligned}
    &\mathfrak{o}_{\ell}(1)=\int_M ((1+(\delta_0+|\nabla u_{p,\delta_0}|^2)^{\frac{n}{2}})^{\frac{p}{n}}-(1+(\delta_0+|\nabla u^\ell|^2)^{\frac{n}{2}})^{\frac{p}{n}}\\
    &\le \liminf_k \int_M ((1+(\delta_0+|\nabla u_k|^2)^{\frac{n}{2}})^{\frac{p}{n}}-(1+(\delta_0+|\nabla u^\ell|^2)^{\frac{n}{2}})^{\frac{p}{n}}\le 0,
    \end{aligned}
\end{equation*}
and then we let $\ell \rightarrow \infty$ to get $E_{p,\delta_0}(u_k) \rightarrow E_{p,\delta_0}(u_{p,\delta_0})$ uing once more the strong convergence of $u^\ell$ to $u_{p,\delta_0}$, which means that $\nabla u_k$ converges strongly to $\nabla u_{p,\delta_0}$ in $W^{1,p}(\M;\mathcal{\N})$ by Lemma \ref{lemma.riesz} (recall that we have a.e. convergence of the gradients).
Passing to the limit \eqref{eq.minmax1}, we have also obtain $E_{p,\delta_0}(u_{p,\delta_0})=\beta_{p,\delta_0}$, whereas with the help of \eqref{eq.minmax2}, Step 4, and the strong convergence just proven, we deduce that $u_{p,\delta_0}$ is a critical point of $E_{p,\delta_0}$: for any $\phi \in W^{1,p}(u_{p,\delta_0}^*T\mathcal{\N})$
\begin{align*}
  &|\scal{d E_{p,\delta_0}(u_{p,\delta_0})}{\phi}|\le |\scal{d E_{p,\delta_0}(u_{p_k,\delta_0})}{\phi}-\scal{d E_{p,\delta_0}(u_{p,\delta_0})}{\phi}|\\
  &+|\scal{d E_{p_k,\delta_0}(u_{p_k,\delta_0})}{\phi}-\scal{d E_{p,\delta_0}(u_{p_k,\delta_0})}{\phi}|+|\scal{d E_{p_k,\delta_0}(u_{p_k,\delta_0})}{\phi}| \longrightarrow 0.
\end{align*}
Concluding, we use once again Morrey's theorem to ensure that $\partial_p E_p$ is (weakly) lower semi-continuous since its integrand $(1+(\delta_0+|s|^2)^{\frac{n}{2}})^{\frac{p}{n}}\log(1+(\delta_0+|s|^2)^{\frac{n}{2}})$ is convex, so using Step 1 together with equation \eqref{eq.minmax0} we get for every $\eta>0$
\begin{align*}
&\tfrac{p}{n}\int_M (1+(\delta_0+|\nabla u_{p,\delta_0}|^2)^{\frac{n}{2}})^{\frac{p}{n}}\log(1+(\delta_0+|\nabla u_{p,\delta_0}|^2)^{\frac{n}{2}})=\partial_p E_{p,\delta_0} (u_{p,\delta_0}) \le \liminf_{k \rightarrow \infty} \partial_p E_{p,\delta_0} (u_k)\\
&\le \tfrac{d \beta_{p,\delta_0}}{dp} +3 \le \tfrac{\eta}{(p-n)\log(\tfrac{1}{p-n})}+3.  
\end{align*}
Multiplying both sides by $p-n$ and letting $p \rightarrow n$ we deduce the entropy condition \eqref{eq.entropyminmax}, concluding the proof. \qedhere
\end{itemize}
\end{proof}
\begin{proof}[Proof of Theorem \ref{th.minmax}]
By monotonicity of the energies $E_{p,\delta}$ in the parameters, for every $p>n$ and $\delta \in [0,1]$ we have $\beta<\beta_{p,\delta}$. For every $\eta>0$, by definition of $\beta$ and by approximation, there exists $h \in H\cap C^0(A;C^{\infty}(\M;\mathcal{N}))$ such that
\begin{equation*}
\sup_{t \in A} E_n(h(\cdot,t)) \le \beta+\eta.
\end{equation*}
For $p$ close enough to $n$ and $\delta$ close enough to $0$, we get through \eqref{eq.elementary3}
\begin{equation*}
\sup_{t \in A} E_{p,\delta}(h(\cdot,t)) \le \beta+\eta+C(n,P_0,\norm{\nabla h}_{\infty})(p-n+\delta) \le \beta+2\eta,
\end{equation*}
from which we easily deduce
\begin{equation*}
\lim_{p \rightarrow n, \delta \rightarrow 0} \beta_{p,\delta}= \beta.
\end{equation*}
In order to conclude, we combine classical min-max principles with our Theorem \ref{th.quantization} and Lemma \ref{lemma.minmax}.
\end{proof}

\section{Applications}\label{sec.applications}
\subsection{Main existence results}
We will now show our existence Theorem \ref{th.existence} for $n$-harmonic spheres, generalising Theorem $5.7$ in \cite{sac},
appealing to our quantization Theorem \ref{th.quantization}.
\begin{proof}[Proof of Theorem \ref{th.existence}, point (ii)]
Using Theorem \ref{th.SU2.8} we get for all $p>3$ and $\delta>0$ the existence of a non-trivial map $u_{p,\delta}:S^3 \rightarrow \N$, critical point of $E_{p,\delta}$. Notice that we must have $E_{p,\delta}(u_{p,\delta}) \ge \e_0>0$ by Corollary \ref{cor.trivial}. For any choice of sequences $p_k \searrow 3$ and $\delta_k \searrow 0$, arguing as in Section \ref{sec.quantization} through the use of Theorem \ref{th.SUregularity}, point $(c)$, we deduce the convergence (up to subsequences) of $u_{p_k,\delta_k}$ to a bubble tree with base $u_3:\M=S^3 \rightarrow \N$ and bubbles $\omega_{i,j}:S^3 \rightarrow \N$. We distinguish two cases. In the first and easier case, we suppose that no bubbles are forming along the limiting procedure. Thus the convergence of $u_{p_k,\delta_k}$ to $u_3$ is in $C^{1,\alpha_0}$, and the energy lower bound passes to the limit giving $\e_0 \le \lim_{p \searrow 3} E_{p_k,\delta_k}(u_{p_k,\delta_k})=D_3(u_3)$; this shows that $u_3$ is a non-trivial map as claimed.
In the other case, there exists at least a bubble forming $\omega_1$, which is non-trivial by definition, having energy $D_3(\omega_1) \ge \e_0 >0$, concluding the proof.
\end{proof}
Adapting the same proof to the case in which the target is homogeneous, relying on Theorem \ref{th.SUregularity} point $(b)$, we get Theorem \ref{th.existence} point $(i)$.
We conclude by remarking that we can recover the existence result of minimizing $n$-harmonic maps already obtained by \cite{duz1} and \cite{wei0}, this time relying on Theorem \ref{th.SUregularity} point $(a)$. The idea of the proofs are completely analogous to the $2$-dimensional case treated in \cite{sac} so we just sketch them.
\begin{theorem}[No bubbles case]\label{th.nobubbles}
If the target satisfies $\pi_n(N)=0$ for some $n \ge 2$, then for any domain manifold $(M^n,g)$, any metric $h$ on $N$, there exists a minimizing $n$-harmonic map in any homotopy class $\gamma \in C^0(M;N)$.
\end{theorem}
\begin{proof}
Consider for any $p>n$ and $\delta \in (0,1)$ minimizers $u_{p,\delta}$ of $E_{p,\delta}$ in a given homotopy class $\gamma=:[u_0] \in C^0(M;N)$. By equation \eqref{eq.boundmin}, they have uniformly bounded energies, hence by Theorem \ref{th.quantization}, we can extract a sequence $u_k:=u_{p_k,\delta_k}$ converging to a "bubble tree" $u_\infty \cup \set{\omega^{i,j}}$. In order to exclude the possibility of bubbling, we modify the maps $u_k$ near every possible bubbling point by gluing them to the map $u_\infty$. By the topological assumption $\pi_n(N)=0$, the newly constructed maps $\tilde{u_k}$ are still in the homotopy class $[u_0]$. Moreover, this procedure can be done with an energy error as small as we wish, so that by the minimizing property of $u_k$, we deduce that also the $u_k$ have small energy around the bubbling point, which means that the bubbles are not forming. In particular, the $u_k$ must be "regularly" converging to a minimizer as required.
\end{proof}
\begin{lemma}[Decomposition - Lemma 5.4 in \cite{sac}]\label{lemma.decomposition}
For an arbitrary homotopy class $\gamma \in \pi_n(N)$ the following duality holds: either $\gamma$ contains a minimizing $n$-harmonic map or for any $\eta>0$ there exist two non-trivial homotopy classes $0 \neq \gamma_1,\gamma_2 \in \pi_n(N)$ such that $\gamma \subset \gamma_1 + \gamma_2$ and $\inf_{\gamma_1} D_n+ \inf_{\gamma_2} D_n < \inf_{\gamma} D_n + \eta$.
\end{lemma}
\begin{proof}
Consider maps $u_{p,\delta}$ minimizing $E_{p,\delta}$ in the class $\gamma$, and apply Theorem \ref{th.quantization} after extracting a subsequence to get convergence to a bubble tree $u_{\infty} \cup \set{\omega^{i,j}}$. If no bubbles are forming, then $[u_\infty] \in \gamma$ is a $D_n$-minimizer. Otherwise there exists at least a bubbling point $x_1$. We consider around $x_1$ the same glued maps $\tilde{u}_k$ constructed above, and set
\begin{equation*}
u^1_k(x):=
\begin{cases}
\tilde{u}_k(x) \ \ \text{on } B_\rho(x_1)\\
u_k(x) \ \ \text{on } S^n \setminus B_\rho(x_1).
\end{cases}
\end{equation*}
Notice that this map may not belong to the homotopy class $\gamma$, however if we set
\begin{equation*}
u^2_k(x):=
\begin{cases}
u_k(x) \ \ \text{on } B_\rho(x_1)\\
\tilde{u}_k \circ \iota (x) \ \ \text{on } S^n \setminus B_\rho(x_1),
\end{cases}
\end{equation*}
where $\iota$ is the conformal reflection leaving the boundary of $B_\rho(x_1)$ fixed, then we have $[u_k] \in [u^1_k] + [u^2_k]$. The inequality on the $\inf D_n$ follows by possibly reducing the radius $\rho$ as done in Theorem \ref{th.nobubbles} above. The non-triviality of $[u^2_k]$ can be seen easily as $E_{p_k,\delta_k}(u^2_k) \ge E_{p_k,\delta_k}(u_k;B_\rho(x_1)) \ge \e_0>0$ (the maps are forming a bubble at $x_1$). On the other hand, we have
\begin{equation*}
E_{p_k,\delta_k}(u^1_k) \ge E_{p_k,\delta_k}(u_k;S^n\setminus B_\rho(x_1)) \ge D_n(u_\infty;S^n) - \mathfrak{o}(\rho).
\end{equation*}
We now have several cases. If $u_\infty$ is not trivial this can be bounded from below by $\tfrac{\e_0}{2}$ for $\rho$ small enough. If $u_\infty$ is a constant map, then either there is another bubbling point $x_2$ so $[u^1_k]$ is not trivial, or else the convergence of $u^1_k$ to this constant is $C^{1,\alpha}$-regular. In particular, $[u^1_k]$ is trivial and $[u_k^2]=\gamma$; in this case, the rescaled maps $u_k^2(r_k x)=u_k(r_k x)$ $C^{1,\alpha}$-converge around $x_1$ to the non-trivial bubble $\omega^1$ which is therefore a $D_n$-minimizer in $\gamma$.
\end{proof}
We explicitly remark that the third case treated above can appear when considering the classical example of bubbling sequence of $\alpha$-harmonic maps from $S^2$ into itself, formed via the action of the M\"obius group.
\begin{theorem}[Generating set]\label{th.generating}
There exists a subset of homotopy classes $\Lambda \subset \pi_n(N)$ generating the full group $\pi_n(N)$, and such that each homotopy class $\lambda \in \Lambda$ contains a minimizing $n$-harmonic map.
\end{theorem}
\begin{proof}
Let $\Lambda\subset \pi_n(N)$ be the subset of homotopy classes containing minimizers of $D_n$. Denote the subgroup it generates by $P:=<\Lambda>$ and suppose that $P$ is a proper subgroup of $\pi_n(N)$. Then there must exist a class $\gamma \in \pi_n(N) \setminus P$. This cannot contain minimizing $n$-harmonic maps, so there must exist non-trivial homotopy classes $\gamma_1$ and $\gamma_2$ such that $\gamma \subset \gamma_1 + \gamma_2$ and such that $\inf_{\gamma_1} D_n+ \inf_{\gamma_2} D_n < \inf_{\gamma} D_n + \tfrac{\e_0}{2}$ by the Lemma above (choosing $\eta=\tfrac{\e_0}{2}$). It cannot be that both the homotopy classes $\gamma_1$ and $\gamma_2$ belong to $P$, otherwise also $\gamma$ would by the subgroup property of $P$. Similarly, at least one of them cannot contain minimizing $n$-harmonic maps so we can iterate this procedure through Lemma \ref{lemma.decomposition} above. At any step, the homotopy classes involved are not trivial, so they must satisfy (for example at the first step) $\inf_{\gamma_i} D_n>\e_0$ by Corollary \ref{cor.trivial} and hence $\inf_{\gamma_i} D_n < \inf_{\gamma} D_n - \tfrac{\e_0}{2}$. After finitely many steps we obtain a class $\gamma_K$ which must be non-trivial by Lemma \ref{lemma.decomposition} but also satisfy $\inf_{\gamma_K} D_n < \e_0$, a contradiction to Corollary \ref{cor.trivial}.
\end{proof}
Finally, we remark that the statement of Theorem \ref{th.existence} is naturally graded, leading to further existence results.
\begin{corollary}\label{cor.graded}
For a fixed dimension $n\ge 2$, suppose $(\N,h)$ is a closed Riemannian manifold such that Theorem \ref{th.existence} applies. Then for any $2 \le i <n$ there exists also a non-trivial, $C^{1,\alpha}$-regular, $i$-harmonic map $u:(S^i,g_{round}) \rightarrow (\N,h)$.
\end{corollary}
\begin{proof}
It suffices to notice that for any $2\le i<n$ we have $\pi_{i+(n-i+k)}(\N)=\pi_{n+k}(\N)$ which is assumed to be a non-trivial group for some $k\in \mathbb{N}$ in the hypotheses of Theorem \ref{th.existence}. Therefore we can apply Theorem \ref{th.existence} with respect to the dimension $i$.
\end{proof}
Even though its proof is quite easy, this statement is not trivial, since we cannot compose the $n$-harmonic map of Theorem \ref{th.existence} with the totally geodesic embedding of $S^i \hookrightarrow S^n$ to get a $i$-harmonic map since the latter is not a $n$-harmonic morphism.

\subsection{Null-homotopic examples - three dimensional case}
In this subsection, we give a list of explicit targets verifying $\pi_3(\N)=0$ but $\pi_{3+k}(\N)\neq 0$ for some $k\in \mathbb{N}$, and a list of topological operations under which this class is stable, proving Theorem \ref{th.infinite} in the three-dimensional case. We will always assume all the manifolds to be connected. Let us define the class of valid targets (in arbitrary dimension $n \ge 2$)
\begin{equation*}
\mathcal{C}_n:= \set{\N \mid \pi_n(\N)=0,\  \pi_{n+k}(\N)\neq 0 \text{ for some } k\in \mathbb{N}}.
\end{equation*}

For $n=3$, the targets may be endowed with any arbitrary metrics $h$, so we just focus on the topological properties of them.

The following targets belong certainly to the class $\mathcal{C}_n$: $S^{\ell}$ with $\ell > n$; $\mathbb{P}^\ell(\R)$ for $\ell > n$; $\mathbb{P}^\ell(\mathbb{C})$ for $2\ell+1 >n$.
In the following we are going to prove the stability of the class $\mathcal{C}_n$ under some topological operations; we are going to prove them with respect to a generic dimension $n \ge 2$ in the hope to remove the homogeneity of the Riemannian metrics $h$ assumed in point $(i)$ of Theorem \ref{th.existence} in the future. 
\begin{lemma}[Stability under aspherical products]\label{lemma.aspherical}
If $\N \in \mathcal{C}_n$ and $\N'$ is an aspherical closed manifold, that is $\pi_i(\N')=0$ for all $i \ge 2$, the closed manifold $\N \times \N'$ belongs to $\mathcal{C}_n$ as well.
\end{lemma}
\begin{proof}
Indeed from Proposition 4.2 in \cite{hat} we have for all $i$
\begin{equation*}
\pi_i(\N \times \N')=\pi_i(\N) \times \pi_i(\N').
\end{equation*}
By our assumption on $\N'$, we deduce that for all $i \ge 2$
\begin{equation*}
\pi_i(\N \times \N') =\pi_i(\N),
\end{equation*}
hence we conclude.
\end{proof}
The same proof yields the following result.
\begin{lemma}[Stability under product]\label{lemma.product}
For any $\N_1,\N_2 \in \mathcal{C}_n$, then also the closed manifold $\N_1 \times \N_2 \in \mathcal{C}_n$.
\end{lemma}
Under some strong hypothesis, we can also connect two such targets.
\begin{lemma}[Stability under connected-sum]\label{lemma.connectedsum}
Suppose two elements $\N_1^\ell,\N_2^\ell \in \mathcal{C}_n$ are both $n$-connected, that is $\pi_i(\N_1)=\pi_i(\N_2)=0$ for all $i=0,...,n$, and have the same dimension $\ell \ge n+2$. Then also $\N_1\# \N_2 \in \mathcal{C}_n$.
\end{lemma}
\begin{proof}
We are going to show that $\N_1\# \N_2$ is also $n$-connected, after that we can conclude since no closed manifolds (CW-complexes) can be have all the homotopy groups trivial (combine Theorems $3.26$, $4.5$ and $4.32$ in \cite{hat}). There is cofibration
\begin{equation*}
S^{\ell-1} \rightarrow (\N_1 \setminus \set{y_1}) \vee (\N_2 \setminus \set{y_2}) \rightarrow \N_1\# \N_2,
\end{equation*}
for points $y_1 \in \N_1$ and $y_2 \in \N_2$. Blakers–Massey excision theorem (Theorem $4.23$ in \cite{hat}) implies the following exact sequence to hold
\begin{equation*}
\pi_i(S^{\ell-1}) \rightarrow \pi_i \Big( (\N_1 \setminus \set{y_1}) \vee (\N_2 \setminus \set{y_2}) \Big) \rightarrow \pi_i (\N_1\# \N_2) \rightarrow \pi_{i-1}(S^{\ell-1}) \rightarrow ...
\end{equation*}
as long as $i \le \ell -2 + n$. Moreover, the map $(\N_1 \setminus \set{y_1}) \vee (\N_2 \setminus \set{y_2}) \rightarrow \N_1\# \N_2$ is $(2 n +1)$-connected, so if $i \le 2 n$ we get
\begin{equation*}
\pi_i \Big( (\N_1 \setminus \set{y_1}) \vee (\N_2 \setminus \set{y_2}) \Big) = \pi_i(\N_1) \oplus \pi_i(\N_2).
\end{equation*}
Therefore, for any $i \le 2n$ and $n \le \ell -2$ we deduce
\begin{equation*}
\pi_i(S^{\ell-1}) \rightarrow \pi_i(\N_1) \oplus \pi_i(\N_2) \rightarrow \pi_i (\N_1\# \N_2) \rightarrow \pi_{i-1}(S^{\ell-1}) \rightarrow ...
\end{equation*}
Finally, for any $i \le n$ we know $\pi_i(S^{\ell-1})=0$, so we get the isomorphisms
\begin{equation*}
\pi_i (\N_1\# \N_2) = \pi_i(\N_1) \oplus \pi_i(\N_2)=0,
\end{equation*}
as required.
\end{proof}
\begin{remark}
Notice that the stronger assumption is necessary, as for example the manifold $\N:=(S^1)^3 \times S^4$ belongs to $\mathcal{C}_2$ thanks to Lemma \ref{lemma.aspherical}, but $\pi_2(\N \# \N)=\mathbb{Z}^3 \neq 0$. However, in order to get new examples, we can apply Lemma \ref{lemma.connectedsum} for example to glue products of the form $S^{\ell_1} \times S^{\ell_2}$ for $\ell_1,\ell_2>n$.
\end{remark}
Finally, we consider the case of $S^r$-bundles.
\begin{lemma}[Sphere bundles] \label{lemma.spherebundle}
Suppose $\N \in \mathcal{C}_n$ is the basis of a sphere bundle $p:\mathcal{E} \rightarrow \N$, where the total space is the closed manifold $\mathcal{E}$ and the fibers are spheres $S^r$ for some $r>n$. If $\pi_{r+1}(\N) \neq \mathbb{Z}$, then $\mathcal{E}$ belongs to $\mathcal{C}_n$ as well.
\end{lemma}
\begin{proof}
The sphere-bundle $p$ is a Serre fibration, and therefore induces a long exact sequence of homotopy groups for all $i$ (Theorem $4.41$ in \cite{hat}):
\begin{equation*}
... \rightarrow \pi_i(S^r) \rightarrow \pi_i(\mathcal{E}) \rightarrow \pi_i (\N) \rightarrow \pi_{i-1}(S^r) \rightarrow \pi_{i-1}(\mathcal{E}) \rightarrow \pi_{i-1}(\N) \rightarrow \pi_{i-2}(S^r) \rightarrow... \rightarrow \pi_0(S^r)=0.
\end{equation*}
If we restrict our attention to the level $i=n$ we get, since $\pi_n(S^r)=0$, that
\begin{equation*}
...\rightarrow 0=\pi_n(S^r) \rightarrow \pi_n(\mathcal{E}) \rightarrow \pi_n(\N)=0 \rightarrow ...,
\end{equation*}
hence $\pi_n(\mathcal{E})=0$. Assume now by contradiction that $\pi_{n+k}(\mathcal{E})=0$ for all $k \in \mathbb{N}$. Then the exact sequence ensures that
\begin{equation}\label{eq.homotopysphere}
\pi_{n+k+1} (\N) = \pi_{n+k}(S^r), \quad \forall k \in \mathbb{N}.
\end{equation}
Choosing $k=r-n$ we get the desired contradiction.
\end{proof}
\begin{remark}
The hypothesis in Lemma \ref{lemma.spherebundle} can be weaken a lot of course, since we only need a condition violating equation \eqref{eq.homotopysphere}. However, Lemma \ref{lemma.spherebundle} has quite a wide range of application anyway. For example, we can consider the unitary tangent bundle $\mathcal{E}=UT \N$ of a product of spheres $\N=S^{\ell} \times S^{\ell} \in \mathcal{C}_n$, for $\ell>n$. Indeed, since the fiber is $S^{2\ell-1}$, we need $\pi_{2 \ell}(S^{\ell} \times S^{\ell}) \neq \mathbb{Z}$. However, $\pi_{2 \ell}(S^{\ell} \times S^{\ell})=\pi_{2\ell}(S^{\ell}) \oplus \pi_{2\ell}(S^{\ell})$ which is always different from $\mathbb{Z}$.
\end{remark}
The following result is to be combined  with Lemma \ref{lemma.connectedsum} to obtain even further examples.
\begin{lemma}[Sphere bundles on $n$-connected]\label{lemma.spherebundleconnected}
If $\N \in \mathcal{C}_n$ is $n$-connected, then the closed total space $\mathcal{E}$ of any $S^r$-bundle $p:\mathcal{E} \rightarrow$, with $r >n$, belongs to $\mathcal{C}_n$ and is $n$-connected as well.
\end{lemma}
\begin{proof}
Arguing as done above, from Serre fibration's long exact homotopy sequence we obtain for any $i \le n$
\begin{equation*}
...\rightarrow 0=\pi_i(S^r) \rightarrow \pi_i(\mathcal{E}) \rightarrow \pi_i(\N)=0 \rightarrow ...,
\end{equation*}
proving the $n$-connectedness. Moreover, as we already said in the proof of Lemma \ref{lemma.connectedsum} above, the closed manifold $\mathcal{E}$ must have a non-trivial homotopy group.
\end{proof}

\subsection{Null-homotopic examples - higher dimensional case}

In the following we are going to list some Riemannian homogeneous targets, to which we can apply Theorem \ref{th.existence}, concluding the proof of Theorem \ref{th.infinite}. In order to keep the list short, we restrict the attention to some stable homotopy groups, hence the list is far from being complete; for the lists of stable homotopy groups, we refer the reader to \cite{rav} and references therein.
\begin{itemize}
    \item $\N=S^\ell$ for $\ell> n$; $\mathbb{P}^\ell(\R)$ for $\ell >n$; $\mathbb{P}^\ell(\mathbb{C})$ for $2 \ell +1 > n$;
    \item stable homotopy examples for $\N=S^\ell$: for $n=4+\ell$ and $\ell \ge 6$ or $n=5+\ell$ and $\ell \ge 7$, since we have $\pi_{4+\ell}(S^\ell)=0$ for all $\ell \ge 6$, $\pi_{5+\ell}(S^\ell)=0$ for all $\ell \ge 7$, but also $\pi_{12}(S^6)=\pi_{13}(S^7)=\mathbb{Z}/(2\mathbb{Z}) \neq 0$ and then (in stable range) $\pi_{6+\ell}(S^\ell)=\mathbb{Z}/(2\mathbb{Z}) \neq 0$ for $\ell \ge 8$; for $n=12+\ell$ and $\ell \ge 14$, since we have $\pi_{12+\ell}(S^\ell)= 0$ for all $\ell \ge 14$ but also $\pi_{27}(S^{14})=\mathbb{Z} \times \mathbb{Z}/(3\mathbb{Z})$ and then (in stable range) $\pi_{13+\ell}(S^\ell)=\mathbb{Z}/(3\mathbb{Z}) \neq 0$ for $\ell \ge 15$;
    \item Classical Lie groups in the stable range: $\N=SO(\ell)$ for $\ell \ge n+2$ and $n \equiv 2,4,5,6$ modulo $8$, since for these parameters $\pi_n(SO(\ell))=0$ but $\pi_{7+8\mathbb{N}}(SO(\ell))=\mathbb{Z}\neq 0$; $\N=Sp(\ell)$ for $n \le 4 \ell+1$ and $n \equiv 0,1,2,6$ modulo $8$, since for these parameters $\pi_n(Sp(\ell))=0$ but $\pi_{7+8\mathbb{N}}(Sp(\ell))=\mathbb{Z}\neq 0$; $\N=SU(\ell)$ for $n < 2 \ell+1$ and $n$ even, since for these parameters $\pi_n(SU(\ell))=0$ but $\pi_{1+2\mathbb{N}}(SU(\ell))=\mathbb{Z}\neq 0$.
\end{itemize}
Notice that we can apply Lemmas \ref{lemma.product} to take products of these examples.

\appendix
\section{}

In this section we give proofs of some elementary inequalities listed in Section \ref{sec.preliminary}.

\begin{proof}[Proof of \eqref{eq.elementary2}]
Without loss of generality, we may assume that $|X|=1$. As in \cite{far}, we will restrict our attention to the $2$-dimensional space spanned by $X$ and $Y$. After a rotation, we can ensure $X=(1,0)$ and $Y=(r \cos(\theta),r \sin{\theta})$, where $r=|Y|$ and $\theta \in [0,2\pi]$ is the angle between $X$ and $Y$. It is convenient to set $g(r):=(s+(\delta+r^2)^{\frac{n}{2}})^{\frac{p-n}{2 n}}(\delta+r^2)^{\frac{n-2}{4}}$, so that we get
\begin{align*}
   &[(s+(\delta+|X|^2)^{\frac{n}{2}})^{\frac{p-n}{n}} (\delta+|X|^2)^{\frac{n-2}{2}} X - (s+(\delta+|Y|^2)^{\frac{n}{2}})^{\frac{p-n}{n}}(\delta+|Y|^2)^{\frac{n-2}{2}} Y]\cdot (X-Y) \\
   &=[(s+(\delta+1)^{\frac{n}{2}})^{\frac{p-n}{n}} (\delta+1)^{\frac{n-2}{2}} (1,0) - (s+(\delta+r^2)^{\frac{n}{2}})^{\frac{p-n}{n}}(\delta+r^2)^{\frac{n-2}{2}} (r\cos(\theta),r \sin(\theta))]\cdot \\
   &\cdot (1-r\cos(\theta),-r \sin(\theta))=g^2(1)(1-r \cos(\theta))-g^2(r) r \cos(\theta)+g^2(r) r^2.
\end{align*}
On the other hand, the right hand side is given by
\begin{align*}
    &c_0|V(X)-V(Y)|^2=c_0 |(s+(\delta+|X|^2)^{\frac{n}{2}})^{\frac{p-n}{2 n}}(\delta+|X|^2)^{\frac{n-2}{4}} X-(s+(\delta+|Y|^2)^{\frac{n}{2}})^{\frac{p-n}{2 n}}(\delta+|Y|^2)^{\frac{n-2}{4}} Y|^2\\
    &=c_0|g(1) (1,0))-g(r) (r\cos(\theta),r \sin(\theta))|^2=c_0[g^2(1)-2g(1)g(r)r \cos(\theta)+g^2(r) r^2].
\end{align*}
Plugging in \eqref{eq.elementary2}, we see that this inequality is equivalent after rearranging to
\begin{equation*}
(1-c_0) g^2(r) r^2 +(1-c_0) g^2(1) \ge [g^2(r)+g^2(1)-2g(1)g(r)c_0] r \cos(\theta).
\end{equation*}
Since $|\cos(\theta)|\le 1$, the inequality follows after we prove the stronger
\begin{equation*}
    (1-c_0) g^2(r) r^2 +(1-c_0) g^2(1) \ge [g^2(r)+g^2(1)-2g(1)g(r)c_0] r .
\end{equation*}
We have used the positivity of the term in square brackets valid through Young's inequality, assuming $c_0 \le 1$. Rewrite this inequality as
\begin{equation*}
 (g^2(r) r-g^2(1))(r-1) \ge c_0 (g(r) r-g(1))^2.
\end{equation*}
We now want to show that the constant $c_0$ can be chosen independently of $p$ and $\delta$. At $r=0$ and $r=1$ this is true with constant $c_0=1$, with equality sign.
At the limit $r=+\infty$ we can choose $c_0=1$. We are going to argue in several steps. First, we compute the derivative in $\delta$ of the fraction in consideration and claim that it is greater or equal than zero. In order to shorten the formula we set $b(r):=(\delta+r^2)$ and  $a(r):=(s+(\delta+r^2)^{\frac{n}{2}})=(s+b(r)^{\frac{n}{2}})$
\begin{align*}
    &\partial_{\delta} \tfrac{(g^2(r)r-g^2(1))(r-1)}{(g(r)r-g(1))^2 }=\partial_{\delta} \tfrac{(a(r)^{\frac{p- n}{ n}} b(r)^{\frac{n-2}{2}}r-a(1)^{\frac{p- n}{n}} b(1)^{\frac{n-2}{2}})(r-1)}{(a(r)^{\frac{p- n}{2 n}} b(r)^{\frac{n-2}{4}}r-a(1)^{\frac{p- n}{2 n}} b(1)^{\frac{n-2}{4}})^2 }\\
    &=\tfrac{r-1}{(g(r)r-g(1))^3} \Big[ \Big( a(r)^{\frac{p- n}{2 n}} b(r)^{\frac{n-2}{4}}r-a(1)^{\frac{p- n}{2 n}} b(1)^{\frac{n-2}{4}} \Big)\cdot \Big( \tfrac{p-n}{n} a(r)^{\frac{p- 2 n}{ n}} \tfrac{n}{2} b(r)^{\frac{n-2}{2}} \cdot 1 \cdot b(r)^{\frac{n-2}{2}} r \\
    &-\tfrac{p-n}{n} a(1)^{\frac{p- 2 n}{ n}} \tfrac{n}{2} b(1)^{\frac{n-2}{2}} \cdot 1 \cdot b(1)^{\frac{n-2}{2}}+ a(r)^{\frac{p- n}{ n}} \tfrac{n-2}{2} b(r)^{\frac{n-4}{2}} \cdot 1 \cdot r-a(1)^{\frac{p- n}{ n}} \tfrac{n-2}{2} b(1)^{\frac{n-4}{2}}\Big) \\
    &-2 \Big( a(r)^{\frac{p- n}{n}} b(r)^{\frac{n-2}{2}}r-a(1)^{\frac{p- n}{n}} b(1)^{\frac{n-2}{2}} \Big) \cdot \Big( \tfrac{p-n}{2 n} a(r)^{\frac{p- 3 n}{2 n}} \tfrac{n}{2} b(r)^{\frac{n-2}{2}} \cdot 1 \cdot b(r)^{\frac{n-2}{4}} r\\
    &-\tfrac{p-n}{2 n} a(1)^{\frac{p- 3 n}{2 n}} \tfrac{n}{2} b(1)^{\frac{n-2}{2}} \cdot 1 \cdot b(1)^{\frac{n-2}{4}}+ a(r)^{\frac{p- n}{2 n}} \tfrac{n-2}{4} b(r)^{\frac{n-6}{4}} \cdot 1 \cdot r-a(1)^{\frac{p- n}{2 n}} \tfrac{n-2}{4} b(1)^{\frac{n-6}{4}} \Big) \Big]\\
    &= \tfrac{r-1}{2(g(r)r-g(1))^3} \Big[ (p-n) a(r)^{\frac{3 p- 5 n}{2 n}} b(r)^{\frac{5(n-2)}{4}} r^2 -(p-n) a(r)^{\frac{p- 2 n}{n}} b(r)^{n-2} r a(1)^{\frac{p-n}{2 n}} b(1)^{\frac{n-2}{4}} \\
    &-(p-n) a(r)^{\frac{p- n}{2 n}} b(r)^{\frac{n-2}{4}} r a(1)^{\frac{p-2 n}{n}} b(1)^{n-2}+(p-n) a(1)^{\frac{3 p- 5 n}{2 n}} b(1)^{\frac{5(n-2)}{4}}+(n-2)a(r)^{\frac{3(p- n)}{2 n}} b(r)^{\frac{3 n-10}{4}}r^2\\
    &-(n-2) a(r)^{\frac{p- n}{n}} b(r)^{\frac{n-4}{2}} r a(1)^{\frac{p-n}{2 n}} b(1)^{\frac{n-2}{4}}-(n-2) a(r)^{\frac{p- n}{2 n}} b(r)^{\frac{n-2}{4}} r a(1)^{\frac{p-n}{n}} b(1)^{\frac{n-4}{2}}\\
    &+(n-2)a(1)^{\frac{3(p- n)}{2 n}} b(1)^{\frac{3 n-10}{4}}-(p-n)a(r)^{\frac{3 p- 5 n}{2 n}} b(r)^{\frac{5(n-2)}{4}} r^2+(p-n)a(r)^{\frac{p- 3 n}{2 n}} b(r)^{\frac{3(n-2)}{4}} r a(1)^{\frac{p-n}{n}} b(1)^{\frac{n-2}{2}}\\
    &+(p-n)a(r)^{\frac{p-n}{n}} b(r)^{\frac{n-2}{2}} r a(1)^{\frac{p-3 n}{2 n}} b(1)^{\frac{3(n-2)}{4}}-(p-n)a(1)^{\frac{3 p-5 n}{2 n}} b(1)^{\frac{5(n-2)}{4}}-(n-2)a(r)^{\frac{3(p-n)}{2 n}} b(r)^{\frac{3 n-10}{4}} r^2\\
    &\text{\footnotesize$+(n-2)a(r)^{\frac{p-n}{2 n}} b(r)^{\frac{n-6}{4}} r a(1)^{\frac{p-n}{n}} b(1)^{\frac{n-2}{2}}+(n-2)a(r)^{\frac{p-n}{n}} b(r)^{\frac{n-2}{2}} r a(1)^{\frac{p-n}{2 n}} b(1)^{\frac{n-6}{4}}-(n-2)a(1)^{\frac{3(p-n)}{2 n}} b(1)^{\frac{3 n-10}{4}}\Big]$}\\
    &=\tfrac{r(r-1)}{2(g(r)r-g(1))^3} \Big(a(r)^{\frac{p- n}{2 n}} b(r)^{\frac{n-2}{4}}-a(1)^{\frac{p- n}{2 n}} b(1)^{\frac{n-2}{4}} \Big) \cdot \Big[(p-n)a(r)^{\frac{p-n}{2 n}} b(r)^{\frac{n-2}{4}} a(1)^{\frac{p-3 n}{2 n}} b(1)^{\frac{3(n-2)}{4}}\\
    &+(n-2)a(r)^{\frac{p-n}{2 n}} b(r)^{\frac{n-2}{4}} a(1)^{\frac{p-n}{2 n}} b(1)^{\frac{n-6}{4}}-(p-n)a(r)^{\frac{p-3 n}{2 n}} b(r)^{\frac{3(n-2)}{4}} a(1)^{\frac{p-n}{2 n}} b(1)^{\frac{n-2}{4}}\\
    &-(n-2)a(r)^{\frac{p-n}{2 n}} b(r)^{\frac{n-6}{4}} a(1)^{\frac{p-n}{2 n}} b(1)^{\frac{n-2}{4}}\Big]=\tfrac{r(r-1)}{2(g(r)r-g(1))^3} \Big(a(r)^{\frac{p- n}{2 n}} b(r)^{\frac{n-2}{4}}-a(1)^{\frac{p- n}{2 n}} b(1)^{\frac{n-2}{4}} \Big) \cdot \\
    &\text{\footnotesize$a(r)^{\frac{p-3 n}{2 n}} b(r)^{\frac{n-6}{4}} a(1)^{\frac{p-3 n}{2 n}} b(1)^{\frac{n-6}{4}}\cdot \Big[ (p-n) \Big( a(r)b(r) b(1)^{\frac{n}{2}}-a(1)b(1) b(r)^{\frac{n}{2}}\Big) +(n-2) \Big( a(r)b(r) a(1)-a(r)a(1) b(1)\Big) \Big]$}.
\end{align*}
In order to study the sign of the term in square brackets we plug in the definition of $a$:
\begin{align*}
&\Big[ ...\Big]=(p-n) \Big( (s+b(r)^{\frac{n}{2}}) b(r) b(1)^{\frac{n}{2}}-(s+b(1)^{\frac{n}{2}}) b(1) b(r)^{\frac{n}{2}}\Big) +(n-2) \Big( (s+b(r)^{\frac{n}{2}}) b(r) (s+b(1)^{\frac{n}{2}})\\
&-(s+b(r)^{\frac{n}{2}}) (s+b(1)^{\frac{n}{2}}) b(1)\Big)= (p-n) s b(r) b(1)^{\frac{n}{2}}+(p-n) b(r)^{\frac{n+2}{2}} b(1)^{\frac{n}{2}}-(p-n) s b(r)^{\frac{n}{2}} b(1)\\
&-(p-n) b(r)^{\frac{n}{2}} b(1)^{\frac{n+2}{2}}+(n-2) s^2 b(r)+(n-2) s b(r)^{\frac{n+2}{2}} +(n-2) s b(r) b(1)^{\frac{n}{2}}+(n-2) b(r)^{\frac{n+2}{2}} b(1)^{\frac{n}{2}}\\
&-(n-2) s^2 b(1)-(n-2) s b(1) b(r)^{\frac{n}{2}}-(n-2) s b(1)^{\frac{n+2}{2}}-(n-2)b(r)^{\frac{n}{2}}b(1)^{\frac{n+2}{2}}= (p-2) s b(r) b(1)^{\frac{n}{2}}\\
&+(p-2) b(r)^{\frac{n+2}{2}} b(1)^{\frac{n}{2}}-(p-2) s b(r)^{\frac{n}{2}} b(1)-(p-2) b(r)^{\frac{n}{2}} b(1)^{\frac{n+2}{2}}+(n-2) s^2 b(r)-(n-2) s^2 b(1)\\
&+(n-2) s b(r)^{\frac{n+2}{2}} -(n-2) s b(1)^{\frac{n+2}{2}} \\
&=\big( b(r)^{\frac{n+2}{2}}+s b(r)\big)\big( (p-2)b(1)^{\frac{n}{2}} +(n-2)s \big)-\big( b(1)^{\frac{n+2}{2}}+s b(1)\big)\big( (p-2)b(r)^{\frac{n}{2}} +(n-2)s \big).
\end{align*}
We now distinguish two cases. If $r \ge 1$, we can verify it to be greater or equal than zero, as
\begin{equation*}
\tfrac{b(r)^{\frac{n+2}{2}}+s b(r)}{b(1)^{\frac{n+2}{2}}+s b(1)} \ge \tfrac{(p-2)b(r)^{\frac{n}{2}} +(n-2)s}{(p-2)b(1)^{\frac{n}{2}} +(n-2)s},
\end{equation*}
is implied by the stronger (here we use that $P_0 < n+1$):
\begin{align*}
&\tfrac{(p-2)b(r)^{\frac{n}{2}} +(n-2)s}{(p-2)b(1)^{\frac{n}{2}} +(n-2)s}=\tfrac{(p-2)(b(r)^{\frac{n}{2}} -b(1)^{\frac{n}{2}})}{(p-2)b(1)^{\frac{n}{2}} +(n-2)s}+1=1+\tfrac{b(r)^{\frac{n}{2}} -b(1)^{\frac{n}{2}}}{b(1)^{\frac{n}{2}} +\frac{n-2}{p-2}s}\le 1+\tfrac{b(r)^{\frac{n}{2}} -b(1)^{\frac{n}{2}}}{b(1)^{\frac{n}{2}}+\frac{n-2}{n-1}s}=1+\tfrac{(n-1)(b(r)^{\frac{n}{2}} -b(1)^{\frac{n}{2}})}{(n-1)b(1)^{\frac{n}{2}} +(n-2)s}\\
&=\tfrac{(n-1) b(r)^{\frac{n}{2}} +(n-2)s}{(n-1)b(1)^{\frac{n}{2}} +(n-2)s} \boldsymbol{\le} \tfrac{b(r)^{\frac{n+2}{2}}+s b(r)}{b(1)^{\frac{n+2}{2}}+s b(1)},
\end{align*}
which is equivalent to
\begin{equation*}
\tfrac{b(r)^{\frac{n+2}{2}}+s b(r)}{(n-1) b(r)^{\frac{n}{2}} +(n-2)s} \boldsymbol{\ge} \tfrac{b(1)^{\frac{n+2}{2}}+s b(1)}{(n-1) b(1)^{\frac{n}{2}} +(n-2)s};
\end{equation*}
the latter, is clearly verified at $r=1$, and can therefore be deduce from
\begin{align*}
&\partial_r \Big( \tfrac{b(r)^{\frac{n+2}{2}}+s b(r)}{(n-1) b(r)^{\frac{n}{2}} +(n-2)s} \Big)= \tfrac{1}{((n-1) b(r)^{\frac{n}{2}} +(n-2)s)^2} \Big[ (\tfrac{n+2}{2} b(r)^{\frac{n}{2}} 2r+2rs)((n-1) b(r)^{\frac{n}{2}} +(n-2)s)\\
&-(b(r)^{\frac{n+2}{2}}+sb(r))\cdot (n-1)\frac{n}{2}b(r)^{\frac{n-2}{2}}2r \Big]=\tfrac{r}{((n-1) b(r)^{\frac{n}{2}} +(n-2)s)^2} \Big[ (n+2)(n-1)b(r)^n+(n^2-4)sb(r)^{\frac{n}{2}}\\
&+2(n-1)sb(r)^{\frac{n}{2}}+2(n-2)s^2-n(n-1)b(r)^n-n(n-1)sb(r)^{\frac{n}{2}} \Big]=\tfrac{r}{((n-1) b(r)^{\frac{n}{2}} +(n-2)s)^2} \Big[2b(r)^n\\
&+(n-2) \big(2b(r)^n+3sb(r)^{\frac{n}{2}}+2s^2 \big) \Big]\ge 0,
\end{align*}
where we have used that $2 y^2 +3 s y +2 s^2 \ge 0$ for any choice of $s \ge 0$ and any $y \ge 0$.

Regarding the second case $r \le 1$, we see that the opposite inequality
\begin{equation*}
\tfrac{b(r)^{\frac{n+2}{2}}+s b(r)}{b(1)^{\frac{n+2}{2}}+s b(1)} \le \tfrac{(p-2)b(r)^{\frac{n}{2}} +(n-2)s}{(p-2)b(1)^{\frac{n}{2}} +(n-2)s},
\end{equation*}
holds true, because of the validity of the stronger (recall that in this case $b(r) \le b(1)$)
\begin{align*}
&\tfrac{(p-2)b(r)^{\frac{n}{2}} +(n-2)s}{(p-2)b(1)^{\frac{n}{2}} +(n-2)s}=\tfrac{(p-2)(b(r)^{\frac{n}{2}} -b(1)^{\frac{n}{2}})}{(p-2)b(1)^{\frac{n}{2}} +(n-2)s}+1=1+\tfrac{b(r)^{\frac{n}{2}} -b(1)^{\frac{n}{2}}}{b(1)^{\frac{n}{2}} +\frac{n-2}{p-2}s}\ge 1+\tfrac{b(r)^{\frac{n}{2}} -b(1)^{\frac{n}{2}}}{b(1)^{\frac{n}{2}}+\frac{n-2}{n-1}s}=1+\tfrac{(n-1)(b(r)^{\frac{n}{2}} -b(1)^{\frac{n}{2}})}{(n-1)b(1)^{\frac{n}{2}} +(n-2)s}\\
&=\tfrac{(n-1) b(r)^{\frac{n}{2}} +(n-2)s}{(n-1)b(1)^{\frac{n}{2}} +(n-2)s} \boldsymbol{\ge} \tfrac{b(r)^{\frac{n+2}{2}}+s b(r)}{b(1)^{\frac{n+2}{2}}+s b(1)},
\end{align*}
which is proven exactly by the above inequality.

Resuming, we have shown that
\begin{equation*}
\Big[... \Big] \ge 0 \quad \text{if } r \ge 1, \ \text{and } \Big[... \Big] \le 0 \quad \text{if } r \le 1.
\end{equation*}
From this and the monotonicity properties of $a(r)$ and $b(r)$, we deduce that
\begin{equation*}
\partial_{\delta} \tfrac{(g^2(r)r-g^2(1))(r-1)}{(g(r)r-g(1))^2 } \ge 0.
\end{equation*}
Notice that for $\delta=0$, $b(r)=r^2$ and $a(r):=s+r^n$, hence
\begin{equation*}
\tfrac{(g^2(r)r-g^2(1))(r-1)}{(g(r)r-g(1))^2 } \ge \tfrac{(g^2(r)r-g^2(1))(r-1)}{(g(r)r-g(1))^2 } \mid_{\delta=0}= \tfrac{((s+r^n)^{\frac{p- n}{ n}} r^{n-1}-(s+1)^{\frac{p- n}{n}})(r-1)}{((s+r^n)^{\frac{p- n}{2 n}} r^{\frac{n}{2}}-(s+1)^{\frac{p- n}{2 n}})^2 }.
\end{equation*}
We claim that also right hand side is bounded away from zero uniformly as $p \searrow n$ and for any $s \ge 0$. In order to show that, we claim that its $s$-derivative is non-negative:
\begin{equation*}
\footnotesize
\begin{aligned}
&\partial_s \tfrac{((s+r^n)^{\frac{p- n}{ n}} r^{n-1}-(s+1)^{\frac{p- n}{n}})(r-1)}{((s+r^n)^{\frac{p- n}{2 n}} r^{\frac{n}{2}}-(s+1)^{\frac{p- n}{2 n}})^2 }=\tfrac{(r-1)}{((s+r^n)^{\frac{p- n}{2 n}} r^{\frac{n}{2}}-(s+1)^{\frac{p- n}{2 n}})^3 } \Big[ \big( \tfrac{p-n}{n} (s+r^n)^{\frac{p- 2n}{n}} r^{n-1}-\tfrac{p-n}{n} (s+1)^{\frac{p- 2n}{n}}\big) \cdot \\
&\cdot \big( (s+r^n)^{\frac{p- n}{2 n}} r^{\frac{n}{2}}-(s+1)^{\frac{p- n}{2 n}}\big)-2 \big( (s+r^n)^{\frac{p- n}{ n}} r^{n-1}-(s+1)^{\frac{p- n}{n}}\big) \big( \tfrac{p-n}{2 n} (s+r^n)^{\frac{p- 3n}{2 n}} r^{\frac{n}{2}}-\tfrac{p-n}{2 n} (s+1)^{\frac{p- 3n}{2 n}}\big)  \Big]\\
&=\tfrac{(r-1)}{((s+r^n)^{\frac{p- n}{2 n}} r^{\frac{n}{2}}-(s+1)^{\frac{p- n}{2 n}})^3 } \tfrac{p-n}{n} \Big[ (s+r^n)^{\frac{3p- 5n}{2 n}}r^{\frac{3n-2}{2}}-(s+r^n)^{\frac{p-n}{2 n}}r^{\frac{n}{2}} (s+1)^{\frac{p- 2n}{n}}-(s+r^n)^{\frac{p- 2n}{n}}r^{n-1}(s+1)^{\frac{p-n}{2 n}}\\
&+(s+1)^{\frac{3p-5n}{2 n}}-(s+r^n)^{\frac{3p-5n}{2 n}}r^{\frac{3n-2}{2}}+(s+r^n)^{\frac{p- 3n}{2 n}}r^{\frac{n}{2}}(s+1)^{\frac{p-n}{n}}+(s+r^n)^{\frac{p-n}{n}}r^{n-1}(s+1)^{\frac{p-3n}{2 n}} -(s+1)^{\frac{3p-5n}{2 n}}     \Big]\\
&=\tfrac{(r-1)}{((s+r^n)^{\frac{p- n}{2 n}} r^{\frac{n}{2}}-(s+1)^{\frac{p- n}{2 n}})^3 } \tfrac{p-n}{n} (s+r^n)^{\frac{p- 3n}{2 n}} (s+1)^{\frac{p- 3n}{2 n}} r^{\frac{n}{2}} \Big[ (s+1)^{\frac{p+n}{2 n}}+(s+r^n)^{\frac{p+n}{2 n}} r^{\frac{n-2}{2}}-(s+r^n)(s+1)^{\frac{p-n}{2 n}}\\
&-(s+r^n)^{\frac{p-n}{2 n}}(s+1)r^{\frac{n-2}{2}} \Big]=\tfrac{(r-1)}{((s+r^n)^{\frac{p- n}{2 n}} r^{\frac{n}{2}}-(s+1)^{\frac{p- n}{2 n}})^3 } \tfrac{p-n}{n} (s+r^n)^{\frac{p- 3n}{2 n}} (s+1)^{\frac{p- 3n}{2 n}} r^{\frac{n}{2}} [(s+r^n)^{\frac{p-n}{2 n}}r^{\frac{n-2}{2}}\\
&-(s+1)^{\frac{p-n}{2 n}}] [r^n-1] \ge 0,
\end{aligned}
\end{equation*}
as we claimed.

Arguing as above, we deduce that
\begin{equation*}
\tfrac{(g^2(r)r-g^2(1))(r-1)}{(g(r)r-g(1))^2 } \ge \tfrac{((s+r^n)^{\frac{p- n}{ n}} r^{n-1}-(s+1)^{\frac{p- n}{n}})(r-1)}{((s+r^n)^{\frac{p- n}{2 n}} r^{\frac{n}{2}}-(s+1)^{\frac{p- n}{2 n}})^2 } \mid_{s=0} \ge \tfrac{(r^{p-1}-1)(r-1)}{(r^{\frac{p}{2}}-1)^2}.
\end{equation*}
In order to conclude, we need the right-hand-side to be bounded away from zero uniformly as $p \searrow n$. Following the classical argument in \cite{gia1}, Lemma $2.1$ there, we get a lower bound for the numerator
\begin{align*}
&(r^{p-1}-1)(r-1)=\int_0^1 \partial_t [(t r+1-t)^{p-1}(r-1)] dt=(p-1) (r-1)^2 \int_0^1 (t r+1-t)^{p-2} dt \\
&\ge (p-1) (r-1)^2 \int_0^1 t^{p-2} r^{p-2}+(1-t)^{p-2} dt= (r-1)^2 (r^{p-2}+1)
\end{align*}
and an upper bound for the denominator
\begin{align*}
&(r^{\frac{p}{2}}-1)^2=\Big(\int_0^1 \partial_t [(t r+1-t)^{\frac{p}{2}}] dt \Big)^2=\tfrac{p^2}{4}(r-1)^2 \Big(\int_0^1 (t r+1-t)^{\frac{p-2}{2}} dt \Big)^2\\
&\le \tfrac{p^2}{4}(r-1)^2 2^{p-4} \Big(\int_0^1 t^{\frac{p-2}{2}} r^{\frac{p-2}{2}}+(1-t)^{\frac{p-2}{2}} dt \Big)^2 \le 2^{p-3} (r-1)^2 (r^{p-2}+1).
\end{align*}
This concludes the proof of the first inequality. The second inequality can be proven analogously.
\end{proof}

\end{document}